\documentclass[a4paper,10pt, reqno]{amsart}

\usepackage[latin1]{inputenc}
\usepackage{color}
\usepackage{a4wide}
\usepackage{geometry}
\usepackage{latexsym }
\usepackage{amsmath,amsthm, amssymb,amsfonts,amscd,wasysym}
\usepackage{graphicx,hyperref,enumerate,verbatim}
\usepackage{epsfig}
\usepackage{mathrsfs}

\definecolor{purple}{rgb}{0.65, 0, 1}
\definecolor{orange}{rgb}{1,.5,0}
\definecolor{purple}{rgb}{0.5, 0, 0.9}
\definecolor{green}{rgb}{0, 0.7, 0}
\definecolor{orange}{rgb}{1,.5,0}
\definecolor{gray}{rgb}{.6,.6,.6}

\newcommand{\corr}[1]{\textcolor{black}{#1}}

\providecommand{\U}[1]{\protect\rule{.1in}{.1in}}
\newcommand{\R}{{\mathbb R}}
\newcommand{\N}{{\mathbb N}}

\renewcommand{\t}{\tau}

\newcommand{\be}[1]{\begin{equation}\label{#1}}
\newcommand{\ee}{\end{equation}}

\renewcommand{\)}{\right)}

\let\pa\partial
\let\r\rho

\let\eps\varepsilon

\newcommand{\beq}{\begin{eqnarray}}
\newcommand{\eeq}{\end{eqnarray}}
\newcommand{\beqs}{\begin{eqnarray*}}
\newcommand{\eeqs}{\end{eqnarray*}}
\newcommand{\bequ}{\begin{equation}}
\newcommand{\eequ}{\end{equation}}

\def\r{\rho}

\newcommand{\cal}{\mathcal}

\def\e{\varepsilon}
\def\sgn{{\rm sgn}}

\newcommand{\W}{{\mathcal W}}
\newcommand{\WW}{{\mathcal K}}
\newcommand{\NN}{{\mathcal N}}
\newcommand{\E}{{\mathcal E}}
\newcommand{\D}{{\mathcal D}}
\newcommand{\supp}{\textnormal{supp}\,}
\let\pa\partial

\let\eps\varepsilon

\makeatother

\numberwithin{equation}{section}
\newtheorem{theorem}{Theorem}[section]
\newtheorem{remark}[theorem]{Remark}
\newtheorem{lemma}[theorem]{Lemma}
\newtheorem{definition}[theorem]{Definition}
\newtheorem{proposition}[theorem]{Proposition}
\newtheorem{corollary}[theorem]{Corollary}

\makeatletter
\renewcommand*\env@matrix[1][\arraystretch]{%
  \edef\arraystretch{#1}%
  \hskip -\arraycolsep
  \let\@ifnextchar\new@ifnextchar
  \array{*\c@MaxMatrixCols c}}
\makeatother

\def\supp{\mathrm{supp}\,}
\def\dist{\mathrm{dist}}
\def\sgn{\mathrm{sgn}\,}
\def\K{W}
\def\R{\mathbb{R}}

\begin{document}

\title[]{Nonlinear Aggregation-Diffusion Equations: Radial Symmetry and Long Time Asymptotics}
\author{J. A. Carrillo$^1$, S. Hittmeir$^2$, B. Volzone$^3$, Y. Yao$^4$}
\thanks{$^1$Department of Mathematics, Imperial College London, SW7 2AZ London, United Kingdom. Email: carrillo@imperial.ac.uk}
\thanks{$^2$Faculty of Mathematics, University of Vienna, Oskar-Morgenstern-Platz 1, A-1090 Vienna, Austria. E-mail: sabine.hittmeir@univie.ac.at}
\thanks{$^3$Dipartimento di Ingegneria, Universi\-t\`{a} degli Studi di Napoli ``Parthenope'', 80143 Italia. E-mail: bruno.volzone@uniparthenope.it}
\thanks{$^4$School of Mathematics, Georgia Institute of Technology, 686 Cherry Street, Atlanta, GA 30332-0160 USA. E-mail: yaoyao@math.gatech.edu}

\begin{abstract}
We analyze under which conditions equilibration between two competing effects, repulsion modeled by nonlinear diffusion and attraction modeled by nonlocal interaction, occurs. This balance leads to continuous compactly supported radially decreasing equilibrium configurations for all masses. All stationary states with suitable regularity are shown to be radially symmetric by means of continuous Steiner symmetrization techniques. Calculus of variations tools allow us to show the existence of global minimizers among these equilibria. Finally, in the particular case of Newtonian interaction in two dimensions they lead to uniqueness of equilibria for any given mass up to translation and to the convergence of solutions of the associated nonlinear aggregation-diffusion equations towards this unique equilibrium profile up to translations as $t\to\infty$.
\end{abstract}

\maketitle


\section{Introduction}
The evolution of interacting particles and their equilibrium configurations has attracted the attention of many applied mathematicians and mathematical analysts for years. Continuum description of interacting particle systems usually leads to analyze the behavior of a mass density $\rho(t,x)$ of individuals at certain location $x\in \R^d$ and time $t\geq 0$. Most of the derived models result in aggregation-diffusion nonlinear partial differential equations through different asymptotic or mean-field limits \cite{O,BV,CCH}. The different effects reflect that equilibria are obtained by competing behaviors: the repulsion between individuals/particles is modeled through nonlinear diffusion terms while their attraction is integrated via nonlocal forces. This attractive nonlocal interaction takes into account that the presence of particles/individuals at a certain location $y\in \R^d$ produces a force at particles/individuals located at $x\in \R^d$ proportional to $-\nabla W(x-y)$ where the given interaction potential $W:\R^d\to \R$ is assumed to be radially symmetric and increasing consistent with attractive forces. The evolution of the mass density of particles/individuals is given by the nonlinear aggregation-diffusion equation of the form:
\begin{equation}
\label{aggregation}
\partial_t \rho = \Delta \rho^m + \nabla \cdot (\rho\nabla(\K*\rho)) \quad x\in\R^d \,, t\geq 0\,
\end{equation}
with initial data $\rho_0 \in L^1_+(\R^d)\cap L^m(\R^d)$. We will work with degenerate diffusions, $m>1$, that appear naturally in modelling repulsion with very concentrated repelling nonlocal forces \cite{O,BV}, but also with linear and fast diffusion ranges $0<m\leq 1$, which are also classical in applications \cite{patlak,Keller-Segel-70}. These models are ubiquitous in mathematical biology where they have been used as macroscopic descriptions for collective behavior or swarming of animal species, see \cite{ML1,BCM1, ML2,MCO,TBC,BCM} for instance, or more classically in chemotaxis-type models, see \cite{patlak,Keller-Segel-70,JL,horst,BDP,BCL,CC2} and the references therein.

On the other hand, this family of PDEs is a particular example of nonlinear gradient flows in the sense of optimal transport between mass densities, see \cite{AGS08,CMV03,Carrillo-McCann-Villani06}. The main implication for us is that there is a natural Lyapunov functional for the evolution of \eqref{aggregation} defined on the set of centered mass densities $\rho \in L^1_+(\R^d)\cap L^m(\R^d)$ given by
\begin{align}
&\mathcal{E} [\rho] = \frac{1}{m-1} \int_{\R^d}  \rho^m(x) \, dx + \frac12  \int_{\R^d\times\R^d} \rho(x)  W(x-y) \rho(y)\, dx\,dy
\label{eq:functional}\\
&\rho(x)\geq 0\, , \quad \int_{\R^d}\rho(x)\, dx = M\geq 0\, , \quad
\int_{\R^d} x\rho(x)\, dx = 0\,  , \nonumber
\end{align}
being the last integral defined in the improper sense, and if $m=1$ we replace the first integral of  $\mathcal{E} [\rho]$ by $ \int_{\R^d} \rho\log\rho dx$. Therefore, if the balance between repulsion and attraction occurs, these two effects should determine stationary states for \eqref{aggregation} including the stable solutions possibly given by local (global) minimizers of the free energy functional \eqref{eq:functional}.

Many properties and results have been obtained in the particular case of Newtonian attractive potential due to its applications in mathematical modeling of chemotaxis \cite{patlak,Keller-Segel-70} and gravitational collapse models \cite{SC}.
In the classical 2D Keller-Segel model with linear diffusion, it is known that equilibria can only happen in the critical mass case \cite{BCC2} while self-similar solutions are the long time asymptotics for subcritical mass cases \cite{BDP,CD}. For supercritical masses, all solutions blow up in finite time \cite{JL}. It was shown in \cite{K,CC} that degenerate diffusion with $m>1$ is able to regularize the 2D classical Keller-Segel problem, where solutions exist globally in time regardless of its mass, and each solution remain uniformly bounded in time. For the Newtonian attraction interaction in dimension $d\geq 3$, the authors in \cite{BCC} show that the value of the degeneracy of the diffusion that allows the mass to be the critical quantity for dichotomy between global existence and finite time blow-up is given by $m=2-2/d$. In fact, based on scaling arguments it is easy to argue that for $m>2-2/d$, the diffusion term dominates when density becomes large, leading to global existence of solutions for all masses. This result was shown in \cite{S} together with the global uniform bound of solutions for all times.

However, in all cases where the diffusion dominates over the aggregation, the long time asymptotics of solutions to \eqref{aggregation} have not been clarified, as pointed out in \cite{BL}. Are there stationary solutions for all masses when the diffusion term dominates? And if so, are they unique up to translations? Do they determine the long time asymptotics for \eqref{aggregation}? Only partial answers to these questions are present in the literature, which we summarize below.

To show the existence of stationary solutions to \eqref{aggregation}, a natural idea is to look for the global minimizer of its associated free energy functional \eqref{eq:functional}. For the 3D case with Newtonian interaction potential and $m>4/3$, Lions' concentration-compactness principle \cite{Lions} gives the existence of a global minimizer of \eqref{eq:functional} for any given mass.  The argument can be extended to kernels that are no more singular than Newtonian potential in $\mathbb{R}^d$ at the origin, and have slow decay at infinity.  The existence result is further generalized by \cite{Bedrossian} to a broader classes of kernels, which can have faster decay at infinity. In all the above cases, the global minimizer of \eqref{eq:functional} corresponds to a stationary solution to \eqref{aggregation} in the sense of distributions. In addition, the global minimizer must be radially decreasing due to Riesz's rearrangement theorem.

Regarding the uniqueness of stationary solutions to \eqref{aggregation}, most of the available results are for Newtonian interaction. For the 3D Newtonian potential with $m>4/3$,  for any given mass, the authors in \cite{LY} prove uniqueness of stationary solutions to \eqref{aggregation} among radial functions, and their method can be generalized to the Newtonian potential in $\mathbb{R}^d$ with $m>2-2/d$. For the 3D case with $m>4/3$, \cite{Strohmer} show that all compactly supported stationary solutions must be radial up to a translation, hence obtaining uniqueness of stationary solutions among compactly supported functions. The proof is based on moving plane techniques, where the compact support of the stationary solution seems crucial, and it also relies on the fact that the Newtonian potential in 3D converges to zero at infinity. Similar results are obtained in \cite{CCV} for 2D Newtonian potential with $m>1$ using an adapted moving plane technique. Again, the uniqueness result is based on showing radial symmetry of compactly supported stationary solutions. Finally, we mention that uniqueness of stationary states has been proved for general attracting kernels in one dimension in the case $m=2$, see \cite{BDF}. To the best of our knowledge, even for Newtonian potential, we are not aware of any results showing that all stationary solutions are radial (up to a translation).

Previous results show the limitations of the present theory: although the existence of stationary states for all masses is obtained for quite general potentials, their uniqueness, crucial for identifying the long time asymptotics, is only known in very particular cases of diffusive dominated problems. The available uniqueness results are not very satisfactory due to the compactly supported restriction on the uniqueness class imposed by the moving plane techniques. And thus, large time asymptotics results are not at all available due to the lack of mass confinement results of any kind uniformly in time together with the difficulty of identifying the long time limits of sequences of solutions due to the restriction on the uniqueness class for stationary solutions.

If one wants to show that the long time asymptotics are uniquely determined by the initial mass and center of mass, a clear strategy used in many other nonlinear diffusion problems, see \cite{Vazquezbook} and the references therein, is the following: one first needs to prove that all stationary solutions are radial up to a translation in a non restrictive class of stationary solutions, then one has to show uniqueness of stationary solutions among radial solutions, and finally this uniqueness will allow to identify the limits of time diverging sequences of solutions, if compactness of these sequences is shown in a suitable functional framework. Let us point out that comparison arguments used in standard porous medium equations are out of the question here due to the lack of maximum principle by the presence of the nonlocal term.

In this work, we will give the first full result of long time asymptotics for a diffusion dominated problem using the previous strategy without smallness assumptions of any kind. More precisely, we will prove that all solutions to the 2D Keller-Segel equation with $m>1$ converge to the global minimizer of its free energy using the previous strategy. The first step will be to show radial symmetry of stationary solutions to \eqref{aggregation} under quite general assumptions on $W$ and the class of stationary solutions. Let us point out that standard rearrangement techniques fail in trying to show radial symmetry of general stationary states to \eqref{aggregation} and they are only useful for showing radial symmetry of global minimizers, see \cite{CCV}. Comparison arguments for radial solutions allow to prove uniqueness of radial stationary solutions in particular cases \cite{LY,KY}. However, up to our knowledge, there is no general result in the literature about radial symmetry of stationary solutions to nonlocal aggregation-diffusion equations.

Our first main result is that all stationary solutions of \eqref{aggregation}, with no restriction on $m>0$, are radially decreasing up to translation by a fully novel application of continuous Steiner symmetrization techniques for the problem \eqref{aggregation}. Continuous Steiner symmetrization has been used in calculus of variations \cite{Brock} for replacing rearrangement inequalities \cite{BLL,LL,Morgan}, but its application to nonlinear nonlocal aggregation-diffusion PDEs is completely new. Most of the results present in the literature using continuous Steiner symmetrization deal with functionals of first order, i.e. functionals involving a power of the modulus of the gradient of the unknown, see \cite[Corollary 7.3]{Brock2} for an application to $p$-Laplacian stationary equations, and in \cite[Section II]{Ka2} and \cite{Ka11,Brock}, while in our case the functional \eqref{eq:functional} is purely of zeroth order. The decay of the attractive Newtonian potential interaction term in $d\geq 3$ follows from \cite[Corollary 2]{Brock} and \cite{Morgan}, which is the only result related to our strategy.

We will construct a curve of measures starting from a stationary state $\rho$ using continuous Steiner symmetrization such that the functional \eqref{eq:functional} decays strictly at first order along that curve unless the base point $\rho$ is radially symmetric, see Proposition \ref{prop:int}. However, the functional \eqref{eq:functional} has at most a quadratic variation when
$\rho$ is a stationary state as the first term in the Taylor expansion cancels. This leads to a contradiction unless the stationary state is radially symmetric. The construction of this curve needs a non-classical technique of slowing-down the velocities of the level sets for the continuous Steiner symmetrization in order to cope with the possible compact support of stationary states in the degenerate case $m>1$, see Proposition \ref{prop:steiner}. This first main result is the content of Section 2 in which we specify the assumptions on the interaction potential and the notion of stationary solutions in details. We point out that the variational structure of \eqref{aggregation} is crucial to show the radially decreasing property of stationary solutions.

The result of radial symmetry for general stationary solutions to \eqref{aggregation} is quite striking in comparison to other gradient flow models in collective behavior based on the competition of attractive and repulsive effects via nonlocal interaction potentials. Actually, there exist numerical and analytical evidence in \cite{KSUB,bigring,BCLR2} that there should be stationary solutions of these fully nonlocal interaction models which are not radially symmetric despite the radial symmetry of the interaction potential. Our first main result shows that this break of symmetry does not happen whenever nonlinear diffusion is chosen to model very strong localized repulsion forces, see \cite{TBC}. Symmetry breaking in nonlinear diffusion equations without interactions has also received a lot of attention lately related to the Caffarelli-Kohn-Nirenberg inequalities, see \cite{DolEstLos,DolEstLos2}. Another consequence of our radial symmetry results is the lack of non-radial local minimizers, and even non-radial critical points, of the free energy functional \eqref{eq:functional}, which is not at all obvious.

We also generalize our radial symmetry result when \eqref{aggregation} has an additional term $\nabla\cdot(\rho\nabla V)$ on the right-hand side, where $V$ is a confining potential (see Section~\ref{sec_v} for precise conditions on $V$), in the sense that it plays the role of preventing particles to drift away in the presence of the diffusion. It is known that with the extra term, the corresponding energy functional has an additional term $\int V(x) \rho(x)\, dx$. The particular case of quadratic confinement $V(x)=\tfrac{|x|^2}2$ is important since it leads to the free energy functional associated to \eqref{aggregation} with homogeneous kernels in self-similar variables \cite{CT00,CCH1,CCH2} and thus, characterizing the self-similar profiles for those problems.

Finally, let us remark that our radial symmetry result applies to stationary states of \eqref{aggregation} for any $m>0$ regardless of being in the diffusion dominated case or not. As soon as stationary states of \eqref{aggregation} exist under suitable assumptions on the interaction potential $W$, and the confining potential $V$ if present, they must be radially symmetric up to a translation. This fact makes our result applicable to the fair-competition cases \cite{BCL,BCC2,BCM08} and the aggregation-dominated cases, see \cite{LS,CLW,CW} with degenerate, linear or fast diffusion. Section \ref{sec:gradient_flow} is finally devoted to deal with the most restrictive case of $\lambda$-convex potentials and the Newtonian potential with $m\geq 1-\tfrac1{d}$. In these cases, we can directly make use of the key first-order decay result of the interaction energy along Continuous Steiner symmetrization curves in Proposition \ref{prop:int}, bypassing the technical result in Proposition \ref{prop:steiner}, in order to give a nice shortcut of the proof of our main Theorem 2.2 based on gradient flow techniques.

We next study more properties of particular radially decreasing stationary solutions. We make use of the variational structure to show the existence of global minimizers to \eqref{eq:functional} under very general hypotheses on the interaction potential $W$ and $m>1$. In Section 3, we show that these global minimizers are in fact radially decreasing continuous functions, compactly supported if $m>1$. These results fully generalize the results in \cite{Strohmer,CCV}. Putting together Sections 2 and 3, the uniqueness and full characterization of the stationary states is reduced to uniqueness among the class of radial solutions. This result is known in the case of Newtonian attraction kernels \cite{LY}.

Finally, we make use of the uniqueness among translations for any given mass of stationary solutions to \eqref{aggregation} to obtain the second main result of this work, namely to answer the open problem of the long time asymptotics to \eqref{aggregation} with Newtonian interaction in 2D and $m>1$. This is accomplished in Section 4 by a compactness argument for which one has to extract the corresponding uniform in time bounds and a careful treatment of the nonlinear terms and dissipation while taking the limit $t\to\infty$. We do not know how to obtain a similar result for Newtonian interaction in $d\geq 3$ due to the lack of uniform in time mass confinement bounds in this case. We essentially cannot show that mass does not escape to infinity while taking the limit $t\to\infty$. However, the compactness and characterization of stationary solutions is still valid in that case.

The present work opens new perspectives to show radial symmetry for stationary solutions to nonlocal aggregation-diffusion problems. While the hypotheses of our result to ensure existence of global radially symmetric minimizers of \eqref{eq:functional}, and in turn of stationary solutions to \eqref{aggregation}, are quite general, we do not know yet whether there is uniqueness among radially symmetric stationary solutions (with a fixed mass) for general non-Newtonian kernels. We even do not have available uniqueness results of radial minimizers beyond Newtonian kernels. Understanding if the existence of radially symmetric local minimizers, that are not global, is possible for functionals of the form \eqref{eq:functional} with radial interaction potential is thus a challenging question. Concerning the long-time asymptotics of \eqref{aggregation}, the lack of a novel approach to find confinement of mass beyond the usual virial techniques and comparison arguments in radial coordinates hinders the advance in their understanding even for Newtonian kernels with $d\geq 3$. Last but not least, our results open a window to obtain rates of convergence towards the unique equilibrium up to translation for the Newtonian kernel in 2D. The lack of general convexity of this variational problem could be compensated by recent results in a restricted class of functions, see \cite{CLM}. However, the problem is quite challenging due to the presence of free boundaries in the evolution of compactly supported solutions to \eqref{aggregation} that rules out direct linearization techniques as in the linear diffusion case \cite{CD}.


\section{Radial Symmetry of stationary states with degenerate diffusion}

Throughout this section, we assume that $m>0$, and $\K$ satisfies the following four assumptions:
\begin{enumerate}[(a)]
\item[(K1)] $\K$ is attracting, i.e., $\K(x) \in C^1(\mathbb{R}^d \setminus \{0\})$ is radially symmetric
\[W(x)=\omega(|x|)=\omega(r)\] and $\omega'(r)>0$ for all $r>0$ with $\omega(1)=0$.
\item[(K2)] $\K$ is no more singular than the Newtonian kernel in $\mathbb{R}^d$ at the origin, i.e., there exists some $C_w>0$ such that $\omega'(r)
    \leq
    C_w r^{1-d}$ for $r\leq 1$.
\item[(K3)] There exists some $C_w>0$ such that $\omega'(r) \leq C_w$ for all $r>1$.
\item[(K4)] Either $\omega(r)$ is bounded for $r\geq 1$  or there exists $C_w>0$ such that for all $a,b\geq0$:
\begin{equation*}
\omega_+(a+b)\leq C_w (1+\omega(1+a)+\omega(1+b))\,.
\end{equation*}
\end{enumerate}
As usual, $\omega_\pm$ denotes the positive and negative part of $\omega$ such that $\omega=\omega_+-\omega_-$. In particular, if $\K=-\NN$, modulo the addition of a constant factor, is the attractive Newtonian potential, where $\NN$ is the fundamental solution of $-\Delta$ operator in $\mathbb{R}^d$, then $\K$ satisfies all the assumptions. Since the equation \eqref{aggregation} does not change by adding a constant to the potential $W$, we will consider that the potential $W$ is defined modulo additive constants from now on.

We denote by $L^{1}_{+}(\R^{d})$ the set of all nonnegative functions in $L^{1}(\R^{d})$. Let us start by defining precisely stationary states to the aggregation equation \eqref{aggregation} with a potential satisfying (K1)-(K4).

\begin{definition}\label{stationarystates}
Given $\rho_s \in L^1_+(\mathbb{R}^d)\cap L^\infty(\mathbb{R}^d)$ we call it a \textbf{stationary state} for the evolution problem \eqref{aggregation} if $\rho_s^{m}\in H^1_{loc} (\mathbb{R}^d)$, $\nabla \psi_s:=\nabla W \ast \rho_s\in L^1_{loc} (\mathbb{R}^d)$, and it satisfies
\begin{equation}\label{steady}
\nabla \rho_s^{m} = - \rho_s\nabla \psi_s \text{ in } \mathbb{\R}^d
\end{equation}
in the sense of distributions in $\R^d$.
\end{definition}

Let us first note that $\nabla\psi_s$ is globally bounded under the assumptions (K1)-(K3). To see this, a direct decomposition in near- and far-field sets yields
\begin{align}
|\nabla \psi_s(x)| &\leq \int_{\R^d} | \nabla \K(x-y) | \rho_s(y) \,dy \leq C_w \int_{\mathcal{A}}\frac{1}{|x-y|^{d-1}} \rho_s(y) \,dy +
          C_w \int_{\mathcal{B}} \rho_s(y) \,dy \label{lipspot0}\\
&\leq C_w\int_{\mathcal{A}} \frac{1}{|x-y|^{d-1}} dy \, \|\rho_s\|_{L^\infty(\R^d)}+C_w \|\rho_s\|_{L^1(\R^d)} \leq C (\|\rho_s\|_{L^1(\R^d)}+\|\rho_s\|_{L^\infty(\R^d)})\,. \nonumber
\end{align}
where we split the integrand into the sets
$\mathcal{A} := \{ y : |x - y| \leq 1 \}$ and $\mathcal{B} :=
\R^d \setminus \mathcal{A}$, and apply the assumptions (K1)-(K3).

Under the additional assumptions (K4) and $\omega(1+|x|)\rho_s \in L^1(\mathbb{R}^d)$, we will show that the potential function $\psi_s(x) = W*\rho_s(x)$ is also locally bounded. First, note that (K1)-(K3) ensures that $|\omega(r)| \leq \tilde C_w \phi(r)$ for all $r\leq 1$ with some $\tilde C_w>0$, where
\begin{equation}\label{defphi}
\phi(r):=\left\{\begin{array}{lr}r^{2-d}-1 & \mbox{if } d\geq 3\\
                                             -\log(r)  & \mbox{if } d=2 \\
                                                    1-r & \mbox{if } d=1 \end{array}\right..
\end{equation}
Hence we can again perform a decomposition in near- and far-field sets and obtain
\begin{align}
|\psi_s(x)| &\leq \int_{\R^d} |\K(x-y) | \rho_s(y) \,dy \leq C_w \int_{\mathcal{A}}\phi(|x-y|) \rho_s(y) \,dy + \int_{\mathcal{B}} \omega_+(|x|+|y|) \rho_s(y) \,dy \nonumber\\
&\leq C_w \int_{\mathcal{A}} \phi(|x-y|) dy \, \|\rho_s\|_{L^\infty(\R^d)}+C_w(1+\omega(1+|x|))\|\rho_s\|_{L^1(\R^d)} + C_w \|\omega(1+|x|)\rho_s\|_{L^1(\R^d)}  \nonumber\\
&\leq C (\|\rho_s\|_{L^1(\R^d)}+\|\rho_s\|_{L^\infty(\R^d)})+\omega(1+|x|)\|\rho_s\|_{L^1(\R^d)} + C_w \|\omega(1+|x|)\rho_s\|_{L^1(\R^d)}\,. \label{lipspot}
\end{align}

Our main goal in this section is the following theorem.
\begin{theorem}
\label{thm:unique}
Assume that $W$ satisfies $(K1)$-$(K4)$ and $m>0$. Let $\rho_s \in L^1_+(\mathbb{R}^d)\cap L^\infty(\mathbb{R}^d)$ with $\omega(1+|x|)\rho_s \in L^1(\mathbb{R}^d)$ be a non-negative stationary state of \eqref{aggregation} in the sense of Definition {\rm \ref{stationarystates}}. Then $\rho_s$ must be radially decreasing up to a translation, i.e. there exists some $x_0\in \mathbb{R}^d$, such that $\rho_s(\cdot - x_0)$ is radially symmetric, and $\rho_s(|x-x_0|)$ is non-increasing in $|x-x_0|$.
\end{theorem}

Before going into the details of the proof, we briefly outline the strategy here.
Assume there is a stationary state $\rho_s$ which is \textbf{not} radially
decreasing under any translation. To obtain a contradiction, we consider the free energy functional $\mathcal{E}[\rho]$ associated with
\eqref{aggregation},
\begin{equation}\label{eq:energy}
\mathcal{E}[\rho] = \frac{1}{m-1} \int_{\mathbb{R}^d} \rho^m dx + \frac{1}{2} \int_{\mathbb{R}^d} \rho (\K*\rho) dx =: \mathcal{S}[\rho] + \mathcal{I}[\rho],
\end{equation}
where $\mathcal{S}[\rho]$ is replaced by $\int \rho\log \rho \,dx$ if $m=1$.
We first observe that $\mathcal{I}[\rho_s]$ is finite since the potential function $\psi_s=W\ast\rho_s \in
\W^{1,\infty}_{loc}(\R^d)$ satisfies \eqref{lipspot} with $\omega(1+|x|)\rho_s \in L^1(\mathbb{R}^d)$. Since $\rho_s \in L^1_+(\mathbb{R}^d)\cap L^\infty(\mathbb{R}^d)$, $\mathcal{S}[\rho_s]$ is finite for all $m>1$, but may be $-\infty$ if $m\in(0,1]$.

Below we discuss the strategy for $m>1$ first, and point out the modification for $m\in (0,1]$ in the next paragraph. Using the assumption that $\rho_s$ is not radially decreasing under any translation, we will apply the continuous Steiner symmetrization to perturb around
$\rho_s$ and construct a continuous family of densities $\mu(\t, \cdot)$ with $\mu(0,\cdot)=\rho_s$, such that $\mathcal{E}[\mu(\t)] -
\mathcal{E}[\rho_s] <
-c\t$
for some $c>0$ and any small $\t>0$.  On the other hand, using that $\rho_s$ is a stationary state, we will show that $|\mathcal{E}[\mu(\t)] -
\mathcal{E}[\rho_s]|
\leq C\t^2$ for some $C>0$ and any small $\t>0$. Combining these two inequalities together gives us a contradiction for sufficiently small  $\t>0$.

For $m\in (0,1)$, even if $\mathcal{S}[\rho_s]$ might be $-\infty$ by itself, the difference $\mathcal{S}[\mu(\tau)] - \mathcal{S}[\rho_s] $ can be still well-defined in the following sense, if we regularize the function $\frac{1}{m-1}\rho^{m}$ by $\frac{1}{m-1}\rho (\rho+\epsilon)^{m-1}$ and take the limit $\epsilon\to 0$:
\begin{equation}\label{def_reg_s}
\mathcal{S}[\mu(\tau)] - \mathcal{S}[\rho_s] := \lim_{\epsilon\to 0} \int \frac{1}{m-1}\Big(\mu(\tau,x) (\mu(\tau,x)+ \epsilon)^{m-1} - \rho_s(x) (\rho_s(x)+ \epsilon)^{m-1}  \Big) dx,
\end{equation}
and if $m=1$ the integrand is replaced by $\mu(\tau,\cdot) \log (\mu(\tau,\cdot)+\epsilon) - \rho_s \log(\rho_s +\epsilon)$. Note that as long as $\mu(\tau)$ has the same distribution as $\rho_s$, the above definition gives $\mathcal{S}[\mu(\tau)] - \mathcal{S}[\rho_s] =0$. With such modification, we will show that the difference $\mathcal{E}[\mu(\t)] - \mathcal{E}[\rho_s]$ is well-defined and satisfies the same two inequalities as the $m>1$ case, so we again have a contradiction for small $\t>0$.

If the kernel $W$ has certain convexity properties and $m\geq 1-\tfrac1{d}$, then it is known that \eqref{aggregation} has a rigorous Wasserstein gradient flow structure. In this case, once we obtain the crucial estimate: $\mathcal{E}[\mu(\t)] -\mathcal{E}[\rho_s] <-c\t$, there is a shortcut that directly lead to the radial symmetry result, which we will discuss in Section \ref{sec:gradient_flow}.

Let us characterize first the set of possible stationary states of \eqref{aggregation} in the sense of Definition \ref{stationarystates} and their regularity. Parts of these arguments are reminiscent from those done in \cite{Strohmer,CCV} in the case of attractive Newtonian potentials.

\begin{lemma}
\label{lem:regularity}
Let $\rho_s \in L^1_+(\mathbb{R}^d)\cap L^\infty(\mathbb{R}^d)$ with $\omega(1+|x|)\rho_s \in L^1(\mathbb{R}^d)$ be a non-negative stationary state of
\eqref{aggregation} for some $m>0$ in the sense of Definition {\rm\ref{stationarystates}}. Then $\rho_s \in \mathcal{C}(\mathbb{R}^d)$, and there exists some
 $C = C(\|\rho_s\|_{L^1}, \|\rho_s\|_{L^\infty}, C_w, d)>0$, such that
\begin{equation}\label{temp_grad1}
\frac{m}{|m-1|}|\nabla (\rho_s^{m-1})| \leq C \quad \text{ in }\supp \rho_s \quad \text{ if } m\neq 1,
\end{equation}
and
\begin{equation}\label{temp_grad2}
|\nabla \log \rho| \leq C \quad \text{ in }\supp \rho_s \quad \text{ if } m= 1.
\end{equation}
In addition, if $m \in (0,1]$, then $\supp \rho_s = \mathbb{R}^d$.
\end{lemma}

\begin{proof}
We have already checked that under these assumptions on $W$ and $\rho_s$, the potential function $\psi_s\in \W^{1,\infty}_{loc}(\R^d)$ due to
\eqref{lipspot0}-\eqref{lipspot}. Since $\rho_s^{m}\in H^1_{loc} (\mathbb{R}^d)$, then $\rho_s^m$ is a weak $H^1_{loc}(\R^{d})$ solution of
\begin{equation}\label{steady2}
\Delta \rho_s^{m} = -\nabla\cdot\left(\rho_s\nabla \psi_s\right) \text{ in } \mathbb{\R}^d
\end{equation}
with right hand side belonging  to $\W^{-1,p}_{loc} (\mathbb{R}^d)$ for all $1\leq p \leq \infty$. As a consequence, $\rho_s^m$ is in fact a weak solution in
$\W_{loc}^{1,p}(\R^d)$ for all $1<p<\infty$ of \eqref{steady2} by classical elliptic regularity results. Sobolev embedding shows that $\rho_s^m$ belongs
to some
H\"older space $\mathcal{C}_{loc}^{0,\alpha}(\mathbb{R}^d)$, and thus $\rho_s\in \mathcal{C}_{loc}^{0,\beta}(\mathbb{R}^d)$ with $\beta := \min\{\alpha/m, 1\}$.
Let us define the set $\Omega=\{x\in\R^d : \rho_s(x)>0\}$. Since $\rho_s\in \mathcal{C}(\R^d)$, then $\Omega$ is an open set and it consists
of a countable number of open possibly unbounded connected components. Let us take any bounded smooth connected open subset $\Theta$ such that $\overline{ \Theta} \subset \Omega$, and start with the case $m\neq 1$.
Since $\rho_s\in \mathcal{C}(\R^d)$, then $\rho_s$ is bounded away from zero in $\Theta$ and thus due to the assumptions on $\rho_s$, we have that
$\frac{m}{m-1}\nabla\rho_s^{m-1} = \frac{1}{\rho_s} \nabla \rho_{s}^{m}$ holds in the distributional sense in $\Theta$. We conclude that wherever $\rho_s$ is positive,
\eqref{steady} can be interpreted as
\begin{equation}\label{steady3}
\nabla \left( \frac{m}{m-1} \rho_s^{m-1} +\psi_s \right) = 0\,,
\end{equation}
in the sense of distributions in $\Omega$. Hence, the function $G(x)=\frac{m}{m-1}\rho_s^{m-1}(x) +\psi_s (x)$ is constant in each connected component of
$\Omega$. From here, we deduce that any stationary state of \eqref{aggregation} in the sense of Definition \ref{stationarystates} is given by
\begin{equation}
\label{eq:rho_s_stat}
\rho_s(x)= \left(\frac{m-1}{m}(G-\psi_s)(x)\right)_+^{\tfrac{1}{m-1}}\,,
\end{equation}
where $G(x)$ is a constant in each connected component of the support of $\rho_s$, and its value may differ in different connected components. Due to $\psi_s\in \W^{1,\infty}_{loc}(\R^d)$, we deduce that $\rho_s \in \mathcal{C}_{loc}^{0,1/(m-1)}(\mathbb{R}^d)$ if $m\geq 2$ and $\rho_s \in\mathcal{C}_{loc}^{0,1}(\mathbb{R}^d)$ for $m \in (0,1)\cup(1,2)$. Putting together \eqref{eq:rho_s_stat} and \eqref{lipspot0}, we conclude the desired estimate.

In addition, from \eqref{eq:rho_s_stat} we have that $\Omega = \mathbb{R}^d$ if $m \in (0,1)$: if not, let $\Omega_0$ be any connected component of $\Omega$, and take $x_0 \in \partial \Omega_0$. As we take a sequence of points $x_n \to x_0$ with $x_n \in \Omega_0$, we have that $\rho_s(x_n)^{m-1}\to \infty$, whereas the sequence $G(x_n) - \psi_s(x_n)$ is bounded (since $\psi_s$ is locally bounded due to \eqref{lipspot}), a contradiction.

If $m=1$, the above argument still goes through except that we replace \eqref{steady3} by
\[
\nabla \left( \log \rho_s +\psi_s \right) = 0
\]
in the sense of distributions in $\Omega$. As a result, the function $G(x)=\log \rho_s +\psi_s (x)$ is constant in each connected component of
$\Omega$. The same argument as the $m\in (0,1)$ case then yields that $\rho_s \in \mathcal{C}_{loc}^{0,1}(\mathbb{R}^d)$ and $\Omega = \mathbb{R}^d$, leading to the estimate $|\nabla \log \rho| \leq C$  in $\mathbb{R}^d$.
\end{proof}



\subsection{Some preliminaries about rearrangements}
Now we briefly recall some standard notions and basic properties of decreasing rearrangements for nonnegative functions that will be used later. For a deeper treatment of these topics, we address the reader to the books \cite{Hardy,Bennett-Sharpley,Ka1,Kesavan,LL} or the papers \cite{Talenti1,Talenti2,Talenti3,Mossino}. We denote by $|E|_{d}$ the Lebesgue measure of a measurable set $E$ in $\R^{d}$. Moreover, the set $E^{\#}$ is defined as the ball centered at the origin such that $|E^{\#}|_{d}=|E|_{d}$. \\[5pt]
A nonnegative measurable function $f$ defined on $\R^{d}$ is called \emph{radially symmetric} if there is a nonnegative function
$\widetilde{f}$ on $[0,\infty)$ such that $f(x)=\widetilde{f}(|x|)$ for all $x\in \R^{d}$. If $f$ is radially symmetric, we will often write $f(x)=f(r)$ for
$r=|x|\ge0$ by a slight abuse of notation. We say that $f$ is \emph{rearranged} if it is radial and $\widetilde{f}$ is a nonnegative
right-continuous,
non-increasing
function of $r>0$. A similar definition can be applied for real functions defined on a ball $B_{R}(0)=\left\{x\in\R^d:|x|<R\right\}$.\\[0.5pt]
We define \emph{the distribution function of $f\in L^{1}_{+}(\R^{d})$} by
\[
\zeta_{f}(\t)=|\left\{x\in\R^{d}:f(x)>\t\right\}|_{d},\quad  \text{ for all }\t>0.
\]
Then the function $f^{\ast}:[0,+\infty)\rightarrow[0,+\infty]$ defined by
\[
f^{\ast}(s)=\sup\left\{\t>0:\zeta_{f}(\t)>s \right\},\quad  s\in [0,+\infty),
\]
will be called the \emph{Hardy-Littlewood one-dimensional decreasing rearrangement of $f$}. By this definition, one could interpret
$f^{\ast}$ as the generalized right-inverse function of $\zeta_{f}(\t)$.

Making use of the definition of $f^{\ast}$, we can define a special radially symmetric decreasing function $f^{\#}$, which we will call the
\emph{Schwarz spherical decreasing rearrangement of $f$} by means of the formula
\begin{equation}
f^{\#}(x)=f^{\ast}(\omega_{d}|x|^{d}) \quad x\in \R^{d},\label{Schwarzsymm}
\end{equation}
where $\omega_d$ is the volume of the unit ball in $\mathbb{R}^d$.
It is clear that if the set $\Omega_{f}=\left\{x\in\R^{d}:f(x)>0\right\}$ of $f$ has finite measure, then $f^{\#}$ is supported in the ball
$\Omega_{f}^{\#}$.\\
One can show that $f^{\ast}$ (and so $f^{\#}$) is equidistributed with $f$ (\emph{i.e.} they have the same distribution function). Thus if $f\in L^{p}(\R^d)$, a simple use of Cavalieri's principle (see \emph{e.g.} \cite{Talenti2, Kesavan}) leads to the invariance property of the $L^{p}$ norms:
\begin{equation}\label{declp}
\|f\|_{L^{p}(\R^{d})} = \|f^\ast\|_{L^{p}(0,\infty)}=\|f^\#\|_{L^{p}(\R^{d})} \qquad  \mbox{for all } 1\leq p \leq \infty \,.
\end{equation}
In particular,using the layer-cake representation formula (see \emph{e.g.} \cite{LL}) one could easily infer that
\[
f^\#(x) = \int_0^\infty \chi_{\{f>\t\}^\#} d\t.
\]
Among the many interesting properties of rearrangements, it is worth mentioning the  \emph{Hardy-Littlewood} inequality (see \cite{Hardy,Bennett-Sharpley,Kesavan} for the proof): for any couple of nonnegative measurable functions $f,\,g$ on $\R^{d}$, we have
\begin{equation}
\int_{\R^d} f(x)g(x)dx\leq \int_{\R^{d}} f^{\#}(x)g^{\#}(x)dx\label{Hardylitt}.
\end{equation}

Since in Section 4 we will use estimates of the solutions Keller-Segel problems in terms of their integrals,
let us now recall the concept of comparison of mass concentration, taken from  \cite{Vsym82}, that is remarkably useful.
\begin{definition}\label{Massconcdef}
Let $f,g\in L^{1}_{loc}(\R^{d})$ be two nonnegative, radially symmetric functions on $\R^{d}$. We say that $f$ is less concentrated than $g$, and we write
$f\prec
g$ if for
all $R>0$ we get
\[
\int_{B_{R}(0)}f(x)dx\leq \int_{B_{R}(0)}g(x)dx.
\]
\end{definition}
The partial order relationship $\prec$ is called \emph{comparison of mass concentrations}.
Of course, this definition can be suitably adapted if $f,g$ are radially symmetric and locally integrable functions on a ball $B_{R}$.
The comparison of mass concentrations enjoys a nice equivalent formulation if $f$ and $g$ are rearranged, whose proof we refer to \cite{AlvTrombLion,Chong, VANS05}:

\begin{lemma}\label{lemma1}
Let $f,g\in L^{1}_{+}(\R^{d})$ be two nonnegative rearranged functions. Then $f\prec g$ if and only if for every convex
nondecreasing
function
$\Phi:[0,\infty)\rightarrow [0,\infty)$ with $\Phi(0)=0$ we have
\begin{equation*}
\int_{\Omega}\Phi(f(x))\,dx\leq \int_{\Omega}\Phi(g(x))\,dx.
\end{equation*}
\end{lemma}
From this Lemma, it easily follows that if $f\prec g$ and $f,g\in L^{p}(\R^{d})$ are rearranged and non-negative, then
\begin{equation*}
\|f\|_{L^{p}(\R^{d})}\leq \|g\|_{L^{p}(\R^{d})}\quad \forall p\in[1,\infty].
\end{equation*}
Let us also observe that if $f,g\in L^{1}_{+}(\R^{d})$ are nonnegative and rearranged, then $f\prec g$ if and only if for all $s\geq0$ we have
\begin{equation*}
\int_{0}^{s}f^{\ast}(\sigma)d\sigma\leq \int_{0}^{s}g^{\ast}(\sigma)d\sigma.
\end{equation*}
If $f\in L^1_+(\R^d)$, we denote by $M_2[f]$ the second moment of $f$, \emph{i.e.}
\begin{equation}M_2[f] := \int_{\mathbb{R}^d} f(x)|x|^2dx.\label{secondmoment}\end{equation}
In this regard, another interesting property which will turn out useful is the following
\begin{lemma}
\label{lem:lp}
Let $f, g \in L^1_+(\mathbb{R}^d)$ with $\|f\|_{L^1(\R^d)} = \|g\|_{L^1(\R^d)}$. If additionally $g$ is rearranged and $f^\# \prec g$, then $M_2[f] \geq M_2[g]$.
\end{lemma}
\begin{proof}
Let us consider the sequence of bounded radially increasing functions  $\left\{\varphi_{n}\right\}$, where $\varphi_{n}(x)=\min\left\{|x|^{2},n\right\}$ is the truncation of the  function $|x|^{2}$ at the level $n$ and define the function
\[
h_{n}=n-\varphi_{n}.
\]
Then $h_{n}$ is nonnegative, bounded and rearranged. Thus using the Hardy-Littlewood inequality \eqref{Hardylitt} and \cite[Corollary 2.1]{AlvTrombLion} we find
\begin{align*}
\int_{\R^{d}}f(x)\,\varphi_{n}(x)dx&=n\|f\|_{L^1(\R^d)}-\int_{\R^{d}}f(x)\,h_{n}(x)dx \geq n\|f\|_{L^1(\R^d)}-\int_{\R^{d}}f^{\#}(x)\,h_{n}(x)dx\\
& \geq  n\|g\|_{L^1(\R^d)}-\int_{\R^{d}}g(x)\,h_{n}(x)dx =\int_{\R^{d}}g(x)\,\varphi_{n}(x)dx
\end{align*}
Then passing to the limit as $n\rightarrow\infty$ we find the desired result.
\end{proof}

\begin{remark}\label{remarkV}
Lemma \ref{lem:lp} can be easily generalized when $|x|^{2}$ is replaced by any nonnegative radially increasing potential $V=V(r)$, $r=|x|$, such that
\[
\lim_{r\rightarrow+\infty}V(r)=+\infty.
\]
\end{remark}

\subsection{Continuous Steiner symmetrization}
Although classical decreasing rearragement techniques are very useful to study properties of the minimizers and for solutions of the evolution problem
\eqref{aggregation} in next sections, we do not know how to use them in connection with showing that stationary states are radially symmetric.
For an introduction of continuous Steiner symmetrization and its properties, see \cite{BLL, Brock, LL}. In this subsection, we will use continuous
Steiner symmetrization to prove the following proposition.

\begin{proposition}
\label{prop:steiner}
Let $\mu_0 \in \mathcal{C}(\mathbb{R}^d) \cap L^1_+(\mathbb{R}^d)$, and assume it is \textbf{not} radially decreasing after any translation.

Moreover, if $m\in (0,1)\cup (1,\infty)$, assume that $|\frac{m}{m-1}\nabla \mu_0^{m-1}| \leq C_0$ in $\supp \mu_0$ for some $C_0$; and if $m=1$ assume that $|\nabla \log \mu_0| \leq C_0$ in $\supp \mu_0$ for some $C_0$. In addition, if $m\in (0,1]$, assume that $\supp \mu_0 = \mathbb{R}^d$.

Then there exist some $\delta_0>0, c_0>0, C_1>0$ (depending on m, $\mu_0$ and $\K$) and a function $\mu \in C([0,\delta_0]\times \mathbb{R}^d)$ with $\mu(0,\cdot) = \mu_0$, such that $\mu$ satisfies the following for a short time
$\t\in
[0,\delta_0]$, where $\mathcal{E}$ is as given in \eqref{eq:energy}:
\begin{equation} \label{E_mu_t}
\mathcal{E}[\mu(\t)] - \mathcal{E}[\mu_0] \leq - c_0 \t,
\end{equation}
\begin{equation} \label{mu_t}
|\mu(\t,x) - \mu_0(x)| \leq C_1 \mu_0(x)^{\max\{1,2-m\}} \t \quad \text{ for all } x\in \mathbb{R}^d,
\end{equation}
\begin{equation}\label{int_0}
 \int_{D_i} \mu(\t,x)-\mu_0(x)dx =0 \text{ for any connected component $D_i$ of $\supp\mu_0$.}
 \end{equation}
\end{proposition}

\subsubsection{Definitions and basic properties of Steiner symmetrization} Let us first introduce the concept of Steiner symmetrization for a
measurable set $E\subset\R^{d}$ . If $d=1$, the Steiner symmetrization of $E$ is the symmetric interval $S(E)=\left\{x\in\R:|x|<|E|_{1}/2\right\}$. Now we want to define the Steiner symmetrization of $E$ with respect to a direction in $\R^{d}$ for $d\geq2$. The direction we symmetrize corresponds to the unit vector $e_{1}=(1,0,\ldots,0)$, although the definition can be modified accordingly when considering any other direction in $\R^{d}$.
\\ \indent
Let us label a point $x\in\R^{d}$ by $(x_{1},x^{\prime})$, where $x^{\prime}=(x_{2},\ldots,x_{d})\in\R^{d-1}$ and $x_{1}\in\R$.
Given any measurable subset $E$ of $\R^{d}$ we define, for all $x^{\prime}\in\R^{d-1}$, the \emph{section} of $E$ with respect to the direction $x_{1}$ as the set
\[
E_{x^{\prime}}=\left\{x_{1}\in \R:(x_{1},x^{\prime})\in E \right\}.
\]
Then we define the Steiner symmetrization of $E$ with respect to the direction $x_{1}$ as the set $S(E)$ which is symmetric about the
hyperplane $\left\{x_{1}=0\right\}$ and is defined by
\[
S(E)=\left\{(x_{1},x^{\prime})\in \R^{d}:x_{1}\in S(E_{x^{\prime}}) \right\}.
\]
In particular we have that $|E|_{d}=|S(E)|_{d}$.

Now, consider a non-negative function $\mu_{0}\in L^{1}(\R^{d})$, for $d\geq2$. For all $x^{\prime}\in\R^{d-1}$, let us consider the distribution function of $\mu_{0}(\cdot,x^{\prime})$, \emph{i.e.} the function
\[
\zeta_{\mu_{0}}(h,x^{\prime})=|U_{x'}^h|_{1}\quad\text{ for all }h>0,\,x^{\prime}\in\R^{d-1},
\]
where
\begin{equation}
U_{x'}^h = \{x_1 \in \mathbb{R}: \mu_0(x_1, x')>h\}\label{U}.
\end{equation}
Then we can give the following definition:
\begin{definition}
 We define the \emph{Steiner symmetrization} (or \emph{Steiner rearrangement})  of $\mu_{0}$ in the direction $x_{1}$ as the function
$S \mu_{0}=S \mu_{0}(x_{1},x^{\prime})$ such that $S \mu_{0}(\cdot,x^{\prime})$ is exactly the Schwarz rearrangement of $\mu_{0}(\cdot,x^{\prime})$
\emph{i.e.} (see \eqref{Schwarzsymm})
\[
S\mu_{0}(x_{1},x^{\prime})=\sup\left\{h>0:\zeta_{\mu_{0}}(h,x^{\prime})>2|x_{1}|\right\}.
\]
\end{definition}
As a consequence, the Steiner symmetrization $S\mu_{0}(x_{1},x^{\prime})$ is a function being symmetric about the hyperplane
$\left\{x_{1}=0\right\}$ and for each $h>0$ the level set
\[
\left\{(x_{1},x^{\prime}):S\mu_{0}(x_{1},x^{\prime})>h\right\}
\]
is equivalent to the Steiner symmetrization
\[
S(\left\{(x_{1},x^{\prime}):\mu_{0}(x_{1},x^{\prime})>h\right\})
\]
which implies that $S\mu_{0}$ and $\mu_{0}$ are equidistributed, yielding the invariance of the $L^{p}$ norms when passing from $\mu_{0}$ to $S\mu_{0}$, that is for all $p\in [1,\infty]$ we have
\[
\|S\mu_{0}\|_{L^{p}(\R^{d})}=\|\mu_{0}\|_{L^{p}(\R^{d})}.
\]Moreover, by the layer-cake representation formula, we have
\begin{equation}
S\mu_{0}(x_{1},x^{\prime})=\int_0^\infty \chi_{S(U_{x'}^h)}(x_1) \,dh\,.
\label{Steinerrepres}
\end{equation}
Now, we introduce a continuous version of this Steiner procedure via an interpolation between a set or a function and their Steiner symmetrizations that we will use in our symmetry arguments for steady states.

\begin{definition}
\label{def:steiner_set}
For an open set $U\subset \mathbb{R}$, we define its \emph{continuous Steiner symmetrization} $M^\t(U)$ for any $\t\geq 0$ as below. In the following we
abbreviate
an open interval $(c-r, c+r)$ by $I(c,r)$, and we denote by $\sgn c$ the sign of $c$ (which is $1$ for positive $c$, $-1$ for negative $c$, and $0$ if
$c=0$).
\begin{enumerate}[(1)]
\item If $U = I(c,r)$, then
\[
M^\t(U):=  \begin{cases}I(c-\t\,\sgn c, r) & \text{ for }0\leq \t< |c|,\\
 I(0,r) &\text{ for }\t\geq |c|.
 \end{cases}
 \]
\item If $U = \cup_{i=1}^N I(c_i,r_i)$ (where all $I(c_i, r_i)$ are disjoint), then $M^\t(U) := \cup_{i=1}^N M^\t(I(c_i, r_i))$ for $0\leq \t<\t_1$,
where
    $\t_1$
    is the first time two intervals $M^\t(I(c_i, r_i))$ share a common endpoint. Once this happens, we merge them into one open interval, and repeat
    this
    process starting from $\t=\t_1$.
\item If $U = \cup_{i=1}^\infty I(c_i, r_i)$ (where all $I(c_i, r_i)$ are disjoint), let $U_N = \cup_{i=1}^N I(c_i, r_i)$ for each $N>0$, and define
    $M^\t(U)
    := \cup_{N=1}^\infty M^\t(U_N)$.
\end{enumerate}
\end{definition}

See Figure \ref{fig:steiner} for illustrations of $M^\t(U)$ in the cases (1) and (2). Also, we point out that case (3) can be seen as a limit of
case (2),
since
for each $N_1<N_2$ one can easily check that $M^\t(U_{N_1}) \subset M^\t(U_{N_2})$ for all $\t\geq 0$. Moreover, according to \cite{Brock}, the definition of $M^{\t}(U)$ can be extended to any measurable set $U$ of $\R$, since
\[
U=\bigcap_{n=1}^{\infty} O_{n}\setminus N,
\]
being $O_{n}\supset O_{n+1}$ $n=1,2,\ldots,$ open sets and $N$ a nullset.
\begin{figure}[ht!]
\begin{center}
\includegraphics[scale=0.8]{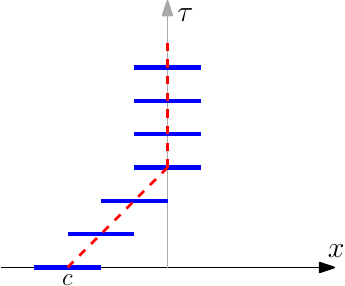}\hspace{2cm}\includegraphics[scale=0.8]{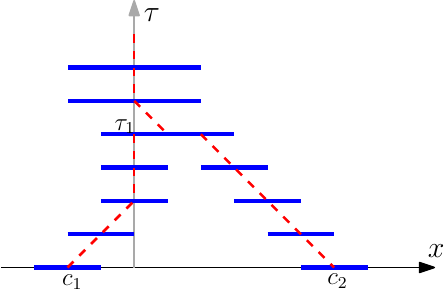}
\caption{Illustrations of $M^\t(U)$ when $U$ is a single open interval (left), and when $U$ is the union of two open intervals (right).
\label{fig:steiner}}
\end{center}
\end{figure}

In the next lemma we state four simple facts about $M^\t$. They can be easily checked for case (1) and (2) (hence true for (3) as well by taking the
limit), and
we omit the proof.
\begin{lemma}
\label{lem:prop_steiner}
Given any open set $U\subset \mathbb{R}$, let $M^\t(U)$ be defined in Definition {\rm\ref{def:steiner_set}}. Then
\begin{enumerate}[(a)]
\item $M^{0}(U)=U$, $M^{\infty}(U)=S(E)$.
\item $|M^\t(U)| = |U|$ for all $\t\geq 0$.
\item If $U_1 \subset U_2$, we have $M^\t(U_1) \subset M^\t(U_2)$ for all $\t\geq 0$.
\item $M^\t$  has the semigroup property: $M^{\t+s}U = M^\t(M^s(U))$ for any $\t,s\geq 0$ and open set $U$.
\end{enumerate}
\end{lemma}

Once we have the continuous Steiner symmetrization for a one-dimensional set, we can define the continuous Steiner symmetrization (in a certain direction) for a non-negative function in $\mathbb{R}^d$.

\begin{definition}
\label{def:steiner_func}
Given $\mu_0 \in L^1_+(\mathbb{R}^d)$, we define its continuous Steiner symmetrization  $S^\t \mu_0$ (in direction $e_1 =
(1,0,\cdots,0)$)
as follows. For any $x_1 \in \mathbb{R}, x'\in \mathbb{R}^{d-1}, h>0$, let
\begin{equation*}
S^\t\mu_0(x_1, x') := \int_0^\infty \chi_{M^\t(U_{x'}^h)}(x_1) dh,
\end{equation*}
where $U_{x'}^h$ is defined in \eqref{U}.
\end{definition}

\begin{figure}[ht!]
\begin{center}
\includegraphics[scale=.9]{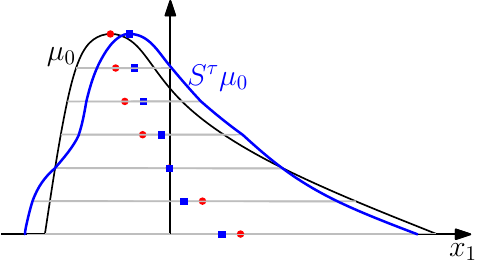}
\caption{Illustrations of $\mu_0$ and $S^\t\mu_0$ (for a small $\t>0$). \label{fig:st}}
\end{center}
\end{figure}

For an illustration of $S^\t \mu_0$ for  $\mu_0\in L^1(\mathbb{R})$, see Figure~\ref{fig:st}.

Using the above definition, Lemma \eqref{lem:prop_steiner}  and the representation \eqref{Steinerrepres} one immediately has
\[
S^{0}\mu_{0}=\mu_{0}, \quad S^{\infty}\mu_{0}=S^{}\mu_{0}.
\]
Furthermore, it is easy to check that $S^\t\mu_0 = \mu_0$ for all $\t$ if and only if $\mu_0$ is symmetric decreasing about the
hyperplane
$H=\{x_1=0\}$. Below is the definition for a function being symmetric decreasing about a hyperplane:
\begin{definition}
Let $\mu_0 \in L^1_+(\mathbb{R}^d)$. For a hyperplane $H \subset \mathbb{R}^d$ (with normal vector $e$), we say $\mu_0$ is \emph{symmetric decreasing}
about $H$
if for any $x\in H$, the function $f(\t):=\mu_0(x+\t e)$ is rearranged, \emph{i.e.} if $f=f^{\#}$.
\end{definition}
Next we state some basic properties of $S^\t$ without proof, see \cite{Brock,Ka1,Ka2} for instance.
\begin{lemma}
\label{lem:Lp}
The continuous Steiner symmetrization  $S^\t \mu_0$ in Definition {\rm\ref{def:steiner_func}} has the following properties:
\begin{enumerate}[(a)]
\item For any $h>0$, $|\{S^\t \mu_0> h\}| = |\{\mu_0>h\}|$. As a result, $\|S^\t \mu_0\|_{L^p(\R^d)} = \|\mu_0\|_{L^p(\R^d)}$ for all $1\leq p\leq +\infty$.
\item $S^\t$  has the semigroup property, that is, $S^{\t+s}\mu_0 = S^\t(S^s\mu_0)$ for any $\t,s\geq 0$ and non-negative $\mu_0 \in L^1(\mathbb{R}^d)$.
\end{enumerate}
\end{lemma}

Lemma \ref{lem:Lp} immediately implies that $\mathcal{S}[S^\t\mu_0]$ is constant in $\t$, where $\mathcal{S}[\cdot]$ is as given in \eqref{eq:energy}.

\subsubsection{Interaction energy under Steiner symmetrization}
 In this subsection, we will investigate $\mathcal{I}[S^\t\mu_0]$. It has been shown in \cite[Corollary 2]{Brock} and \cite[Theorem 3.7]{LL}
 that  $\mathcal{I}[S^\t\mu_0]$
 is non-increasing in $\t$. Indeed, in the case that $\mu_0$ is a characteristic function $\chi_{\Omega_0}$, it is shown in
\cite{Morgan} that $\mathcal{I}[S^\t\mu_0]$ is strictly decreasing for $\t$ small enough if $\Omega_0$ is not a ball. However, in order
to obtain \eqref{E_mu_t} for a strictly positive $c_0$, some refined estimates are needed, and we will prove the following:

\begin{proposition}
\label{prop:int}
Let $\mu_0 \in \mathcal{C}(\mathbb{R}^d) \cap L^1_+(\mathbb{R}^d)$. Assume the hyperplane $H=\{x_1=0\}$ splits the mass of $\mu_0$ into half and half, and
$\mu_0$
is
\textbf{not} symmetric decreasing about $H$. Let $\mathcal{I}[\cdot]$  be given in \eqref{eq:energy}, where $\K$ satisfies the assumptions $(K1)$-$(K3)$.
Then
$\mathcal{I}[S^\t\mu_0]$ is non-increasing in $\t$, and there exists some $\delta_0>0$ (depending on $\mu_0$) and $c_0>0$ (depending on $\mu_0$ and $\K$),
such
that
\begin{equation*}
\mathcal{I}[S^\t \mu_0] \leq \mathcal{I}[\mu_0] - c_0 \t \quad\text{ for all } \t\in [0,\delta_0].
\end{equation*}
\end{proposition}

The building blocks to prove Proposition \ref{prop:int} are a couple of lemmas estimating how the interaction energy between two \emph{one-dimensional}
densities
$\mu_1, \mu_2$ changes under continuous Steiner symmetrization for each of them. That is, we will investigate how
\begin{equation} \label{def:I}
I_\WW[\mu_1,\mu_2](\t) := \int_{\mathbb{R}\times \mathbb{R}} (S^\t \mu_1)(x) (S^\t\mu_2)(y) \WW(x-y) dxdy
\end{equation}
changes in $\t$ for a given one dimensional kernel $\WW$ to be determined. We start with the basic case where $\mu_1, \mu_2$ are both
characteristic
functions of some open interval.

\begin{lemma}
\label{lem1}
Assume $\WW(x) \in \mathcal{C}^1(\mathbb{R})$ is an even function with $\WW'(x)<0$ for all $x>0$. For $i=1,2$, let $\mu_i := \chi_{I(c_i,r_i)}$ respectively,
 where $I(c,r)$ is as given in Definition \ref{def:steiner_set}.
Then the following holds for the function $I(\t)
:= I_\WW[\mu_1, \mu_2](\t)$ introduced in \eqref{def:I}:
\begin{enumerate}[(a)]
\item $\frac{d^+}{d \t} I(0) \geq 0$. (Here $\frac{d^+}{d\tau }$ stands for the the right derivative.)
\item If in addition $\sgn c_1 \neq \sgn c_2$, then
\begin{equation}\label{eq_diff_sign}
\frac{d^+}{d \t} I(0) \geq  c_w  \min\{r_1, r_2\} |c_2-c_1|   > 0,
\end{equation}
where  $c_w$ is the minimum of $|\WW'(r)|$ for $r\in [\frac{|c_2-c_1|}{2}, r_1+r_2 + |c_2-c_1|]$.
\end{enumerate}
\end{lemma}

\begin{proof}
By definition of $S^\t$, we have $S^\t \mu_i = \chi_{M^\t(I(c_i, r_i))}$ for $i=1,2$ and all $\t\geq 0$. If $\sgn c_1 = \sgn c_2$,  the two intervals
$M^\t(I(c_i,
r_i))$ are moving towards the same direction for small enough $\t$, during which their interaction energy $I(\t)$ remains constant, implying $\frac{d}{d \t}
I(0)=0$.
Hence
it suffices to focus on $\sgn c_1 \neq \sgn c_2$ and prove \eqref{eq_diff_sign}.

Without loss of generality, we assume that $c_2>c_1$, so that $\sgn c_2- \sgn c_1$ is either 2 or 1. The definition of $M^\t$ gives
\begin{equation*}
\begin{split}
I(\t) &= \int_{-r_1+c_1- \t\sgn c_1}^{r_1 + c_1 - \t\sgn c_1} \int_{-r_2+c_2- \t\sgn c_2}^{r_2 + c_2 - \t\sgn c_2}\WW(x-y)dydx\\
&= \int_{-r_1}^{r_1} \int_{-r_2}^{r_2}\WW(x-y + (c_1 - c_2) + \t(\sgn c_2 - \sgn c_1))dydx.
\end{split}
\end{equation*}
Taking its right derivative in $\t$ yields
\begin{equation*}
\begin{split}
\frac{d^+}{d\tau} I(0) &= (\sgn c_2 - \sgn c_1) \int_{-r_1}^{r_1} \int_{-r_2}^{r_2}\WW'(x-y + (c_1 - c_2) )dydx.
\end{split}
\end{equation*}
Let us deal with the case $r_1\leq r_2$ first. In this case we rewrite $\frac{d^+}{d\tau} I(0)$ as
\begin{equation}\label{eq:I_0}
\frac{d^+}{d\tau} I(0) = (\sgn c_2 - \sgn c_1) \int_{ Q} \WW'(x-y)dxdy,
\end{equation}
 where $Q$ is the rectangle $[-r_1,r_1]\times [-r_2+(c_2-c_1), r_2+(c_2-c_1)]$, as illustrated in Figure~\ref{fig:Q}. Let $ Q^- =  Q \cap \{x-y>0\}$,
and
 $Q^+ =
 Q \cap \{x-y<0\}$. The assumptions on $\WW$ imply $\WW'(x-y)<0$ in $Q^-$, and $\WW'(x-y)>0$ in $Q^+$.

\begin{figure}[ht!]
\begin{center}
\includegraphics[scale=1.1]{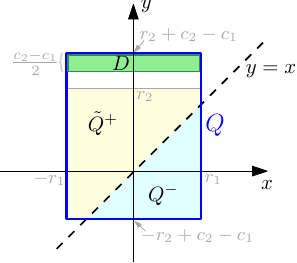}
\caption{Illustration of the sets $Q,  Q^-, \tilde Q^+$ and $D$ in the proof of Lemma \ref{lem1}.\label{fig:Q}}
\end{center}
\end{figure}

Let $\tilde Q^+ :=  Q^+ \cap \{y\leq r_2\}$, and $D:= [-r_1, r_1]\times [r_2 + \frac{c_2-c_1}{2}, r_2+(c_2-c_1)]$.
($\tilde Q^+$ and $D$ are the yellow set
and
green set in Figure \ref{fig:Q} respectively). By definition, $\tilde Q^+$ and $D$ are disjoint subsets of $Q^+$, so
\begin{equation}\label{eq:temp0}
\frac{d^+}{d\tau} I(0)\geq \! (\sgn c_2 - \sgn c_1) \!\Bigg(\!\underbrace{\int_{Q^-}\!\!\!\!\! \WW'(x-y)dxdy}_{\leq 0} + \underbrace{\int_{\tilde Q^+}\!\!\!\!\!
\WW'(x-y)dxdy}_{\geq 0} + \underbrace{\int_{D}\!\! \WW'(x-y)dxdy}_{> 0}\!\Bigg).
\end{equation}
We claim that $\int_{Q^-} \WW'(x-y)dxdy + \int_{\tilde Q^+} \WW'(x-y)dxdy \geq 0$. To see this, note that $ Q^- \cup \tilde Q^+ $ forms a
rectangle, whose
center
has a zero $x$-coordinate and a positive $y$-coordinate. Hence for any $h>0$, the line segment $\tilde Q^+ \cup\{x-y = -h\}$ is longer
than $Q^-\cup\{x-y =
h\}$,
which gives the claim.

Therefore, \eqref{eq:temp0} becomes
\[\frac{d^+}{d\tau} I(0) \geq (\sgn c_2 - \sgn c_1) \int_{D}\WW'(x-y)dxdy \geq \int_{D}\WW'(x-y)dxdy \geq |D| \min_{(x,y)\in D}\WW'(x-y)\]
Note that $D$ is a rectangle with area $r_1(c_2-c_1)$, and for any $(x,y)\in D$, we have (recall that $r_{2}>r_{1}$)
\[
\frac{|c_2-c_1|}{2}+r_{2}-r_{1}\leq y-x \leq r_1+r_2 + |c_2-c_1|.
\]
This finally gives
\[\frac{d^+}{d\tau} I(0)\geq r_1 (c_2-c_1) \min_{r\in [\frac{|c_2-c_1|}{2}, r_1+r_2 + |c_2-c_1|]} |\WW'(r)|.\]
Similarly, if $r_1>r_2$, then $I'(0)$ can be written as \eqref{eq:I_0} with $\tilde Q$ defined as $[-r_1+(c_2-c_1),r_1+(c_2-c_1)]\times [-r_2, r_2]$
instead, and
the above inequality would hold with the roles of $r_1$ and $r_2$ interchanged. Combining these two cases, we have
\[
\frac{d^+}{d\tau} I(0) \geq c_w \min\{r_1, r_2\} |c_2-c_1| \quad\text{ for }\sgn c_1 \neq \sgn c_2,
\]
where $c_w$ is the minimum of $|\WW'(r)|$ for $r\in [\frac{|c_2-c_1|}{2}, r_1+r_2 + |c_2-c_1|]$.
\end{proof}

The next lemma generalizes the above result to open sets with finite measures.

\begin{lemma}
\label{lem:dIopenset}
Assume $\WW(x) \in \mathcal{C}^1(\mathbb{R})$ is an even function with $\WW'(r)<0$ for all $r>0$.
For open sets $U_1, U_2 \subset \mathbb{R}$ with finite measure, let $\mu_i := \chi_{U_i}$ for $i=1,2$, and $I(\t) := I_\WW[\mu_1,\mu_2](\t)$ is as defined
in
\eqref{def:I}. Then
\begin{enumerate}[(a)]
\item $\frac{d}{d \t}I(\t)\geq 0$ for all $\t\geq 0$;
\item In addition, assume that there exists some $a\in(0,1)$ and $R>\max\{|U_1|, |U_2|\}$ such that $|U_1 \cap (\frac{|U_1|}{2}, R)|>a$, and $|U_2 \cap
    (-R,
    -\frac{|U_2|}{2})|>a$.
Then for all $\t\in [0,a/4]$, we have
\begin{equation}\label{ineqI_2}
\frac{d^+}{d \t} I(\t) \geq \frac{1}{128} c_w a^3   > 0,
\end{equation}
where $c_w$ is the minimum of $|\WW'(r)|$ for $r\in [\frac{a}{4}, 4R]$.
\end{enumerate}
\end{lemma}

\begin{proof}
It suffices to focus on the case when $U_1, U_2$ both consist of a finite disjoint union of open intervals, and for the general case we can take the
limit.
Recall that $S^\t\mu_i = \chi_{M^\t(U_i)}$ for $i=1,2$ and all $\t\geq 0$.

To show (a), due to the semigroup property of $S^\t$ in Lemma \ref{lem:Lp}, all we need to show is $\frac{d^+}{d \t}I(0)\geq 0$. By writing $U_1, U_2$
each as a
union
of disjoint open intervals and expressing $I(\t)$ a sum of the pairwise interaction energy, (a) immediately follows from Lemma~\ref{lem1}(a).

We will prove (b) next. First, we claim that
\begin{equation}\label{eq:u1}
A_1(\t) := \left|M^\t(U_1) \cap  \left(\frac{|U_1|}{2}+\frac{a}{4}, \,R\right)\right|>\frac{a}{4}  \quad\text{ for all $\t\in[0,\frac{a}{4}]$}.
\end{equation}
To see this, note that $A_1(0)>\frac{3a}{4}$ due to the assumption $|U_1 \cap (\frac{|U_1|}{2}, R)|>a$. Since each interval in $M^\t(U_1)$
moves with speed
either
0 or $\pm 1$ at each $\t$, we know $A_1'(\t)\geq -2$ for all $\t$, yielding the claim. (Similarly, $A_2(\t) := |M^\t(U_2) \cap  (-R,
-\frac{|U_2|}{2}-\frac{a}{4})|>\frac{a}{4}$ for all $\t\in[0,\frac{a}{4}]$.)

Now we pick any $\tau_0 \in [0,\frac{a}{4}]$, and we aim to prove \eqref{ineqI_2} at this particular time $\tau_0$. At $\t=\tau_0$, write $M^{\tau_0}(U_1):=
\cup_{k=1}^{N_1} I(c_k^1, r_k^1)$, where all intervals $I(c_k^1, r_k^1)$ are disjoint, and none of them share common endpoints -- if they do, we merge them into one interval.

Note that for every $x\in M^{\tau_0}(U_1) \cap  (\frac{|U_1|}{2}+\frac{a}{4}, \,R)$, $x$ must belong to some $I(c_k^1, r_k^1)$ with $a/4\leq c_k^1 \leq R+|U_1|/2$. Otherwise, the length of $I(c_k^1, r_k^1)$ would exceed $|U_1|$, contradicting Lemma \ref{lem:prop_steiner}(a). We then define
\[
\mathscr{I}_1:= \left\{1\leq k\leq N_1:  \frac{a}{4}\leq c_k^1 \leq R+|U_1|/2\right\}.
\]
Combining the above discussion with \eqref{eq:u1}, we have $\sum_{k\in \mathscr{I}_1}|I(c_k^1, r_k^1)| \geq a/4$, i.e.
\begin{equation}\label{eq:sumk}
\sum_{k\in \mathscr{I}_1} r_k^1 \geq \frac{a}{8}.
\end{equation}
Likewise,  let $M^{\tau_0}(U_2) := \cup_{k=1}^{N_2} I(c_k^2, r_k^2)$, and denote by $\mathscr{I}_2$ the set of indices $k$ such that\\ $-R-|U_2|/2\leq c_k^2
\leq
-\frac{a}{4}$, and similarly we have $\sum_{k\in \mathscr{I}_2} r_k^2 \geq a/8$.

The semigroup property of $M^\t$ in Lemma~\ref{lem:prop_steiner} gives that for all $s>0$,
\[
M^{\tau_0+s}(U_1) = M^s(M^{\tau_0} (U_1)) = M^s( \cup_{k=1}^{N_1} I(c_k^1, r_k^1)).
\]
Since none of the intervals $I(c_k^1, r_k^1)$ share common endpoints, we have
\[
M^s( \cup_{k=1}^{N_1} I(c_k^1, r_k^1))= \cup_{k=1}^{N_1} M^s(I(c_k^1, r_k^1)) \quad\text{ for sufficiently small }s> 0.
\]
A similar result holds for $M^{\tau_0+s}(U_2)$, hence we obtain for sufficiently small $s> 0$:
\[
I(\tau_0+s) = I_\WW[\chi_{M^{\tau_0}(U_1)}, \chi_{M^{\tau_0}(U_2)}](s)= \sum_{k=1}^{N_1} \sum_{l=1}^{N_2} I_\WW[\chi_{I(c_k^1, r_k^1)}, \chi_{I(c_l^2, r_l^2)}](s).
\]
Applying Lemma \ref{lem1}(a) to the above identity yields
\begin{equation}\label{dI_tau}
\begin{split}
\frac{d^+}{d\tau} I(\tau_0) &=  \sum_{k=1}^{N_1} \sum_{l=1}^{N_2} \frac{d}{ds}I_\WW[\chi_{I(c_k^1, r_k^1)}, \chi_{I(c_l^2, r_l^2)}]\Big|_{s=0}\geq
\sum_{k\in\mathscr{I}_1}
\sum_{l\in\mathscr{I}_2} \underbrace{\frac{d}{ds}I_\WW[\chi_{I(c_k^1, r_k^1)}, \chi_{I(c_l^2, r_l^2)}]\Big|_{s=0}}_{=: T_{kl}}.
\end{split}
\end{equation}
Next we will obtain a lower bound for $T_{kl}$. By definition of  $\mathscr{I}_1$ and $ \mathscr{I}_2$, for each $k\in \mathscr{I}_1$ and $l\in
\mathscr{I}_2$ we
have that $c_k^1 \geq \frac{a}{4}$ and $c_l^2 \leq -\frac{a}{4}$, hence $|c_l^2 - c_k^1| \geq \frac{a}{2}$. Thus Lemma \ref{lem1}(b) yields
\[
T_{kl} \geq c_w \min\{r_k^1, r_l^2\} |c_l^2 - c_k^1| \geq c_w \frac{a}{2} \min\{r_k^1, r_l^2\} \quad\text{ for }k\in \mathscr{I}_1, l\in \mathscr{I}_2,
\]
where $c_w = \min_{r\in[\frac{a}{4}, 4R]} |\WW'(r)|$ (here we used that for $k\in \mathscr{I}_1, l\in \mathscr{I}_2$, we have $r_k^1+r_l^2 +
|c_l^2-c_k^1|
\leq
|U_1|/2+|U_2|/2 + (R+|U_1|/2) + (R+|U_2|/2) \leq 4R$, due to the assumption $R>\max\{|U_1|, |U_2|\}$.)

Plugging the above inequality into \eqref{dI_tau} and using $\min\{u,v\} \geq \min\{u,1\} \min\{v,1\}$ for $u,v>0$, we have
\[
\begin{split}
\frac{d^+}{d\tau} I(\tau_0) &\geq \frac{ac_w}{2}  \sum_{k\in\mathscr{I}_1} \sum_{l\in\mathscr{I}_2}  \min\{r_k^1, 1\}  \min\{r_l^2, 1\}\\
&=\frac{ac_w}{2}  \left( \sum_{k\in\mathscr{I}_1}  \min\{r_k^1, 1\} \right) \left(\sum_{l\in\mathscr{I}_2}   \min\{r_l^2, 1\} \right)\\
&\geq \frac{ac_w}{2}   \min\left\{1, \sum_{k\in\mathscr{I}_1} r_k^1\right\}  \min\left\{1, \sum_{l\in\mathscr{I}_2} r_l^2\right\}\\
&\geq \frac{ac_w}{2} \min\left\{1, \frac{a}{8} \right\}^2 \geq  \frac{1}{128} c_w a^3 ,
\end{split}
\]
here we applied \eqref{eq:sumk} in the second-to-last inequality, and used the assumption $a\in(0, 1)$ for the last inequality. Since $\tau_0 \in[0,a/4]$ is
arbitrary, we can conclude.
\end{proof}

Now we are ready to prove Proposition \ref{prop:int}.

\begin{proof}[Proof of  Proposition {\rm\ref{prop:int}}]
Since $\mu_0 \in \mathcal{C}(\mathbb{R}^d) \cap L^1_+(\mathbb{R}^d)$ is \emph{not} symmetric decreasing about $H = \{x_1 = 0\}$, we know that there exists some $x' \in \mathbb{R}^{d-1}$ and $h>0$,
such that
$U_{x'}^h := \{x_1\in\mathbb{R}: \mu_0(x_1, x')>h\}$ has finite measure, and its difference from $(-|U_{x'}^h|/2, |U_{x'}^h|/2)$ has nonzero measure.

For $R>0, a>0$, define
\[
B_1^{R,a} = \left\{(x',h)\in\mathbb{R}^{d-1}\times(0,+\infty): \left|U_{x'}^h \cap (|U_{x'}^h|/2, R)\right|>a, |x'|\leq R\right\},
\]
\[
B_2^{R,a} = \left\{(x',h)\in\mathbb{R}^{d-1}\times(0,+\infty): \left|U_{x'}^h \cap (-R, -|U_{x'}^h|/2)\right|>a, |x'|\leq R\right\}.
\]
Our discussion above yields that at least one of $B_1^{R,a}$ and $B_2^{R,a}$ is nonempty when $R$ is sufficiently large and $a>0$ sufficiently
small (hence
at
least one of them must have nonzero measure by continuity of $\mu_0$). Next let us discuss two cases.

Case 1: Both $B_1^{R,a}$ and $B_2^{R,a}$ have nonzero measure when $R$ is sufficiently large and $a>0$ sufficiently
small. 

Let us define a one-dimensional kernel $K_l(r) := -\tfrac12\K(\sqrt{r^2+l^2})$. Note that for any $l>0$, the kernel $K_l \in \mathcal{C}^1(\mathbb{R})$ is even
in $r$,
and
$K_l'(r)<0$ for all $r>0$. By definition of $S^\t$, we can rewrite $\mathcal{I}[S^\t\mu_0]$ as
\[
\mathcal{I}[S^\t\mu_0] =\!-\!\!\int_{(\mathbb{R}^+)^2}\!\int_{\mathbb{R}^{2(d-1)}}\!  \int_{\mathbb{R}^2} \!\!\chi_{ M^\t(U_{x'}^{h_1}) }(x_1) \chi_{
M^\t(U_{y'}^{h_2}) }(y_1)  \,K_{|x'-y'|}(|x_1-y_1|)  dx_1 dy_1 dx' dy' dh_1dh_2.
\]
Thus using the notation in \eqref{def:I}, $\mathcal{I}[S^\t\mu_0] $ can be rewritten as
\begin{equation}\label{I_temp00}
\mathcal{I}[S^\t\mu_0] =-\int_{(\mathbb{R}^+)^2}\int_{\mathbb{R}^{2(d-1)}}  I_{K_{|x'-y'|}}[\chi_{U_{x'}^{h_1}},\chi_{U_{y'}^{h_2}}](\t)  ~dx' dy' dh_1dh_2,
\end{equation}
and taking its right derivative (and applying Lemma \ref{lem:dIopenset}(a)) yields
\begin{equation}\label{eq:dt_int}
\begin{split}
-\frac{d^+}{d \t}\mathcal{I}[S^\t\mu_0]   \geq \int_{(x',h_1)\in B_1^{R,a}}\int_{(y',h_2)\in B_2^{R,a}} \frac{d}{d \t}
I_{K_{|x'-y'|}}[\chi_{U_{x'}^{h_1}},\chi_{U_{y'}^{h_2}}](\t)~  dy'dh_2 dx'dh_1.
\end{split}
\end{equation}
By definition of $B_1^{R,a}$ and $B_2^{R,a}$, for any $(x',h_1)\in B_1^{R,a}$ and $(y',h_2)\in B_2^{R,a}$, we can apply
Lemma~\ref{lem:dIopenset}(b) to
obtain
\begin{equation}\label{eq:integrand}
\begin{split}
\frac{d^+}{d \t} I_{K_{|x'-y'|}}[\chi_{U_{x'}^{h_1}},\chi_{U_{y'}^{h_2}}](\t) \geq \frac{1}{128}c_w a^3 \quad\text{ for any }\t\in[0,a/4],
\end{split}
\end{equation}
where $c_w$ is the minimum of $|K_{|x'-y'|}'(r)|$ in $[a/4, 4R]$. By definition of $K_l(r)$, we have
\[
K_{|x'-y'|}'(r) = - \tfrac12\K'(\sqrt{r^2+|x'-y'|^2}) \frac{r}{\sqrt{r^2+|x'-y'|^2}}.
\]
Using $|x'|\leq R$ and $|y'|\leq R$ (due to definition of $B_1, B_2$), we have $ \frac{r}{\sqrt{r^2+|x'-y'|^2}} \geq \frac{a}{20R}$ for all
$r\in[a/4,4R]$,
hence
$c_w \geq \frac{a}{40R} \min_{r\in[\frac{a}{4}, 4R]}\K'(r)$.

Plugging \eqref{eq:integrand} (with the above $c_w$) into \eqref{eq:dt_int} finally yields
\[
-\frac{d^+}{d \t}\mathcal{I}[S^\t\mu_0] \geq \frac{1}{6000} |B_1^{R,a}| |B_2^{R,a}|  \min_{r\in[\frac{a}{4}, 4R]}\K'(r) a^4>0 \quad\text{for all
$\t\in[0,a/4]$},
\]
hence we can conclude the desired estimate.

Case 2: Only one of $B_1^{R,a}$ and $B_2^{R,a}$ has nonzero measure for $R\gg 1$ and $0<a\ll 1$. (WLOG assume $B_1^{R,a}$ satisfies this property.) Since $|B_2^{R,a}|=0$ for all $R>0$, $a>0$, it implies $U_{x'}^h \subset (-|U_{x'}^h|/2, +\infty)$ for almost every $x'\in\mathbb{R}^{d-1}$ and $h>0$. Thus using the layer cake representation formula 
\eqref{Steinerrepres}, we have
 $\mu_0\leq S\mu_0$ in $\{x_1<0\}$, where $S\mu_0$ is the Steiner symmetrization of $\mu_0$. 
On the other hand, using the assumption that $H$ splits the mass of $\mu_0$ into half and half, $\mu_0$ and $S\mu_0$ must have the same mass in $\{x_1<0\}$, implying $\mu_0= S\mu_0$ in $\{x_1<0\}$, i.e. 
\begin{equation}
\label{temp01}(-|U_{x'}^h|/2, 0)\subset U_{x'}^h \quad\text{for all }x'\in\mathbb{R}^{d-1}, h>0.
\end{equation} Combining this with $|B_1^{R,a}|>0$, some $U_{x_0'}^{h_0}$ must contain disjoint intervals with a positive gap. By the continuity of $\mu_0$, there exists some $0\leq x_l < x_r$ and some sufficiently small $a>0$, such that
\[
E^a:= \{(x',h)\in\mathbb{R}^{d-1}\times(0,+\infty): U_{x'}^h \cap (x_l, x_r) = \emptyset, |U_{x'}^h \cap (-\infty, x_l)|>2a, |U_{x'}^h \cap (x_r,+\infty)| >2a\}
\]
has a nonzero measure. For $(x', h)\in E^a$, we define
$
V_{x'}^{h} := U_{x'}^h \cap (-\infty, x_l),  W_{x'}^{h} := U_{x'}^h \cap (x_r, +\infty).
$
By \eqref{temp01}, $(-|U_{x'}^h|/2, 0) \subset V_{x'}^h$, and the definition of $E^a$ gives $|U_{x'}^h|=|V_{x'}^h|+|W_{x'}^h| > |V_{x'}^h|+2a$. Thus
$
|V_{x'}^{h}  \cap (-\infty, -|V_{x'}^h|/2) | > a
$ for all $(x',h)\in E^a$,
implying 
\[
\tilde B_2^{R,a} := \left\{(x',h)\in E^a: \left|V_{x'}^h \cap (-R, -|V_{x'}^h|/2)\right|>a, |x'|\leq R\right\}
\]
has a nonzero measure for the above $a>0$, and for $R>0$ sufficiently large.

Let us denote $\tau_0 := \frac{1}{2}(x_r-x_l)>0$. For $(x',h)\in E^a$, since $V_{x'}^{h} $ and $W_{x'}^{h} $ has at least a $2\tau_0$ gap between them, 
$
M^\tau(V_{x'}^{h} )$ and $M^\tau(W_{x'}^{h} )$  remains disjoint for $0<\tau<\tau_0$. Thus for $(x',h)\in E^a$, 
\[
M^\tau(U_{x'}^{h}) = M^\tau(V_{x'}^{h} ) \dot\cup M^\tau(W_{x'}^{h} ) \quad\text{ for }0<\tau<\tau_0,
\]
where $\dot\cup$ represents the disjoint union.
Now for all $0<\tau<\tau_0$, we are ready to take the right derivative of \eqref{I_temp00} (and applying Lemma \ref{lem:dIopenset}(a)) to obtain
\begin{equation}
\begin{split}
-\frac{d^+}{d \t}\mathcal{I}[S^\t\mu_0]   
&\geq \int_{(x',h_1)\in B_1^{R,a}}\int_{(y',h_2)\in E^a} \frac{d}{d \t}
I_{K_{|x'-y'|}}[\chi_{U_{x'}^{h_1}},\,\chi_{V_{y'}^{h_2}} + \chi_{W_{y'}^{h_2}}](\t)~  dy'dh_2 dx'dh_1\\
&\geq \int_{(x',h_1)\in  B_1^{R,a}}\int_{(y',h_2)\in \tilde B_2^{R,a}} \frac{d}{d \t}
I_{K_{|x'-y'|}}[\chi_{U_{x'}^{h_1}},\chi_{V_{y'}^{h_2}}](\t)~  dy'dh_2 dx'dh_1.
\end{split}
\end{equation}
Since $| B_1^{R,a}|>0$ and $|\tilde B_2^{R,a}|>0$, the rest of the argument is identical to the last part of Case 1, and at the end we obtain
\[
-\frac{d^+}{d \t}\mathcal{I}[S^\t\mu_0] \geq \frac{1}{6000} |B_1^{R,a}| |\tilde B_2^{R,a}|  \min_{r\in[\frac{a}{4}, 4R]}\K'(r) a^4>0 \quad\text{for all
$\t\in[0,\min\{a/4,\tau_0\}]$},
\]
finishing the proof for Case 2.
\end{proof}

\subsubsection{Proof of Proposition {\rm\ref{prop:steiner}}}

In the statement of Proposition \ref{prop:steiner}, we assume that $\mu_0$ is not radially decreasing up to any translation. Since Steiner symmetrization only deals with symmetrizing in one direction, we will use the following simple lemma linking radial symmetry with being symmetric decreasing about hyperplanes. Although the result is standard (see \cite[Lemma 1.8]{Fraenkel}), for the sake of completeness we include here the details of the proof.

\begin{lemma}
\label{lem:steiner_radial}
Let $\mu_0 \in \mathcal{C}(\mathbb{R}^d)$. Suppose for every unit vector $e$, there exists a hyperplane $H \subset \mathbb{R}^d$ with normal vector $e$, such
that
$\mu_0$ is symmetric decreasing about $H$. Then $\mu_0$ must be radially decreasing  up to a translation.
\end{lemma}

\begin{proof}
For $i=1,\dots,d$, let $e_i$ be the unit vector with $i$-th coordinate 1 and all the other coordinates 0. By assumption, for each $i$,
there exists some
hyperplane $H_i$ with normal vector $e_i$, such that $\mu_0$ is symmetric decreasing about $H_i$. We then represent each $H_i$ as
$\{(x_1,\dots,x_d): x_i =
a_i$\}
for some $a_i\in\mathbb{R}$, and then define $a\in\mathbb{R}^d$ as $a:=(a_1,\dots, a_d)$. Our goal is to prove that $\mu_0(\cdot-a)$ is radially
decreasing.

We first claim that $\mu_0(x) = \mu_0(2a-x)$ for all $x\in\mathbb{R}^d$. For any hyperplane $H\subset \mathbb{R}^d$, let $T_H: \mathbb{R}^d\to\mathbb{R}^d$
be
the reflection about the hyperplane $H$. Since $\mu_0$ is symmetric with respect to $H_1, \dots, H_d$, we have $\mu_0(x) = \mu_0(T_{H_i}x)$ for
$x\in\mathbb{R}^d$
and all $i=1,\dots,d$, thus $\mu_0(x) = \mu_0(T_{H_1}\dots T_{H_d} x) = \mu_0(2a-x)$.

The claim implies that every hyperplane $H$ passing through $a$ must split the mass of $\mu_0$ into half and half. Denote the normal vector of $H$ by $e$.
By assumption, $\mu_0$ is symmetric decreasing about some hyperplane $H'$ with normal vector $e$. The definition of symmetric decreasing implies that $H'$ is
the only hyperplane with normal vector $e$ that splits the mass into half and half, hence $H'$ must coincide with $H$. Thus $\mu_0$ is symmetric decreasing
about every hyperplane passing through $a$, hence we can conclude.
\end{proof}

\begin{proof}[Proof of Proposition {\rm\ref{prop:steiner}}]
Since $\mu_0$ is not radially decreasing up to any translation, by Lemma \ref{lem:steiner_radial}, there exists some unit vector $e$, such that $\mu_0$ is
not
symmetric decreasing about any hyperplane with normal vector $e$. In particular, there is a hyperplane $H$ with normal vector $e$ that splits the mass of
$\mu_0$
into half and half, and $\mu_0$ is not symmetric decreasing about $H$. We set $e=(1,0,\dots,0)$ and $H = \{x_1=0\}$ throughout the proof without loss of
generality. For the rest of the proof, we will discuss two different cases $m\in (0,1]$ and $m>1$, and construct $\mu(\t, \cdot)$ in different ways.

\noindent\textbf{Case 1:} $m \in (0,1]$. In this case, we simply set $\mu(\t,\cdot) = S^\t \mu_0$. By Proposition~\ref{prop:int}, $\mathcal{I}[S^\t\mu_0]$ is decreasing at least linearly for a short time. Since continuous Steiner symmetrization preserves the distribution function, even if $\mathcal{S}[\mu_0] = -\infty$ by itself, we still have the difference $\mathcal{S}[\mu(\t)] - \mathcal{S}[\mu_0] \equiv 0$ in the sense of \eqref{def_reg_s}. Thus \eqref{E_mu_t} holds for all sufficiently small $\tau>0$. In addition, \eqref{int_0} is automatically satisfied since we assumed that $\supp \mu_0 = \mathbb{R}^d$ for $m\in (0,1]$, and recall that $S^\tau$ is mass-preserving by definition.

It then suffices to prove \eqref{mu_t} for all sufficiently small $\tau>0$. Let us discuss the case $m=1$ first. By assumption,  $|\nabla \log \mu_0| \leq C_0$. For any $y\in \mathbb{R}^d$ and $\tau>0$ we claim that
\begin{equation}\label{logmu}
 \log \mu_0(y) - C_0\tau \leq \log \mu(\tau,y) \leq \log \mu_0(y) + C_0\tau.
\end{equation}
To see this, let us fix any $y = (y_1, y')\in\mathbb{R}^d$. Since $\log \mu_0(\cdot, y')$ is Lipschitz with constant $C_0$, for any $\tau>0$, the following two inequalities hold:
\[
\dist(y_1, \{x_1\in \mathbb{R}: \log \mu_0(x_1, y') > \log \mu_0(y_{1},y^{\prime}) + C_0 \tau\}) \geq \tau
\]
and
\[
\dist(y_1, \{x_1\in \mathbb{R}: \log \mu_0(x_1, y') < \log \mu_0(y_{1},y^{\prime}) - C_0 \tau\}) \geq \tau.
\]
Since the level sets of $\mu_0$ are moving with velocity at most 1 (and note that any level set of $\mu_0$ is also a level set of $\log \mu_0$), we obtain \eqref{logmu}. It implies
\[
\mu_0(y) (e^{-C_0\tau }-1) \leq \mu(\tau,y) - \mu_0(y) \leq \mu_0(y) (e^{C_0\tau }-1).
\]
We then have $|\mu(\tau,y) - \mu_0(y)| \leq 2C_0\mu_0(y) \tau$ for all $\tau \in (0, \frac{\log 2}{C_0})$ and all $y\in \mathbb{R}^d$.

Now we move on to $m\in (0,1)$, where we aim to show that $|\mu(\tau,y) - \mu_0(y)| \leq C_1 \mu_0^{2-m}(y)\tau$ for some $C_1$ for all sufficiently small $\tau>0$. Using the assumption $|\nabla \frac{m}{1-m} \mu_0^{m-1}| \leq C_0$, the same argument to obtain \eqref{logmu} then gives the following for all $y\in \mathbb{R}^d, \tau>0$:
\[
\frac{m}{1-m} \mu_0^{m-1}(y) - C_0\tau \leq \frac{m}{1-m} \mu^{m-1}(\tau,y) \leq \frac{m}{1-m} \mu_0^{m-1}(y) + C_0\tau .
\]
Note that $\mu_0^{m-1}(y)\geq \|\mu_0\|_\infty^{m-1}$, since $\mu_0 \in L^\infty$ and $m\in (0,1)$. Let us set $\delta_0 = \frac{m}{2(1-m)C_0}\|\mu_0\|_\infty^{m-1}$. For any $\tau\in (0,\delta_0)$, the left hand side of the above inequality is strictly positive, thus we have
\begin{equation}
\left( \mu_0^{m-1}(y) + \frac{C_0(1-m)}{m}\tau\right)^{\tfrac{1}{m-1}}\leq \mu(\tau,y) \leq \left( \mu_0^{m-1}(y) - \frac{C_0(1-m)}{m}\tau\right)^{\tfrac{1}{m-1}},\label{eqm<1}
\end{equation}
and note that our choice of $\delta_0$ ensures that $$\mu_0^{m-1}(y) - \frac{C_0(1-m)}{m}\tau \geq\mu_0^{m-1}(y)- \frac{1}{2}\|\mu_0\|_\infty^{m-1} \geq \frac{1}{2}\mu_0^{m-1}(y)$$ for all $\tau\in (0,\delta_0)$.
Let $f(a) := \left( \mu_0^{m-1}(y) +a\right)^{\frac{1}{m-1}} - \mu_0(y)$, which is a convex and decreasing function in $a$ with $f(0)=0$. Using this function $f$, the above inequality \eqref{eqm<1} can be rewritten as
\[
f\left( \frac{C_0(1-m)}{m}\tau\right)\leq \mu(\tau,y) - \mu_0(y) \leq f\left( -\frac{C_0(1-m)}{m}\tau\right).\]
 Since $f$ is convex and decreasing, for all $|a| \leq \frac{C_0(1-m)}{m}\delta_0=\frac{1}{2}\|\mu_{0}\|_{\infty}^{m-1}$ we have
  $$|f'(a)| \leq \frac{1}{|m-1|} \left(\frac{1}{2}\mu_0^{m-1}(y)\right)^{\frac{2-m}{m-1}} = \frac{2^{\frac{m-2}{m-1}}}{|m-1|} \mu_0(y)^{2-m},$$ and this leads to $$|\mu(\tau,y) - \mu_0(y)| \leq C_1 \mu_0(y)^{2-m} \tau \quad \text{ for all }\tau \in (0,\delta_0) $$ with $C_1 := \frac{2^{\frac{m-2}{m-1}}}{m}C_{0}$, which gives \eqref{mu_t}.

\noindent\textbf{Case 2:} $m>1$.
Note that if we set $\mu(\t,\cdot) = S^\t \mu_0$, then it directly satisfies \eqref{E_mu_t} for a short time, since $\mathcal{I}[S^\t\mu_0]$ is decreasing at
least
linearly for a short time by Proposition~\ref{prop:int}, and we also have  $\mathcal{S}[S^\t\mu_0]$ is constant in $\t$. However, $S^\t\mu_0$ does not
satisfy
\eqref{mu_t} and \eqref{int_0}. To solve this problem, we will modify $S^\t \mu_0$ into $\tilde S^\t \mu_0$, where we make the set  $U_{x'}^h :=
\{x_1\in\mathbb{R}: \mu_0(x_1, x')>h\}$ travels at speed $v(h)$ rather than at constant speed 1, with $v(h)$ given by
\begin{equation}
\label{def:v}
v(h) := \begin{cases}
1 & h\geq h_0,\\
\left(\dfrac{h}{h_0}\right)^{m-1} & 0<h<h_0,
\end{cases}
\end{equation}
for some sufficiently small constant $h_0>0$ to be determined later. More precisely, we define $\mu(\t,\cdot) = \tilde S^\t \mu_0$ as
\begin{equation}\label{def_modified_s}
\tilde S^\t\mu_0(x_1, x') := \int_0^\infty \chi_{M^{v(h)\t}(U_{x'}^h)}(x_1) dh
\end{equation}
with $v(h)$ as in \eqref{def:v} For
an illustration on the difference between $S^\t\mu_0$ and $\tilde S^\t \mu_0$, see the left figure of Figure~\ref{fig:st2}.

\begin{figure}[ht!]
\begin{center}
\includegraphics[scale=0.8]{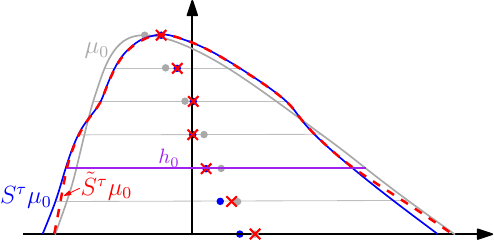} \quad\quad\quad \includegraphics[scale=0.8]{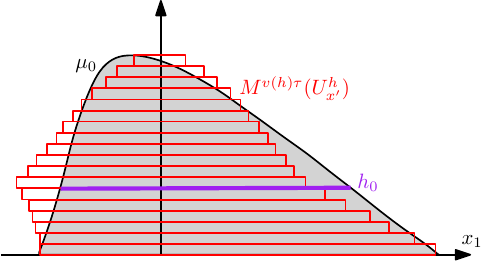}
\caption{Left: A sketch on $\mu_0$ (grey), $S^\t \mu_0$ (blue) and $\tilde S^\t\mu_0$ (red dashed) for a small $\t>0$.
Right: In the construction of $\tilde S^\t$, due to a reduced speed at lower values, a higher value level set may travel over a lower value level set. The figure illustrates this phenomenon for a large $\tau>0$.\label{fig:st2} }
\end{center}
\end{figure}

Note that $\tilde S^\t \mu_0$ and $S^\t \mu_0$ do not necessarily have the same distribution function. Due to a reduced speed $v(h)$ for $h\in(0,h_0)$ in the construction of $\tilde S^\t$, a higher block may travel over a lower block, as illustrated in the right figure of  Figure~\ref{fig:st2}. When this happens, the part that is hanging outside would ``drop down'' as we integrate in $h$ in \eqref{def_modified_s}, thus changing the distribution function of $\tilde S^\t\mu_0$. But, this is not likely (and even impossible) to happen when $\t\ll 1$: indeed, using the regularity assumption $|\nabla \mu_0^{m-1}|\leq C_0$ and the particular $v(h)$ in \eqref{def:v}, one can show that the level sets remain ordered for small enough $\t$. But we will not pursue in this direction, since later we will show in \eqref{eq:s_dec} that $\mathcal{S}[\tilde S^\t\mu_0] \leq \mathcal{S}[\mu_0]$ for all $\t>0$, which is sufficient for us.

Our goal is to show that such $\mu(\t,\cdot)$ satisfies \eqref{E_mu_t}, \eqref{mu_t} and \eqref{int_0} for small enough $\t$.
Let us first prove that for any $h_0>0$, $\mu(\t,\cdot)$ satisfies \eqref{mu_t} and \eqref{int_0} for $\t\in [0,\delta_1]$, where $\delta_1 =
\delta_1(m,h_0,C_0)>0$.
To show \eqref{int_0}, note that the assumption $|\nabla (\mu_0^{m-1})| \leq C_0$ directly leads to the following: for any $x,y\in\mathbb{R}^d$ with
$\mu_0(x)\geq
h>0$ and $\mu_0(y)=0$, we have that $|x-y| \geq h^{m-1}/C_0$.
This implies that
 for any connected component $D_i \subset \supp \mu_0$,
\begin{equation}
\label{eq:dist0}
\dist \left(\{\mu_0>h\} \cap D_i\,, \partial D_i\right) \geq \frac{h^{m-1}}{C_0} \quad \text{ for all } h>0.
\end{equation}
Now define $D_{i,x'}$ as the one-dimensional set $\{x_1\in\mathbb{R}: (x_1, x') \in D_i\}$. The inequality \eqref{eq:dist0} yields
\[
M^{v(h)\t} (U_{x'}^h \cap D_{i,x'}) \subset D_{i,x'} \quad \text{ for all } x'\in\mathbb{R}^{d-1}, \,h>0,\, \t\leq \dfrac{h^{m-1}}{C_0 v(h)},
\]
and note that for any $h>0$, we have $h^{m-1}/(C_0 v(h))\geq h_0^{m-1}/C_0$ by definition of $v(h)$. Using the above equation, the definition of $\tilde S^\t$ and the fact that $M^{v(h)\t}$ is measure-preserving, we have that \eqref{int_0} holds for all $\t\leq h_0^{m-1}/C_0$.

Next we prove \eqref{mu_t}. Let us fix any $y = (y_1, y')\in\mathbb{R}^d$, and denote $h=\mu_0(y)$. Using $|\nabla \mu_0^{m-1}| \leq C_0$, we
have that for
any
$\lambda>1$,
\[
\dist(y_1, U_{y'}^{\lambda h})\geq  \frac{(\lambda^{m-1}-1) h^{m-1}}{C_0}.
\]
So we have $y_1 \not\in M^{v(\lambda h)\t} \left(U_{y'}^{\lambda h}\right)$ for all $\t\leq \frac{(\lambda^{m-1}-1)h^{m-1}}{C_0 v(\lambda h)}$, which is
uniformly
bounded below by $ \frac{(\lambda^{m-1}-1)h_0^{m-1}}{C_0 \lambda^{m-1}}$ due to the fact that $v(\lambda h) \leq (\lambda h/h_0)^{m-1}$ for all $h$. By
definition
of $\tilde S^\t$ and the fact that $\mu_0(y)=h$, the following holds for all $\lambda > 1$:
\[
\tilde S^\t[\mu_0](y) \leq \lambda \mu_0(y) \quad\text{ for all } \t\leq \frac{(\lambda^{m-1}-1)h_0^{m-1}}{C_0 \lambda^{m-1}}.
\]
Note that there exists $c_m^1>0$ only depending on $m$, such that $\lambda^{m-1}-1 \geq c_m^1(\lambda - 1)$ for all $1<\lambda<2$. Hence for all
$1<\lambda<2$ we
have
\[
\tilde S^\t[\mu_0](y) - \mu_0(y) \leq (\lambda -1 ) \mu_0(y) \quad\text{ for } \t= \frac{c_m^1  h_0^{m-1}}{C_0 2^{m-1}}(\lambda-1),
\]
and this directly implies
\begin{equation}
\label{temp1}
\tilde S^\t[\mu_0](y) - \mu_0(y) \leq \frac{C_0 2^{m-1}}{c_m^1 h_0^{m-1} }\mu_0(y)\t \quad\text{ for all } \t\leq \frac{c_m^1 h_0^{m-1}}{C_0 2^{m-1}}.
\end{equation}

Similarly, for any $0<\eta < 1$ we have
$
\dist(y_1, (U_{y'}^{\eta h})^c)\geq  \frac{(1-\eta^{m-1}) h^{m-1}}{C_0},
$
and an identical argument as above gives us
\[
\tilde S^\t[\mu_0](y) \geq \eta \mu_0(y) \quad\text{ for all } \t\leq \frac{(1-\eta^{m-1})h_0^{m-1}}{C_0 \eta^{m-1}}.
\]
Now we let $c_m^2>0$ be such $1-\eta^{m-1} \geq c_m^2(1-\eta)$ for all $\frac{1}{2}<\eta<1$. Hence we have $\tilde S^\t[\mu_0](y) - \mu_0(y) \geq
-(1-\eta)\mu_0(y)$ for $\t= \dfrac{c_m^2 h_0^{m-1}}{C_0}(1-\eta)$,
which implies
\begin{equation}
\label{temp2}
\tilde S^\t[\mu_0](y) - \mu_0(y) \geq -\frac{C_0}{c_m^2 h_0^{m-1} }\mu_0(y)\t \quad\text{ for all } \t\leq \frac{c_m^2 h_0^{m-1}}{2C_0}.
\end{equation}
Combining \eqref{temp1} and \eqref{temp2} together, we have that for any $h_0>0$, \eqref{mu_t} holds for some $C_1$ for all $\t\in[0,\delta_1]$,
where both
$C_1>0$ and $\delta_1>0$ depend on $C_0, h_0$ and $m$.

Finally, we will show that \eqref{E_mu_t} holds for $\mu(\t) = \tilde S^\t[\mu_0]$ if we choose $h_0>0$ to be sufficiently small. First, we point out that
$\mathcal{S}[\tilde S^\t\mu_0]$ is \emph{not} preserved for all $\t$. This is because when different level sets are moving at different speed $v(h)$, we no
longer
have that $M^{v(h_1)\t}(U_{x'}^{h_1}) \subset M^{v(h_2)\t}(U_{x'}^{h_2}) $ for all $h_1>h_2$. Nevertheless, we claim it is still true that
\begin{equation}
\label{eq:s_dec}
\mathcal{S}[\tilde S^\t\mu_0] \leq \mathcal{S}[\mu_0] \text{ for all }\t\geq 0.
\end{equation}
To see this, note that the definition of $\tilde S^\t$ and the fact that $M^{v(h)\t}$ is measure preserving give us
 \[
\left|\{\tilde S^\t\mu_0 > h\}\right| \leq \left|\{\mu_0>h\}\right| \quad\text{ for all }h>0, \t\geq 0,
\]
regardless of the definition of $v(h)$. This implies that $\int f(\tilde S^\t\mu_0(x))dx \leq \int f(\mu_0(x))dx$ for any convex increasing function $f$,
yielding
\eqref{eq:s_dec}.

Due to \eqref{eq:s_dec} and the fact that $\mathcal{E}[\cdot] = \mathcal{S}[\cdot]+\mathcal{I}[\cdot]$, in order to prove \eqref{E_mu_t}, it suffices to
show
\begin{equation}\label{eq:goal0}
\mathcal{I}[\tilde S^\t\mu_0] \leq \mathcal{I}[\mu_0] - c_0 \t \quad\text{ for } \t\in[0,\delta_0], \text{ for some $c_0>0$ and $\delta_0>0$.}
\end{equation}
Recall that Proposition \ref{prop:int} gives that $\mathcal{I}[S^\t\mu_0] \leq \mathcal{I}[\mu_0] - c\t$ for $\t \in [0,\delta]$ with some $c>0$ and
$\delta>0$. As
a result, to show \eqref{eq:goal0}, all we need is to prove that if $h_0>0$ is sufficiently small, then
\begin{equation}
\label{eq:goal1}
\left|\mathcal{I}[\tilde S^\t\mu_0] - \mathcal{I}[ S^\t\mu_0]\right| \leq \frac{c\t}{2}\quad \text{ for all $\t$}.
\end{equation}

To show \eqref{eq:goal1}, we first split $S^\t\mu_0$ as the sum of two integrals in $h\in[h_0,\infty)$ and $h\in[0,h_0)$:
\begin{equation}
\label{eq:split1}
S^\t\mu_0(x_1, x') = \int_{h_0}^\infty \chi_{M^\t(U_{x'}^h)}(x_1)dh + \int_{0}^{h_0} \chi_{M^\t(U_{x'}^h)}(x_1)dh =: f_1(\t,x)+f_2(\t,x).
\end{equation}
We then split $\tilde S^\t \mu_0$ similarly, and since $v(h)=1$ for all $h>h_0$ we obtain
\begin{equation}
\label{eq:split2}
\tilde S^\t \mu_0(x_1, x') = f_1(\t,x) +  \int_{0}^{h_0} \chi_{M^{v(h)\t}(U_{x'}^h)}(x_1)dh =: f_1(\t,x) + \tilde f_2(\t,x).
\end{equation}
For any $\t\geq 0$, we have $\|f_1(\t,\cdot)\|_{L^\infty(\R^d)} \leq \|\mu_0\|_{L^\infty(\R^d)}$, while $\|f_2(\t,\cdot)\|_{L^\infty(\R^d)}$ and $\|\tilde
f_2(\t,\cdot)\|_{L^\infty(\R^d)}$ are both bounded by $h_0$. As for the $L^1$ norm, we have that $\|f_1(\t,\cdot)\|_{L^1(\R^d)} \leq \|\mu_0\|_{L^1(\R^d)}$, and
\[\|f_2(\t,\cdot)\|_{L^1(\R^d)} = \|\tilde f_2(\t,\cdot)\|_{L^1(\R^d)} = \int_{\mathbb{R}^d} \min\{\mu_0(x), h_0\}dx =: m_{\mu_0}(h_0),\]
 where $m_{\mu_0}(h_0)$ approaches 0 as $h_0\searrow 0$.

Also, since $v(h)\leq 1$, we know that for each $\t\geq 0$, there is a transport map $\mathcal{T}(\t,\cdot):[0,\infty)\times \mathbb{R}^d\to\mathbb{R}^d$ with
$\sup_{x\in\mathbb{R}^d} |\mathcal{T}(\t,x)-x|\leq 2 \t$, such that $\mathcal{T}(\t,\cdot)\# f_2(\t,\cdot)=\tilde f_2(\t,\cdot)$ (that is, $\int \tilde f_2(\t,x)
\varphi(x)dx = \int
f_2(\t,x)\varphi(\mathcal{T}(\t,x))dx$ for any measurable function $\varphi$). Indeed, since the level sets of $f_2$ are traveling at speed 1 and the level sets of $\tilde f_2$ are traveling with speed $v(h)$, for each $\t$ we can find a transport plan between them with maximal displacement $L^\infty$ distance at most $2\t$ in its support. Let us remark that since these densities are both in $L^\infty$, there is some optimal transport map $\tilde{\mathcal{T}}$ for the $\infty$-Wasserstein such that $|\tilde{\mathcal{T}}(\t,x)-x|\leq 2\t$. Although existence of an optimal map is known \cite{CPJ}, we just need a transport map with this property below.

Using the decompositions \eqref{eq:split1}, \eqref{eq:split2} and the definition of $\mathcal{I}[\cdot]$, we obtain, omitting the $\t$
dependence on the right hand side,
\begin{equation*}
\begin{split}
\left|\mathcal{I}[\tilde S^\t\mu_0] - \mathcal{I}[ S^\t\mu_0]\right|  \leq & \underbrace{\left|\int f_2(\K*f_1) dx - \int \tilde
f_2(\K*f_1)dx\right|}_{=:A_1(\t)}
\\ &+ \frac{1}{2}\underbrace{\left| \int f_2(\K*f_2)dx - \int \tilde f_2(\K*\tilde f_2)dx\right|}_{=:A_2(\t)},
\end{split}
\end{equation*}
and we will bound $A_1(\t)$ and $A_2(\t)$ in the following. For $A_1(\t)$, denote $\Phi(\t,\cdot) =: \K*f_1(\t,\cdot)$, and using the $L^\infty$,
$L^1$ bounds on
$f_1$
and the assumptions (K2),(K3), we proceed in the same way as in \eqref{lipspot} to obtain that $\|\nabla \Phi\|_{L^\infty(\R^d)} \leq C = C(\|\mu_0\|_{L^\infty(\R^d)},
\|\mu_0\|_{L^1(\R^d)}, C_w, d)$.

Using that $\mathcal{T}(\t,\cdot)\# f_2(\t,\cdot) = \tilde f_2(\t,\cdot)$, we can rewrite $A_1(\t)$ as
\[
\begin{split}
A_1(\t) &= \left| \int f_2(x) \Big(\Phi(x) - \Phi(\mathcal{T}(\t,x))\Big)dx\right| \\
&\leq \|f_2(\t)\|_{L^1(\R^d)} \sup_{x\in\mathbb{R}^d}|\Phi(x)-\Phi(\mathcal{T}(\t,x))| \leq m_{\mu_0}(h_0) \|\nabla \Phi\|_{L^\infty(\R^d)} 2\t \\
&\leq m_{\mu_0}(h_0) C(\|\mu_0\|_{L^\infty(\R^d)}, \|\mu_0\|_{L^1(\R^d)}, C_w, d) \t,
\end{split}
\]
where the coefficient of $\t$ can be made arbitrarily small by choosing $h_0$ sufficiently small. To control $A_2(\t)$, we first use the identity $\int f(\K*g)dx =
\int
g(\K*f)dx$ to bound it by
\[
A_2(\t) \leq \left| \int f_2(\K*f_2)dx - \int \tilde f_2(\K* f_2)dx\right| + \left| \int f_2(\K*\tilde f_2)dx - \int \tilde f_2(\K*\tilde f_2)dx\right|,
\]
and both terms can be controlled in the same way as $A_1(\t)$, since both $\Phi_2 := \K*f_2$ and $\tilde \Phi_2 := \K*\tilde f_2$
satisfy the same estimate
as
$\Phi$. Combining the estimates for $A_1(\t)$ and $A_2(\t)$, we can choose $h_0>0$ sufficiently small, depending on $\mu_0$ and $\K$, such that equation \eqref{eq:goal1} would hold for all $\t$, which finishes the proof.
\end{proof}

\subsection{Proof of Theorem \ref{thm:unique}}

\begin{proof}
Towards a contradiction, assume there is a stationary state $\rho_s$ that is not radially decreasing. Due to Lemma \ref{lem:regularity}, we have that
$\rho_s \in \mathcal{C}(\mathbb{R}^d) \cap L^1_+(\mathbb{R}^d)$, and $|\frac{m}{m-1}\nabla \rho_s^{m-1}| \leq C_0$ in $\supp \rho_s$ for some $C_0>0$ (and if $m=1$, it becomes $|\nabla \log \rho_s|\leq C_0$). In addition, if $m\in (0,1]$, the same lemma also gives $\supp \rho_s = \mathbb{R}^d$. This enables us to apply Proposition
\ref{prop:steiner} to $\rho_s$, hence there exists a continuous family of $\mu(\t,\cdot)$ with $\mu(0,\cdot) = \rho_s$ and constants $C_1>0,
c_0>0,\delta_0>0$,
such that the following holds for all $\t\in[0,\delta_0]$:
\begin{equation}\label{eq1}
\mathcal{E}[\mu(\t)] - \mathcal{E}[\rho_s] \leq - c_0 \t,
\end{equation}
\begin{equation} \label{eq11}
|\mu(\t,x) - \rho_s(x)| \leq C_1 \rho_s(x)^{\max\{1,2-m\}} \t \quad \text{ for all }x\in \mathbb{R}^d,
\end{equation}
\begin{equation}\label{eq02}
 \int_{D_i} \mu(\t,x)-\rho_s(x)dx =0 \text{ for any connected component $D_i$ of $\supp\rho_s$.}
 \end{equation}
Next we will use \eqref{eq11} and \eqref{eq02} to directly estimate $\mathcal{E}[\mu(\t)] - \mathcal{E}[\rho_s]$, and our goal is to show that
there exists
some
$C_2>0$, such that
\begin{equation}\label{eq2}
\big| \mathcal{E}[\mu(\t)] - \mathcal{E}[\rho_s] \big| \leq C_2 \t^2 \quad\text{ for $\t$ sufficiently small.}
\end{equation}
We then directly obtain a contradiction between \eqref{eq1} and \eqref{eq2} for sufficiently small $\t>0$.

Let $g(\t,x) := \mu(\t,x)-\rho_s(x)$. Due to \eqref{eq11}, we have $|g(\t,x)| \leq C_1 \rho_s(x)^{\max\{1,2-m\}} \t$ for all $x\in\mathbb{R}^d$ and
$\t\in [0,\delta_0]$. From now on, we set $\delta_0$ to be the minimum of its previous value and $(2C_1(1+\|\rho_s\|_\infty))^{-1}$. Such $\delta_0$ ensures that $\supp
g(\t,\cdot) \subset \supp \rho_s$ and $|g(\t,x)/\rho_s(x)|\leq \frac{1}{2}$ for all $\tau \in [0,\delta_0]$.

Since the energy $\mathcal{E}$ takes different formulas for $m\neq 1$ and $m=1$, we will treat these two cases differently. Let us start with the case $m\in (0,1)\cup(1,+\infty)$. Using the notation $g(\t,x)$, we have the following: (where in the integrand we omit the $x$ dependence, due to space limitations)
\begin{align}
\mathcal{E}[\mu(\t)] - \mathcal{E}[\rho_s] \!=& \!\int \!\frac{\left((\rho_s+g(\t))^m - \rho_s^m\right)}{m-1}  dx +\frac{1}{2} \! \int\! (\rho_s+g(\t)) \big(\K*(\rho_s+g(\t))\big)  - \rho_s  (\K*\rho_s) dx \nonumber\\
=&\int_{\supp \rho_s} \underbrace{\frac{\rho_s^m}{m-1}\left(\left(1+\frac{g(\t)}{\rho_s}\right)^m - 1\right)}_{:= T(\t,x)} dx \nonumber\\
&+ \int \left[g(\t)(\K*\rho_s)+\frac{1}{2}g(\t)(\K*g(\t))\right] dx. \label{eq:e_estimate}
\end{align}
Recall that for all $|a|<1/2$, we have the elementary inequality
\[
\big| (1+a)^m -1 -ma\big| \leq C(m)a^2 \text{ for some }C(m)>0.
\]
Since for all $x\in \supp \rho_s$ and $\tau \in [0,\delta_0]$ we have $|g(\t,x)/\rho_s(x)|\leq \frac{1}{2}$, we can replace $a$ by $g(x)/\rho_s(x)$ in the
above inequality, then multiply $\frac{1}{|m-1|}\rho_s^m$ to both sides to obtain the following (with $C_2(m)=C(m)/|m-1|$):
\[
\left|T(\t,x) - \frac{m}{m-1} g(\t,x) \rho_s(x)^{m-1}\right|\leq C_2(m) \rho_s^{m-2} g(\t)^2.
\]
Applying this to \eqref{eq:e_estimate}, we have the following for all $\t\leq \min\{\delta_0, C_1/2\}$:
\[
\begin{split}
\big|\mathcal{E}[\mu(\t)] - \mathcal{E}[\rho_s]\big| \leq&\left| \int_{\supp\rho_s}g(\t) \left(\frac{m}{m-1} \rho_s^{m-1} + \K*\rho_s\right) dx\right| + \left| \frac{1}{2}\int
g(\t)(\K*g(\t)) dx\right| \\& \qquad\qquad\qquad\qquad\qquad\qquad\qquad\!\!\!+ C_2(m)\left|\int \rho_s^{m-2} g(\t)^2 dx\right|\\
&=: I_1+I_2+I_3.
\end{split}
\]
Since $\rho_s$ is a steady state solution, from \eqref{eq:rho_s_stat} we have $\frac{m}{m-1}\rho_s^{m-1} + \K*\rho_s = C_i$ in each connected component $D_i\subset \supp \rho_s$, hence $I_1 \equiv 0$
for
all $\t\in[0,\delta_0]$ due to \eqref{eq02} and the definition of $g(\t,\cdot)$.

For $I_2$ and $I_3$, since  $|g(\t,x)| \leq C_1 \rho_s(x)^{\max\{1,2-m\}} \t$ for $\t\in[0,\delta_0]$, for $m>1$ it becomes $|g(\t,x)| \leq C_1 \rho_s(x) \t$, thus we directly have
\[
I_2 \leq \frac{1}{2}C_1^2 \t^2 \int |\rho_s(\K*\rho_s)|dx \leq A \t^2,
\]
\[
I_3 \leq C_2(m) C_1^2 \t^2 \int \rho_s^m dx \leq A \t^2,
\]
for some $A>0$ depending on $\|\rho_s\|_{1}, \|\rho_s\|_{\infty}, m$ and $d$ (where we use \eqref{lipspot} and $\rho_s \omega(1+|x|)\in L^1$ to control $I_2$). For $m\in (0,1)$, the bound of $g$ implies $|g(\t,x)| \leq C_1 \|\rho_s\|_{\infty}^{1-m} \rho_s(x) \t$. Plugging this into $I_2$ gives the same bound as above (with a different $A$). And for $I_3$, plugging in $|g(\t,x)| \leq C_1 \rho_s(x)^{2-m} \t$ gives
\[
I_3 \leq C_2(m) C_1^2 \t^2 \int \rho_s^{2-m} \leq A\t^2,
\]
where in the last inequality we used that $2-m>1$ and $\rho_s \in L^1 \cap L^\infty$.
 Putting them together finally gives $\big|\mathcal{E}[\mu(\t)] - \mathcal{E}[\rho_s]\big| \leq 2A\t^2$ for all $\t\leq \delta_0$, finishing the proof for $m\in (0,1)\cup (1,+\infty)$.

Next we move on to the case $m=1$. Using the notation $g(\t,x)$, the difference $\mathcal{E}[\mu(\t)] - \mathcal{E}[\rho_s]$ can be rewritten as follows: (where  we again omit the $x$ dependence in the integrand)
\begin{align*}
\mathcal{E}[\mu(\t)] &- \mathcal{E}[\rho_s] = \!\int \!\left[(\rho_s+g(\t)) \log(\rho_s+g(\t)) - \rho_s \log \rho_s  +
g(\t)(\K*\rho_s)+\frac{1}{2}g(\t)(\K*g(\t))\right]
dx \\
&= \int g(\t) \left( \log \rho_s + W*\rho_s\right)dx+  \int (\rho_s + g(\t)) \log\left(1+\frac{g(\t)}{\rho_s}\right) dx + \frac{1}{2} \int  g(\t)(\K*g(\t)) dx \\
&=: J_1 + J_2 + J_3.
\end{align*}
Again, we have $J_1 = 0$ since $\int g(\t)dx = 0$, and $\log \rho_s + W*\rho_s = C$ in $\mathbb{R}^d$. $J_3$ is the same term as $I_2$, thus again can be controlled by $A\tau^2$. Finally it remains to control $J_2$. Let us break $J_2$ into
\[
J_2 = \int \rho_s \log\left(1+\frac{g(\t)}{\rho_s}\right) dx + \int g(\t) \log\left(1+\frac{g(\t)}{\rho_s}\right) dx =: J_{21} + J_{22}.
\]
For $J_{22}$, using the inequality $\log(1+a)<a$ for all $a>0$, we have
\begin{equation}\label{j22}
J_{22}\leq \int  \frac{g(\t)^2}{\rho_s} dx  \leq \int C_1^2 \tau^2 \rho_s dx \leq C_1^2 \|\rho_s\|_{1} \tau^2 ,
\end{equation} where we use \eqref{eq11} in the second inequality. To control $J_{21}$, due to the elementary inequality
\[
\left| \log\left(1+a\right) -a\right| \leq C a^2 \quad\text{ for all }a>0
\]
for some universal constant $C$, letting $a = \frac{g(\t)}{\rho_s}$ and apply it to $J_{21}$ gives
\[
\Big|J_{21} - \underbrace{\int g(\tau)dx}_{=0 \text{ by \eqref{eq02}}}\Big| \leq C \int  \frac{g(\t)^2}{\rho_s} dx  \leq C C_1^2 \|\rho_s\|_{1} \tau^2 ,
\]
where the last inequality is obtained in the same way as \eqref{j22}. Combining these estimates above gives $|\mathcal{E}[\mu(\t)] - \mathcal{E}[\rho_s]| \leq A\tau^2$ for some $A>0$ depending on $\|\rho_s\|_{1}, \|\rho_s\|_{\infty}$ and $d$, which completes the proof.
\end{proof}

\subsection{A shortcut for equations with a gradient flow structure} \label{sec:gradient_flow}

In this subsection, we would like to discuss a shortcut for proving Theorem \ref{thm:unique}, once the first order decay under continuous Steiner symmetrization in Proposition \ref{prop:int} has been established, if the equation \eqref{aggregation} has a rigorous gradient flow structure. Over the past two decades, it was discovered that many evolution PDEs have a Wasserstein gradient flow structure including the heat equation, porous medium equation, and the aggregation-diffusion equation \eqref{aggregation} if the kernel $W$ has certain convexity properties, see \cite{
JKO,Otto,AGS08,Carrillo-McCann-Villani06,Craig}. More precisely, for \eqref{aggregation}, if $W$ is known to be $\lambda$-convex, then given any $\rho_0 \in \mathcal{P}_2(\mathbb{R}^d)$  (space of non-negative probability measures with finite second-moment) with $\mathcal{E}[\rho_0]<\infty$, there exists a unique gradient flow $\rho(t)$ of the free energy functional $\mathcal{E}[\rho_0]$ in the space $\mathcal{P}_2(\mathbb{R}^d)$ endowed by the 2-Wasserstein distance. In addition, the gradient flow coincides with the unique weak solution if the velocity field has the necessary integrability conditions.

The $\lambda$-convexity of the potential $W$ does not hold in the generality of our assumptions (K1)-(K4). However, the $\lambda$-convexity assumption on $W$ has been recently relaxed in the following works for the particular, but important, case of the attractive Newtonian kernel. \cite{Craig} has shown that the gradient flow is well-posed if the energy $\mathcal{E}$ is $\xi$-convex, where $\xi$ is a modulus of convexity. \cite{CS2018} has recently shown that for  \eqref{aggregation} with attractive Newtonian potential, for any $\rho_0$ in $L^\infty(\mathbb{R}^d) \cap \mathcal{P}_2(\mathbb{R}^d)$, there is a local-in-time gradient flow solution. The authors show that there are local in time $L^\infty$ bounds at the discrete variational level allowing for local in time well defined gradient flow solutions. Furthermore, this gradient flow solution is unique among a large class of weak solutions due to the earlier results \cite{CLM}. There, it was also shown that the free energy functional $\mathcal{E}$ is $\xi$-convex for $m\geq 1-\tfrac1{d}$ in the set of bounded densities $L^\infty(\mathbb{R}^d) \cap \mathcal{P}_2(\mathbb{R}^d)$ with a given fixed bound allowing the use of the recent theory of $\xi$-convex gradient flows in \cite{Craig}. Summarizing, the recent results for the Newtonian attractive kernel \cite{Craig,CLM,CS2018} allow for a rigorous gradient flow structure of the Newtonian attractive kernel case for $m\geq 1-\tfrac1{d}$ with initial data in $L^\infty(\mathbb{R}^d) \cap \mathcal{P}_2(\mathbb{R}^d)$.

In short we now know two particular more restrictive classes of potentials than the assumptions (K1)-(K4), including the Newtonian kernel case, for which a rigorous gradient flow theory has been developed for \eqref{aggregation}. Next we will show that under a rigorous gradient flow structure, once we use continuous Steiner symmetrization to obtain Proposition \ref{prop:int}, it almost directly leads to radial symmetry via the following shortcut. In particular, Proposition \ref{prop:steiner} is not needed. Below is the statement and proof of the new proposition that we include for the sake of completeness. Note that it is weaker than Theorem \ref{thm:unique}, since Wasserstein gradient flow requires solutions to have a finite second moment, and furthermore for the existence of the gradient flow solutions we need to assume $m\geq 1-\tfrac1{d}$. We will discuss this difference in Remark \ref{rmk:difference}.

\begin{proposition}\label{thm:gradient_flow}
Assume that $W$ is such that \eqref{aggregation} has a local-in-time unique gradient flow solution. Let $\rho_s \in  L^\infty(\mathbb{R}^d) \cap  \mathcal{P}_2(\mathbb{R}^d) $ be a stationary solution of \eqref{aggregation} with $ \mathcal{E}[\rho_s]$ being finite. Then $\rho_s$ must be radially decreasing after a translation.
\end{proposition}

\begin{proof}
Towards a contradiction, assume there is a stationary state $\rho_s$ that is not radially decreasing after any translation. As before, Lemma \ref{lem:regularity} yields that
$\rho_s \in
\mathcal{C}(\mathbb{R}^d) \cap L^1_+(\mathbb{R}^d)$. Applying Lemma~\ref{lem:steiner_radial} to $\rho_s$ allows us to find a hyperplane $H$ that splits the mass of $\rho_s$ into half and half, but $\rho_s$ is not symmetric decreasing about $H$. Without loss of generality assume $H = \{x_1=0\}$. Applying Proposition \ref{prop:int} to $\rho_s$ and using the fact that the $L^m$ norm is conserved under the continuous Steiner symmetrization $S^\tau$, we directly have that
\begin{equation}\label{eq:temp111}
\mathcal{E}[S^\tau \rho_s] \leq \mathcal{E}[ \rho_s] - c_0 \tau \quad\text{ for all }\tau \in [0,\delta_0],
\end{equation}
where $c_0, \delta_0$ are strictly positive constants that depend on $\rho_s$. In addition, since the continuous Steiner symmetrization $S^\tau$ gives an explicit transport plan from $\rho_s$ to $S^\tau\rho_s$, where each layer is shifted by no more than distance $\tau$, we have $W_\infty(\rho_s, S^\tau\rho_s) \leq \tau$, thus
\begin{equation}\label{eq:temp222}
W_2(\rho_s, S^\tau\rho_s) \leq W_\infty(\rho_s, S^\tau\rho_s) \leq \tau \quad \text{ for all }\tau>0.
\end{equation}
Using \eqref{eq:temp111} and \eqref{eq:temp222}, the metric slope  $|\partial \mathcal{E}|(\rho_s)$ as defined in \cite[Definition 1.2.4]{AGS08} satisfies
\[
|\partial \mathcal{E}|(\rho_s) = \limsup_{\rho\to \rho_s} \frac{(\mathcal{E}[\rho_s] - \mathcal{E}[\rho])^+}{W_2(\rho_s, \rho)} \geq  \limsup_{\tau\to 0} \frac{(\mathcal{E}[\rho_s] - \mathcal{E}[S^\tau \rho_s])^+}{W_2(\rho_s, S^\tau\rho_s)} \geq c_0.
\]
On the other hand, the local in time gradient flow solution $\rho(t)$ with initial solution $\rho_s$ satisfies an Evolution Differential Inequality (EVI) (see \cite[Definition 2.10]{Craig} when $W$ is the Newtonian kernel), then arguing as in \cite[Proposition 3.6]{AmbrosioGigli}, see also \cite{CLM}, we have that the following energy dissipation inequality is satisfied, for all $t\geq0$
\begin{equation}\label{EDE}
\mathcal{E}(\rho(t))-\mathcal{E}(\rho_{s})\leq-\frac{1}{2}\int_{0}^{t}|\partial \mathcal{E}|^{2}(\rho(\tau))d\tau-\frac{1}{2}\int_{0}^{t}|\rho^{\prime}(\tau)|^{2}d\tau
\end{equation}
both for $\lambda$-convex potentials, actually \eqref{EDE} holds with equality, and for the Newtonian attractive potential. This is a consequence of the map $t\rightarrow |\partial \mathcal{E}|(\rho(t))$ being decreasing and lower semicontinuous, see for instance \cite[Theorem 2.4.15]{AGS08} in the $\lambda$-convex case and \cite[Theorem 3.12]{Craig} in the Newtonian kernel case. Since $\rho(t)\equiv \rho_s$ is a gradient flow solution, plugging it into \eqref{EDE} yields that the left hand side is 0, whereas the right hand side is less than $-\frac{1}{2}c_0^2 t$ which is negative for all $t>0$, a contradiction.
\end{proof}

\begin{remark}\label{rmk:difference}The assumption that $\rho_s$ is a probability measure does not create any actual restriction. If $\rho_s$ is a stationary solution of \eqref{aggregation} with mass $M_0 \neq 1$, we can simply apply Theorem~\ref{thm:gradient_flow} to  $\tilde\rho_s := \frac{\rho_s}{M_0}$, which has mass 1, and it is a stationary solution of  \eqref{aggregation} with some positive coefficients multiplied to the two terms on the right hand side. However, the assumption that $\rho_s$ has finite second moment (which comes in the definition of $\mathcal{P}_2(\mathbb{R}^d)$) makes it more restrictive than Theorem \ref{thm:unique}, which only requires $\omega(1+|x|)\rho_s \in L^1(\mathbb{R}^d)$. Moreover, the assumption of the existence of a local-in-time unique gradient flow solution implies the more restrictive condition on the nonlinear diffusion $m\geq 1-\tfrac1{d}$ in order to be proved with the available literature \cite{AmbrosioGigli,Craig}.
\end{remark}

At the end of this subsection, let us point out that for our main application in this work, where $W = -\mathcal{N}$ is the attractive Newtonian kernel modulo translation and $m>1$, we could have used this shortcut to show that all stationary solution $\rho_s \in L^1_+(\mathbb{R}^d)\cap L^\infty(\mathbb{R}^d)$ with finite second moment must be radially decreasing. However the longer approach (via Proposition \ref{prop:steiner} and Theorem \ref{thm:unique}) has a larger interest for two reasons. One is that as discussed in Remark \ref{rmk:difference}, Theorem \ref{thm:unique} proves radial symmetry in a more general class of stationary solutions and more general nonlinear diffusions. Another reason is that the longer approach does not rely on any convexity assumption on $W$, thus it works even if the equation does not have a rigorous gradient flow structure. Even more, part of the authors have also recently shown that this longer proof can be generalized to kernels that are more singular than Newtonian \cite{CHMV} for which a rigorous gradient flow theory is missing.

\subsection{Including a potential  term}\label{sec_v}
In this subsection, we consider the aggregation-diffusion equation with an extra drift term given by a potential $V(x)$:
\begin{equation}
\label{eq_v}
\partial_t \rho = \Delta \rho^m + \nabla \cdot (\rho\nabla(\K*\rho + V)) \quad x\in\R^d \,, t\geq 0\,
\end{equation}
where we assume that $m>0$, $V(x)\in \mathcal{C}^1(\mathbb{R}^d)$ is radially symmetric, and $V'(r)>0$ for all $r>0$.

For this equation, its stationary solution is defined in the same way as Definition~\ref{stationarystates}, with \eqref{steady} replaced by $\nabla \rho_s^{m} = -\rho_s\nabla (\psi_s + V)$. We point out that Lemma~\ref{lem:regularity} still holds, except that the right hand side of \eqref{temp_grad1} and \eqref{temp_grad2} are now replaced by an $x$-dependent bound $C + |\nabla V(x)|$. From its proof, we know that if $\rho_s$ is a stationary solution, then
\[
\frac{m}{m-1}\rho_s^{m-1} + \rho_s * \K + V = C_i \quad\text{ in } \supp \rho_s,
\]
where $C_i$ may take different values in different components. As before, if $m=1$ then $\frac{m}{m-1}\rho_s^{m-1}$ is replaced by $\log \rho_s$; and if $0<m\leq 1$ we again have that $\supp \rho_s = \mathbb{R}^d$.

Due to the extra potential term, the energy functional $\mathcal{E}[\rho]$ is now given by $\mathcal{S}[\rho] + \mathcal{I}[\rho] + \mathcal{V}[\rho]$, with the extra potential energy  $\mathcal{V}[\rho] := \int \rho V dx$.
We start with a simple observation that the potential energy is non-increasing under continuous Steiner symmetrization, a consequence of properties of continuous Steiner symmetrization in \cite{Brock}.

\begin{lemma} Let $V \in \mathcal{C}(\mathbb{R}^d)$ be radially symmetric and non-decreasing in $|x|$. Let  $\mu \in L^1_+(\mathbb{R}^d)\cap L^\infty(\mathbb{R}^d)$ be such that $\int \mu V dx < \infty$. Then $\int S^\t[\mu] V dx$ is non-increasing for all $\t>0$. \end{lemma}
\begin{proof}
For any $n\in \mathbb{N}_+$, let $\varphi_n(x) := \max\{0, V(n)-V(x)\}$. (Here we define $V(n):=V(x)|_{|x|=n}$ by a slight abuse of notation.) Note that $\supp \varphi_n \subset B(0,n)$, and is non-increasing in $|x|$. By the Hardy-Littlewood inequality for continuous Steiner symmetrization \cite[Lemma 4]{Brock}, we have
\begin{equation}\label{ineq_asdf}
\int_{\mathbb{R}^d} S^\t[\mu] \varphi_n dx = \int_{\mathbb{R}^d} S^\t[\mu] S^\t[\varphi_n] dx \geq \int_{\mathbb{R}^d} \mu  \varphi_n dx \quad\text{ for all }\t\geq 0, n\in\mathbb{N}^+
\end{equation}
Note that $-\varphi_n = \min\{V(n), V(x)\} - V(n).$ Since $\int S^\t[\mu] dx = \int \mu dx$, \eqref{ineq_asdf} is equivalent with
\[
\int_{\mathbb{R}^d} S^\t[\mu] \min\{V(x), V(n)\} dx \leq \int_{\mathbb{R}^d} \mu \min\{V(x), V(n)\} dx \quad\text{ for all }\t\geq 0, n\in\mathbb{N}^+.
\]
Sending $n\to\infty$, the above inequality becomes $\int S^\t[\mu] V dx\leq \int \mu V dx$ for all $\t\geq 0$. The semigroup property of $S^\t$ then gives us the desired result.
\end{proof}

The above lemma gives that $\frac{d^+}{d\t} \int S^\t[\mu] V dx \leq 0$, but it turns out that we have to improve it into a strict inequality if $\mu$ is not symmetric decreasing about $H = \{x_1=0\}$, which we prove below.

\begin{lemma}
\label{lem_v2} Let $V \in \mathcal{C}(\mathbb{R}^d)$ be radially symmetric and strictly increasing in $|x|$. Assume $\mu \in L^1_+(\mathbb{R}^d)\cap L^\infty(\mathbb{R}^d)$ is such that $\int \mu V dx < \infty$, and $\mu$ is not symmetric decreasing about $H = \{x_1=0\}$. Then $\frac{d^+}{d\t} \int S^\t[\mu] V dx \big|_{\t=0}< 0.$ As a consequence, for such $\mu$, there is a constant $c_{0}>0$ (depending on $\mu$ and $V$) such that for small $\tau>0,$
\[
\int S^\t[\mu] V dx \leq \int \mu\, V dx-c_{0}\t.
\]
\end{lemma}

\begin{proof} Recall that for each $x'\in \mathbb{R}^{d-1}$, $h\in \mathbb{R}^+$, the set $U_{x'}^h$ is an at most countable union of subintervals. Without loss of generality we assume the subintervals do not share a common endpoint; if so, we add a point to merge them into one interval. Each subinterval can be written in the form $I(c,r)=(c-r,c+r)$. Since $\mu$ is not symmetric decreasing about $H$, some of these subintervals must have their center not at $0$ for some $x', h$. This motivates us to define the set $B_\delta \subset \mathbb{R}^{d-1} \times \mathbb{R}^+$ for $0< \delta \ll 1$:
\[
B_\delta := \{(x', h) \in \mathbb{R}^{d-1} \times \mathbb{R}^+ : |x'| \leq \delta^{-1}, \text{ and }U_{x'}^h \text{ has a subinterval }I(c,r) \text{ with }|c|, r \in [\delta, \delta^{-1}]\}.
\]
The assumption of $\mu$ implies that $|B_\delta| > 0$ for sufficiently small $\delta>0$.

By Definition~\ref{def:steiner_func}, $\int S^\t[\mu] V dx$ can be written as
\begin{equation}\label{multi_int}
\int S^\t[\mu] V dx = \int_{\mathbb{R}^+} \int_{\mathbb{R}^{d-1}} \int_{\mathbb{R}} \chi_{M^\t(U_{x'}^h)}(x_1) V(x_1, x') dx_1 dx' dh.
\end{equation}
Now let us investigate the innermost integral.  For any open set $U \subset \mathbb{R}$, let us define
\[
\Phi(\tau; U,x'):= \int_{\mathbb{R}} \chi_{M^\t(U)}(x_1) V(x_1, x') dx_1.
\]
With this notation, the innermost integral in \eqref{multi_int} becomes $\Phi(\tau; U_{x'}^h,x')$.

To estimate $\frac{d^+}{d\t} \Phi(\tau; U_{x'}^h,x')|_{\t=0}$, let us start with an easier estimate $\frac{d^+}{d\t} \Phi(\tau; U,x')|_{\t=0}$ when $U$ is a single interval $I(c,r)$. If $c=0$, clearly $\frac{d^+}{d\t} \Phi(\tau; U,x')\big|_{\t=0} = 0$. If $c\neq 0$ (WLOG assume $c<0$), then $M^\t(U) = I(c+\t, r)$ for sufficiently small $\t>0$,
thus
\[
\frac{d^+}{d\t} \Phi(\tau; U,x')\Big|_{\t=0} = V(c+r, x') - V(c-r, x') < 0,
\]
where we use $|c+r|<|c-r|$ in the last inequality, which follows from $c<0$, and actually we have $|c-r|-|c+r|\geq \min\{2|c|, 2r\}$. And if $c,r, x'$ satisfy  $|c|, r\in [\delta, \delta^{-1}]$ and $|x'| \leq \delta^{-1}$, we have the quantitative estimate
\[
\frac{d^+}{d\t} \Phi(\tau; U,x')\Big|_{\t=0} \leq -C_\delta < 0,
\]
where $C_\delta$ is given by
\[
C_\delta := \inf_{a_1,a_2,b\in\mathbb{R}}\left\{ V\left(\sqrt{a_1^2 + b^2}\right) - V\left(\sqrt{a_2^2 + b^2}\right): |a_1|-|a_2|\geq 2\delta, |a_1|,|a_2|\leq 2\delta^{-1}, |b|\leq \delta^{-1}\right\},
\]
where we denote $V(x)=V(|x|)$ by a slight abuse of notation. The strict positivity of $C_\delta$ follows from the fact that $V(r)$ is strictly increasing in $r$ for $r\geq 0$, as well as the compactness of the set $\{ |a_1|-|a_2|\geq 2\delta, |a_1|,|a_2|\leq 2\delta^{-1}, |b|\leq \delta^{-1}\}$.

The above argument immediately leads to the crude estimate
\[
\frac{d^+}{d\t} \Phi(\tau; U_{x'}^h,x')|_{\t=0}\leq 0 \quad\text{ for all } (x', h)\in\mathbb{R}^{d-1}\times R^+
\] as we take the sum of the estimate $\frac{d^+}{d\t} \Phi(\tau; U,x')|_{\t=0}\leq 0$  over all the subintervals $U \subset U_{x'}^h$. In addition, if $|x'|\leq \delta^{-1}$ and $U_{x'}^h$ has a subinterval $I(c,r)$ with $|c|, r\in [\delta, \delta^{-1}]$, we have the quantitative estimate $\frac{d^+}{d\t} \Phi(\tau; U_{x'}^h,x')|_{\t=0} \leq -C_\delta <0$. By definition of $B_\delta$ at the beginning of this proof, we have
\[
\frac{d^+}{d\t} \Phi(\tau; U_{x'}^h,x')\Big|_{\t=0} \leq -C_\delta <0 \quad\text{ for all } (x', h)\in B_\delta,
\]
thus
\[
\frac{d^+}{d\t} \int S^\t[\mu] V dx \Big|_{\t=0} =  \int_{\mathbb{R}^+} \int_{\mathbb{R}^{d-1}} \frac{d^+}{d\t}\Phi(\tau; U_{x'}^h,x')\Big|_{\t=0}  dx' dh \leq -C_\delta | B_\delta| < 0,
\]
finishing the proof.\end{proof}

Our goal of this subsection is to show that the radial symmetry result in Theorem~\ref{thm:unique} can be generalized to \eqref{eq_v} for certain classes of potential $V$. We will work with one of the following two classes of $V$:

(V1) $0<V'(r)\leq C$ for some $C$ for all $r>0$.

(V2) $V'(r)>0$ for all $r>0$, and $V'(r)\to+\infty$ as $r\to+\infty$.

In the following theorem we prove radial symmetry of stationary solutions under assumption (V1) for all $m>0$, and under assumption (V2) for $m>1$. We expect that when $m\in (0,1]$, it should be possible to refine some estimates in the proof and obtain symmetry for a wider class than (V1). We will not pursue this direction for presentation simplicity, and we  leave further generalizations to interested readers.

\begin{theorem}
\label{thm:symm_v}
Assume that $W$ satisfies $(K1)$-$(K4)$ and $m>0$. Let $\rho_s \in L^1_+(\mathbb{R}^d)\cap L^\infty(\mathbb{R}^d)$ satisfy $\omega(1+|x|)\rho_s \in L^1(\mathbb{R}^d)$ and $ \rho_sV \in L^1(\mathbb{R}^d)$. Assume that $\rho_s$ is a non-negative stationary state of \eqref{eq_v} in the sense of Definition {\rm \ref{stationarystates}}, with \eqref{steady} replaced by $\nabla \rho_s^{m} = -\rho_s\nabla (\psi_s + V)$. Then if $V$ satisfies (V1), or if $V$ satisfies (V2) in addition to $m>1$, then $\rho_s$ is radially decreasing about the origin.
\end{theorem}

\begin{proof}
Note that  Lemma~\ref{lem:regularity} still holds with a potential $V$, except that right hand sides of \eqref{temp_grad1} and \eqref{temp_grad2} are now replaced by an $x$-dependent bound $C + |\nabla V(x)|$, which is uniformly bounded in $x$ under (V1). And under the assumptions (V2) and $m>1$, we will prove in Lemma~\ref{lem_cpt_spt} that $\rho_s$ must be compactly supported. Thus in both cases, the right hand sides of \eqref{temp_grad1} and \eqref{temp_grad2} are still uniformly bounded in $x$ in $\supp \rho_s$.

The rest of the proof follows a similar approach as Theorem~\ref{thm:unique} and Proposition~\ref{prop:steiner}, with $\mathcal{E}$ including an extra potential energy $\mathcal{V}[\rho] := \int \rho V dx$. However, some crucial modifications in the proof of Proposition~\ref{prop:steiner} are needed, which we highlight below.

First, note that with a potential $V$, we will prove radial symmetry about the origin, rather than up to a translation. For this reason, we take an arbitrary hyperplane $H$ passing through the origin, and aim to prove that $\rho_s$ is symmetric decreasing about $H$. (WLOG we let $H = \{x_1=0\}$.) Since $H$ does not split the mass of $\rho_s$ into half-and-half, it is possible that for all $x' \in \mathbb{R}^{d-1}$ and $h>0$, every line segment in $U_{x'}^h$ has its center lying on one side of $H$. Therefore, the estimate in Proposition~\ref{prop:int} might fail for $\rho_s$, and all we have is the crude estimate
\begin{equation}\label{crude111}
\mathcal{I}[S^\t \rho_s] - \mathcal{I}[\rho_s] \leq 0.
\end{equation}

Despite this weaker estimate in the interaction energy, we will show that all 3 estimates of Proposition~\ref{prop:steiner} still hold, if we define $\mu(\cdot,\t)$ in the same way as in its proof. Clearly, \eqref{mu_t} and \eqref{int_0} remain true since $\mu(\cdot,\t)$ is defined the same as before. We claim that \eqref{E_mu_t} still holds, but with a different reason as before: the coefficient $c_0>0$ used to come from contribution from the interaction energy via Proposition~\ref{prop:int}, but now it comes from the potential energy. To see this, consider the following two cases.

Case 1:  $m\in (0,1]$. Combining \eqref{crude111}, Lemma~\ref{lem_v2} with $\mathcal{S}[\mu(\t)] -  \mathcal{S}[\rho_s] \equiv 0$ (where the difference is defined in the sense of \eqref{def_reg_s}), we again have \eqref{E_mu_t} for some $c_0>0$ for all sufficiently small $\t>0$.

Case 2: $m>1$. In this case, recall that $\mu(\t,\cdot) = \tilde S^\t[\mu(0,\cdot)]$, where $\mu(0,\cdot)=\rho_s$ and $\tilde S^\t$ is the continuous Steiner symmetrization which ``slows-down'' at height $h \in (0,h_0)$. From the proof of Lemma \ref{lem_v2}, we know that if $B_\delta$ has a positive measure, then $B_\delta \cap \{(x',h): h>h_0\}$ also has a positive measure for all sufficiently small $h_0>0$, thus Lemma~\ref{lem_v2} still holds for $\mu(\t) = \tilde S^\t[\mu_0]$ if $h_0$ is sufficiently small, leading to
\[
\mathcal{V}[\mu(\t)] - \mathcal{V}[\rho_s] \leq -c\t \text{\quad for some $c>0$ for all sufficiently small $\t>0$.}
\]
 In addition, for sufficiently small $h_0$ we still have \eqref{eq:goal1} (where we fix $c$ to be the constant from the above equation), and combining it with \eqref{crude111} gives
 \[
 \mathcal{I}[\mu(\t)] - \mathcal{I}[\rho_s] \leq \frac{c\t}{2},
 \]
 and adding them together with \eqref{eq:s_dec} gives \eqref{E_mu_t}.

Once we obtain Proposition~\ref{prop:steiner}, the rest of the proof follows closely the proof of Theorem~\ref{thm:unique}, except the following minor changes. With an extra potential energy in $\mathcal{E}$, the right hand side of \eqref{eq:e_estimate} has an addition term $\int g(\t) V dx$. As a result, $I_1$ has a different definition
\[
I_1 =\left| \int_{\supp\rho_s}g(\t) \left(\frac{m}{m-1} \rho_s^{m-1} + \K*\rho_s + V\right) dx\right|,
\]
which is still 0, since the equation for stationary solution now becomes
\[
\frac{m}{m-1}\rho_s^{m-1} + \rho_s * \K + V = C_i \quad\text{ in } \supp \rho_s.
\]
The $m=1$ case is done with a similar modification, where $J_1$ is now $\int g(\t) \left( \log \rho_s + W*\rho_s + V\right)dx$, and again we have $J_1=0$ since $\rho_s$ is stationary. Finally, we obtain the same contradiction as the proof of Theorem~\ref{thm:unique} if $\rho_s$ is not symmetric decreasing about $H$. And since $H$ is an arbitrary hyperplane through the origin, we have that $\rho_s$ is radially decreasing about the origin.
\end{proof}

Finally we state and prove the lemma used in the proof of Theorem~\ref{thm:symm_v}, which shows all stationary solutions must be compactly supported if $m>1$ and $V$ satisfies (V2).

\begin{lemma}\label{lem_cpt_spt}Assume that $m>1$, $W$ satisfies $(K1)$-$(K4)$, and $V$ satisfies (V2). Let $\rho_s \in L^1_+(\mathbb{R}^d)\cap L^\infty(\mathbb{R}^d)$ satisfy $\omega(1+|x|)\rho_s \in L^1(\mathbb{R}^d)$. Assume that $\rho_s$ is a non-negative stationary state of \eqref{eq_v} in the sense of Definition {\rm \ref{stationarystates}}, with \eqref{steady} replaced by $\nabla \rho_s^{m} = -\rho_s\nabla (\psi_s + V)$.  Then $\rho_s$ is compactly supported.
\end{lemma}

\begin{proof}
With a potential term, we have that
\begin{equation}\label{temp_stat_eq}
\frac{m}{m-1}\rho_s^{m-1} + \rho_s * \K + V = C_i \quad\text{ in } \supp \rho_s,
\end{equation}
where $C_i$ takes different values in different connected components of $\supp \rho_s$. By a similar computation as \eqref{lipspot} (with $\K$ replaced by $\min\{\K,0\}$), we have $\rho_s*\K \geq -C(\|\rho_s\|_1, \|\rho_s\|_\infty, \K)$. Thus the first two terms of \eqref{temp_stat_eq} are uniformly bounded below. As a result, every connected component $D$ of $\supp \rho_s$ must be bounded: if not, the left hand side would be unbounded in $D$ due to $\lim_{|x|\to\infty} V(|x|)=\infty$, contradicting with \eqref{temp_stat_eq}.

Note that every connected component being bounded does not imply that $\supp \rho_s$ is bounded: there may be a countable number of connected components going to infinity. We claim that there is some $R(\|\rho_s\|_1, \|\rho_s\|_\infty, \K, V)>0$, such that every connected component $D$ must satisfy that $D \cap B(0,R) \not= \emptyset$. As we will see later, this will help us control the outmost point of $D$.

If $0\in D$, then clearly $D \cap B(0,R) \not= \emptyset$. If $0\not\in D$, we find some unit vector $\nu \in \mathbb{R}^d$, such that the ray starting at origin with direction $\nu$ has a non-empty intersection with $D$. Let
$
t_0 = \inf\{t>0: t\nu \in D\},
$
and let $x_0 = t_0 \nu$.  We take a sequence of points $(t_n)_{n=1}^\infty$ such that $t_n\searrow t_0$ and $t_n \nu \in D$, and denote $x_n = t_n \nu$. Since $x_n \in D$ and $x_0 \in \partial D$,  the left hand side of \eqref{temp_stat_eq} takes the same constant value $C_i$ at $x_0$ and all $x_n$. As a result, for all $n\geq 1$ we have
\[
\frac{\frac{m}{m-1}\left(\rho_s^{m-1}(x_n) - \rho_s^{m-1}(x_0)\right)}{t_n-t_0} + \frac{(\rho_s * \K)(x_n) -  (\rho_s * \K)(x_0)}{t_n-t_0} + \frac{V(x_n)-V(x_0)}{t_n-t_0} = 0.
\]
Note that the first term is non-negative since $\rho_s(x_0)=0$ (which follows from $x_0\in \partial D$ and $\rho_s\in \mathcal{C}(\mathbb{R}^d)$). The second term converges to $\nabla(\rho_s*\K)\cdot \nu$, whose absolute value is bounded by $C(\|\rho_s\|_1, \|\rho_s\|_\infty, \K)$ by \eqref{lipspot0}. The third term converges to $\nabla V(x_0) \cdot \nu = V'(t_0)$. Putting the three estimates together gives that
\[
V'(t_0) \leq C(\|\rho_s\|_1, \|\rho_s\|_\infty, \K),
\]
thus assumption (V2) gives that $t_0 \leq R(\|\rho_s\|_1, \|\rho_s\|_\infty, \K, V)$, finishing the proof of the claim.

Finally, we will show that $D \cap B(0,R) \not= \emptyset$ implies the outmost point of $D$ cannot get too far. Take any $x_1 \in D\cap B(0,R)$, and let $x_2$ be the outmost point of $D$. Taking the difference of \eqref{temp_stat_eq} at $x_2$ and $x_1$ gives
\[  V(x_2)-V(R) \leq V(x_2)-V(x_1) = \frac{m}{m-1}\rho_s^{m-1}\Big|_{x_2}^{x_1} + (\rho_s * W)\Big|_{x_2}^{x_1}.
\]
Due to \eqref{lipspot}, we bound the right hand side by $C(\|\rho_s\|_1,\|\rho_s\|_\infty, \|\omega(1+|x|)\rho_s\|_1, \K) + \omega(1+|x_2|) \|\rho_s\|_1$. Note that the left hand grows superlinearly in $|x_2|$ due to (V2), whereas $\omega(1+|x_2|)$ at most grows linearly in $|x_2|$ by assumption (K3) on $\K$. This leads to
\[
|x_2| \leq C(\|\rho_s\|_1,\|\rho_s\|_\infty, \|\omega(1+|x|)\rho_s\|_1, \K, V),
\]
which completes the proof.
\end{proof}


\section{Existence of global minimizers}

In Section 2, we showed that if $\rho_s \in L^1_+(\mathbb{R}^d)\cap L^\infty(\mathbb{R}^d)$ is a stationary state of \eqref{aggregation} in the sense of Definition  \ref{stationarystates} and it satisfies $\omega(1+|x|)\rho_s \in L^1(\mathbb{R}^d)$, then it must be radially decreasing up to a translation. This section is concerned with the existence of such stationary solutions. Namely, under (K1)-(K4) and one of the extra assumptions (K5) or (K6) below, we will show that for any given mass, there indeed exists a stationary solution satisfying the above conditions.
We will generalize the arguments of \cite{CCV} to show that there exists a radially decreasing global minimizer $\rho$ of the functional \eqref{eq:energy} given by
\begin{align*}
\E[\rho]=\frac{1}{m-1}\int_{\R^{d}}\rho^{m}\,dx+\frac{1}{2}\int_{\R^{d}}\int_{\R^{d}}W(x-y)\,\rho(x)\rho(y)\,dx\,dy\,
\end{align*}
over the class of admissible densities
\begin{equation*}
\mathcal{Y}_{M}:=\left\{\rho\in L_{+}^{1}(\R^{d})\cap L^{m}(\R^{d}):\,\|\rho\|_{L^1(\R^{d})}=M,\int_{\R^{2}}x\rho(x)\,dx=0,\,\omega(1+|x|)\,\rho(x)\in
L^{1}(\R^{d})\right\},
\end{equation*}
and with the potential satisfying at least (K1)-(K4). Note that the condition on the zero center of mass has to be understood in the improper integral sense, i.e.
$$
\int_{\R^{d}}x\rho(x)\,dx=\lim_{R\to\infty} \int_{|x|<R}x\rho(x)\,dx = 0\,
$$
since we do not assume that the first moment is bounded in the class $\mathcal{Y}_{M}$. We emphasize that from now on we will work in the \emph{dominated regime} with degenerate diffusion, namely when
\begin{equation}
m>\max\left\{2-\frac{2}{d},1\right\}.\label{DDR}
\end{equation}
In order to avoid loss of mass at infinity, we need to assume some growth condition at infinity. In this section, we will obtain the existence of global minimizers under two different conditions related to the works \cite{Lions,Bedrossian,CCV}, and show that such global minimizers are indeed $L^1$ and $L^\infty$ stationary solutions. Namely, we assume further that the potential $W$ satisfies at infinity either the property
\begin{enumerate}[(a)]
\item[(K5)] $\displaystyle\lim_{r\rightarrow+\infty}\omega_{+}(r)=+\infty$,
\end{enumerate}
or
\begin{enumerate}[(a)]
\item[(K6)]
$\displaystyle\lim_{r\rightarrow+\infty}\omega_{+}(r)=\ell\in(0,+\infty)$ where the non-negative potential $\mathcal{K}:=\ell -W$ is such that, in the case $m>2$,
$\mathcal{K}\in L^{\hat{p}}(\R^{d}\setminus B_{1}(0))$, for some $1\leq\hat{p}<\infty$,
while for the case $2-(2/d)<m\leq 2$ we will require that
$\mathcal{K}\in L^{p,\infty}(\R^{d}\setminus B_{1}(0))$, for some $1\leq p<\infty$.
Moreover, there exists an $\alpha\in (0,d)$ for which $m>1+\alpha/d$ and
\begin{equation}
\mathcal{K} (\t x)\geq \t^{-\alpha} \mathcal{K}(x),\quad\forall \t\geq1,\, \mbox{for a.e. } x\in \R^{d}.\label{homogeneity}
\end{equation}
\end{enumerate}
Here, we denote by $L^{p,\infty}(\R^{d})$ the weak-$L^p$ or Marcinkiewicz space of index $1\leq p<\infty$. In particular, the attractive Newtonian potential (which is the fundamental solution of $-\Delta$ operator in $\mathbb{R}^d$) is covered by these assumptions: for $d=1,2$ it satisfies (K5), whereas for $d\geq 3$ it satisfies (K6) with $\alpha = d-2$.

Notice that the subadditivity-type condition (K4) allows to claim that $\E[\rho]$ is finite over the class $\mathcal{Y}_{M}$: indeed if we split the $W$ into its positive part $W_{+}$ and negative part $W_{-}$ as done in the bound of $\psi_s$ in Section 2, the integral with kernel $W_{-}$ is finite by the HLS inequality, see \eqref{minimizer2} below, while by (K4) we infer
\begin{align*}
\int_{\R^{d}}\int_{\R^{d}}W_{+}(x-y)\,\rho(x)\rho(y)\,dx\,dy&=\int_{\R^{d}}\int_{|x-y|\geq1}\omega_{+}(|x-y|)\,\rho(x)\rho(y)\,dx\,dy\\
&\leq C M^{2}+2M\int_{\R^{d}}\omega(1+|x|)\rho(x)dx.
\end{align*}

\subsection{Minimization of the Free Energy functional}\label{Minimizsub}
The existence of minimizers of the functional $\E$ can be proven with different arguments according to the choice between condition $(K5)$ or $(K6)$: indeed, $(K5)$ produces a quantitative version of the mass confinement effect while $(K6)$ does it in a nonconstructive way. For such a difference, we first briefly discuss the case when condition $(K6)$ is employed, as it can be proven by a simple application of Lion's concentration-compactness principle \cite{Lions} and its variant in \cite{Bedrossian}.
\begin{theorem}\label{conccomp}
Assume that conditions \eqref{DDR}, $(K1)$-$(K4)$ and $(K6)$ hold. Then for any positive mass $M$, there exists a global minimizer
$\rho_{0}$, which is radially symmetric and decreasing, of the free energy functional
$\E$ in $\mathcal{Y}_{M}$. Moreover, all global minimizers are radially symmetric and decreasing.
\end{theorem}
\begin{proof}
We write
$\E[\rho]=\widetilde{\E}[\rho]+\frac{\ell}{2}M^{2}$, where
\[
\widetilde{\E}[\rho]=\frac{1}{m-1}\int_{\R^{d}}\rho^{m}\,dx-\frac{1}{2}\int_{\R^{d}}\int_{\R^{d}}\mathcal{K}(x-y)\,\rho(x)\rho(y)\,dx\,dy,
\]
being the kernel $\mathcal{K}$ nonnegative and radially decreasing; furthermore condition $(K3)$ implies $\mathcal{K}\in L^{p,\infty}(B_{1}(0))$, where $p=d/(d-2)$. Then we are in position to apply \cite[Theorem 1]{Bedrossian} for $m>2$ and \cite[Corollary II.1]{Lions} for $2-(2/d)<m\leq2$ to get the existence of a radially decreasing minimizer $\rho_{0}\in \mathcal{Y}_{M}$ of $\widetilde{\E}$ (and then of ${\E}$). Moreover, since $\mathcal{K}$ is strictly radially decreasing, all global minimizers are radially decreasing.
\end{proof}
When considering the presence of condition $(K5)$ the concentration-compactness principle is not applicable but a direct control of the mass confinement phenomenon is possible. Then we first prove the following Lemma, which provides a reversed Riesz inequality, allowing to reduce the study the minimization of $\E$ to the set of all the radially decreasing density in $\mathcal{Y}_{M}$.

\begin{lemma}\label{reversedRiesz}
Assume that conditions $(K1)$-$(K5)$ hold and take a density $\rho$ such that
\[
\rho\in L_{+}^{1}(\R^{d}),\,\omega(1+|x|)\,\rho(x)\in L^{1}(\R^{d}).
\]
Then the following inequality holds:
\[
\mathcal{I}[\rho]=\int_{\R^{2}}\int_{\R^{2}}W(x-y)\rho(x)\rho(y)dx\,dy\geq\int_{\R^{2}}\int_{\R^{2}}W(x-y)\rho^{\#}(x)\rho^{\#}(y)dx\,dy=\mathcal{I}[\rho^\#]
\]
and the equality occurs if and only if $\rho$ is a translate of $\rho^{\#}$.
\end{lemma}
\begin{proof}
The proof proceeds exactly as in \cite[Lemma 2]{CarlenLoss}, up to replacing the function $k(r)$ defined there by the function
\[
\kappa(r)=
\left\{
\begin{array}
[c]{lll}%
-\omega(r)   &  & \text{if }r\leq r_{0}%
\\[10pt]
-\omega(r_{0})-\int_{r_{0}}^{r}\omega^{\prime}(s)\frac{1+r_{0}^2}{1+s^{2}}ds &  & \text{if }r>r_{0},
\end{array}
\right. %
\]
being $r_{0}>0$ fixed.
\end{proof}

\begin{theorem}\label{thm36}
Assume that \eqref{DDR} and $(K1)$-$(K5)$ hold, then the conclusions of Theorem {\rm\ref{conccomp}} remain true.
\end{theorem}
\begin{proof}
We follow the main lines of \cite[Theorem 2.1]{CCV}. By Lemma \ref{reversedRiesz} we can restrict ourselves to consider only radially decreasing densities $\rho$. In order to show that $\mathcal{I}[\rho]$ is bounded from below, we first argue in the case $d\geq3$. Thanks to conditions $(K1)$-$(K2)$ we have
\begin{align*}
\int_{\R^{d}}\int_{\R^{d}}W(x-y)\rho(x)\rho(y)dx\,dy\geq -C \int_{\R^{d}}\int_{|x-y|\leq1}\frac{\rho(x)\rho(y)}{|x-y|^{d-2}}dx\,dy
\geq -C \int_{\R^{d}}\int_{\R^{d}}\frac{\rho(x)\rho(y)}{|x-y|^{d-2}}dx\,dy.
\end{align*}
Now we observe that by \eqref{DDR} we have
\begin{equation*}
1<\frac{2d}{d+2}<m
\end{equation*}
and $\frac{d-2}{d}+\frac{d+2}{2d}+\frac{d+2}{2d}=2$, then by the classical HLS and  $L^{p}$ interpolation inequalities, we find
\begin{equation}
\mathcal{I}[\rho]=
\int_{\R^{d}}\int_{\R^{d}}W(x-y)\rho(x)\rho(y)dx\,dy\geq -C\|\rho\|^{2}_{L^{2d/(d+2)}(\R^{d})}\geq -C\|\rho\|^{2\alpha}_{L^{1}(\R^{d})}
\|\rho\|^{2(1-\alpha)}_{L^{m}(\R^{d})}\label{minimizer2},
\end{equation}
where $\alpha=\frac{1}{m-1}\(m\frac{d+2}{2d}-1\)$.
Then by \eqref{minimizer2} we find that
\begin{equation}
\E[\rho]\geq \frac{1}{m-1}\|\rho\|^{m}_{L^{m}(\R^{d})}-CM^{2\alpha}
\|\rho\|^{2(1-\alpha)}_{L^{m}(\R^{d})}\label{minimizer3}
\end{equation}
where we notice that
$m>2(1-\alpha)$
if and only if
$m>2-\frac{2}{d}$,
that is \eqref{DDR}. Then by \eqref{minimizer3} we can find a constant $C_{1}>0$ and a sufficiently large constant $C_{2}$
such
that
\[
\E[\rho]\geq-C_{1}+C_{2}\|\rho\|^{m}_{L^{m}(\R^{d})}.
\]
Concerning the case $d=2$, we observe that conditions $(K1)$-$(K2)$ yields
\begin{align*}
\int_{\R^{2}}\int_{\R^{2}}W(x-y)\rho(x)\rho(y)dx\,dy & \geq -C\int_{\R^2}dx\int_{|x-y|\leq1}\log(|x-y|)\rho(x)\rho(y)dxdy\\
&\geq- C\int_{\R^2}dx\int_{\R^2}\log(|x-y|)\rho(x)\rho(y)dxdy
\end{align*}
and we can use the classical log-HLS inequality and the arguments of \cite{CCV} to conclude.\\[2pt]
Concerning the mass confinement, due to (K5) and the same arguments in \cite{CCV}, see also Lemma \ref{conv.weak.L1.long}, allow us to show
\[
\int_{|x|>R}\rho(x)\,dx\leq \frac{C}{\omega(R)}\underset{R\rightarrow\infty}{\longrightarrow0}.
\]

Finally, we should check that the interaction potential $W$ is lower semicontinuous as shown in \cite[page8]{CCV}. Indeed, the only technical point to verify in this more general setting relates to the control of the truncated interaction potential $\mathsf{A}^{\varepsilon}$ for $d\geq 3$. Notice that we can estimate due to \eqref{defphi}
\begin{align*}
|\mathsf{A}^{\varepsilon}[\rho]|:=\left|\int_{\R^{d}}\int_{|x-y|\leq\varepsilon}W(x-y)\rho(x)\rho(y)dxdy\right|&\leq C \int_{\R^{d}}\int_{|x-y|\leq\varepsilon}\frac{\rho(x)\rho(y)}{|x-y|^{d-2}}dxdy\\&=C \int_{\R^{d}}\int_{\R^{d}}\rho(x)\,(\chi_{B_{\varepsilon}(0)}\,\mathcal{N})(x-y)\,\rho(y)\,dxdy.
\end{align*}
Now recall that the Newtonian potential
\[
\Psi_{\rho}(x)=\int_{\R^{d}}\frac{\rho(y)}{|x-y|^{d-2}}dy
\]
is well defined for a.e. $x\in \R^{d}$ and is in $L^{1}_{loc}(\R^{d})$, see \cite[Theorem 2.21]{Folland}, then for a.e. $x\in\R^{d}$ we have
$\chi_{B_{\varepsilon}(0)}\,\mathcal{N}\ast\rho\rightarrow0$ as $\varepsilon\rightarrow 0$.
Moreover, by the HLS inequality we have
\[
\rho(x)(\chi_{B_{\varepsilon}(0)}\,\mathcal{N}\ast\rho)(x)\leq \rho(x)\Psi_{\rho}(x)\in L^{1}_{loc}(\R^{d})
\]
with
\[
\|\rho(x)\Psi_{\rho}\|_{L^{1}(\R^{d})}\leq C\|\rho\|^{2\alpha}_{L^{1}(\R^{d})}
\|\rho\|^{2(1-\alpha)}_{L^{m}(\R^{d})}.
\]
Then Lebesgue's dominated convergence theorem allows to conclude that $\mathsf{A}^{\varepsilon}[\rho]\rightarrow 0$ as $\varepsilon\rightarrow 0$. This convergence is uniform taken on a minimizing sequence $\rho_{n}$.

Now, all ingredients are there to argue as in \cite{CCV} showing that $\E$ achieves its infimum in the class of all radially decreasing densities in $\mathcal{Y}_{M}$.
\end{proof}

\begin{remark}
According to Theorem \ref{thm:unique}, the radial symmetry of the global minimizers of $\E$, which are particular critical points of $\E$, is not a surprise. Nevertheless, as pointed out in the proofs of Theorems \ref{conccomp}-\ref{thm36}, this property can be much more easily achieved by rearrangement inequalities.
\end{remark}

A useful result, which will be used in the next arguments, regards the behavior at infinity of the so called $W$-potential, namely the function
\begin{equation*}
\psi_{f}(x)=\int_{\R^{d}}W(x-y)f(y)dy.
\end{equation*}
Following the blueprint of \cite[Lemma 1.1]{CT}, we have the following result.
\begin{lemma}\label{asymptotic}
Assume that $(K1)$-$(K5)$ hold, and let
\[
f\in L^{1}(\R^{d})\cap L^{\infty}(\R^{d}\setminus B_1(0)) ,\,\omega(1+|x|)\, f(x)\in L^{1}(\R^{d}).
\]
Then
\[
\frac{\psi_{f}(x)}{W(x)}\rightarrow\int_{\R^{d}}f(y)dy\quad\text{as }|x|\rightarrow+\infty.
\]
\end{lemma}
\begin{proof}
As in Chae-Tarantello \cite{CT}, we first set
\begin{equation}
\label{eq:sigma}
\sigma(x):= \frac{\psi_{f}(x)}{\omega(|x|)}-\int_{\R^{d}}f(y)dy = \frac{1}{\omega(|x|)}\int_{\R^{d}}\left[\omega(|x-y|)-\omega(|x|)\right]f(y)dy
\end{equation}
so that our aim will be to show that $\sigma(x)\rightarrow0$ as $|x|\rightarrow\infty$. Assume that $|x|>2$.
We then write
\[
\sigma(x)=\sigma_{1}(x)+\sigma_{2}(x)+\sigma_{3}(x),
\]
where $\sigma_{i}$, $i=1,2,3$, are defined by breaking the integral on the right hand side of \eqref{eq:sigma} into:
$$
D_{1}=\left\{y:|x-y|<1\right\} \, , \, D_{2}=\left\{y:|x-y|>1,\,|y|\leq R\right\}  \mbox{ and } D_{3}=\left\{y:|x-y|>1,\,|y|>R\right\}
$$
respectively, where $R>2$
is a fixed constant. Recall that $(K2)$ implies
$|\omega(r)|\leq C \phi(r)$ for $r\leq1$,
with $\phi$ given in \eqref{defphi}. Thus, we have
\begin{align*}
|\sigma_{1}(x)|&\leq\frac{1}{\omega(|x|)}\int_{|x-y|<1}\left|\omega(|x-y|)-\omega(|x|)\right||f(y)|dy\\
&\leq \frac{C}{\omega(|x|)}\int_{|x-y|<1}\phi(|x-y|)\,|f(y)|dy+\int_{|y|>|x|-1}|f|dy\\
& \leq \frac{C\|f\|_{L^{\infty}(\R^{d}\setminus B_1(0))} \|\phi\|_{L^1(B_1(0))}}{\omega(|x|)}+\int_{|y|>|x|-1}|f|dy\,,
\end{align*}
where we used $f\in L^{\infty}(\R^{d}\setminus B_1(0))$ and $|x|>2$ in the last inequality.
This means that $\sigma_{1}(x)\rightarrow0$ as $|x|\rightarrow\infty$. Moreover, we notice
that
\[
|\sigma_{2}(x)|\leq\frac{1}{\omega(|x|)}\int_{\left\{y\in\R^{d}:|x-y|>1,\,|y|<R\right\}}\left|\omega(|x-y|)-\omega(|x|)\right||f(y)|dy\,.
\]
Since by property $(K3)$ we can estimate in the region $D_{2}$
\[
\left|\omega(|x-y|)-\omega(|x|)\right|\leq C \big||x-y|-|x|\big|\leq C|y|\leq CR\,,
\]
such that
\[
|\sigma_{2}(x)|\leq C\frac{R}{\omega(|x|)}\|f\|_{L^{1}(\R^{d})},
\]
which implies that also $\sigma_{2}(x)\rightarrow0$ as $|x|\rightarrow+\infty$. As for $\sigma_3$, for $x$ such that $|x|>R$, using (K4)-(K5) we write
\begin{align*}
|\sigma_{3}(x)|\leq&\, \frac{1}{\omega(|x|)}\int_{\left\{y\in\R^{d}:|x-y|>1,\,R<|y|<|x|\right\}}\left|\omega(|x-y|)-\omega(|x|)\right||f(y)|dy\\
&+\frac{1}{\omega(|x|)}\int_{\left\{y\in\R^{d}:|x-y|>1,\,|y|>|x|\right\}}\omega(|x-y|)|f(y)|dy+\int_{|y|>R}|f(y)|dy\\
\leq&\, \frac{\omega(2|x|)}{\omega(|x|)}\int_{|y|>R}|f(y)|dy+2\int_{|y|>R}|f(y)|dy\\
&+\frac{C_w}{\omega(|x|)}\int_{\left\{y\in\R^{d}:|x-y|>1,\,|y|>|x|\right\}}
\left[1+\omega(1+|x|)+\omega(1+|y|)\right]|f(y)|dy\\
\leq&\, C\left(1+\frac{1}{\omega(|x|)}\right)\int_{|y|>R}|f(y)|dy+\frac{1}{\omega(|x|)}\int_{\R^d}\omega(1+|y|)|f(y)|dy\\
&\rightarrow  C \int_{|y|>R}|f(y)|dy
\end{align*}
as $|x|\rightarrow+\infty$, for any fixed $R>1$.
Hence letting $R\rightarrow+\infty$ we get $\sigma_{3}(x)\rightarrow0$.
\end{proof}

In case of assumption (K6), we prove the following Lemma.

\begin{lemma}\label{asymptotic2}
Assume \eqref{DDR}, $(K1)$-$(K4)$ and $(K6)$ hold, and let $\mathcal{K} := \ell-W$ be as defined in $(K6)$. Then the following holds for any radially decreasing $ f\in L^1_+(\R^{d})$:
\begin{equation}
\label{eq:temp1}
\lim_{|x|\to\infty} \int_{\R^{d}}\mathcal{K}(x-y)f(y)dy =0,
\end{equation}
and
\begin{equation}
\label{eq:temp2}
\int_{\R^{d}}\mathcal{K}(x-y)f(y)dy \geq c \mathcal{K}(x) \text{ for all }|x|>1,
\end{equation}
where $c:= 2^{-\alpha}  \int_{B_1(0)} f(y)dy>0$, with $\alpha>0$ as given in $(K6)$.
\end{lemma}
\begin{proof}
Since both $f$ and $\mathcal{K}$ are radially symmetric,  we define $\bar f$, $\bar{\mathcal{K}}: [0,+\infty)\to \mathbb{R}$ such that $\bar f(|x|) = f(x)$, $\bar{\mathcal{K}}(|x|)=\mathcal{K}(x)$. Note that $\lim_{r\to\infty} \bar f(r) = \lim_{r\to\infty} \bar{\mathcal{K}}(r) = 0$ due to (K1), (K6) and the assumption on $f$. To prove \eqref{eq:temp1}, we break $\int_{\R^{d}}\mathcal{K}(x-y)f(y)dy$ into the following three parts with $|x|>1$ and control them respectively by:
\[
\int_{|y|>\frac{|x|}{2}, |x-y|\leq 1}\mathcal{K}(x-y)f(y)dy \leq \|\mathcal{K}\|_{L^1(B_1(0))} \bar{f}(|x|-1),
\]
\[
\int_{|y|>\frac{|x|}{2}, |x-y|> 1}\mathcal{K}(x-y)f(y)dy \leq \bar{\mathcal{K}}(1) \int_{|y|>\frac{|x|}{2}} f(y)dy,
\]
and
\[
\int_{|y| \leq \frac{|x|}{2}}\mathcal{K}(x-y)f(y)dy \leq \bar{\mathcal{K}}\left(\frac{|x|}{2}\right) \|f\|_{L^1}.
\]
Since all the three parts tend to 0 as $|x|\to\infty$, we obtain \eqref{eq:temp1}. To show \eqref{eq:temp2}, we use $\mathcal{K}, f \geq 0$ to estimate
\[
\begin{split}
\int_{\R^{d}}\mathcal{K}(x-y)f(y)dy &\geq \int_{|y|\leq 1}\mathcal{K}(x-y)f(y)dy \geq \bar{\mathcal{K}}(|x|+1) \int_{B_1(0)} f(y)dy \\
&\geq \left(\frac{|x|+1}{|x|}\right)^{-\alpha} \bar{\mathcal{K}}(|x|) \int_{B_1(0)} f(y)dy \geq c \mathcal{K}(x) \quad\text{ for any }|x|>1,
\end{split}
\]
where  we apply (K6) to obtain the third inequality, and in the last inequality we define $c:= 2^{-\alpha}  \int_{B_1(0)} f(y)dy>0$.
\end{proof}

Using similar arguments as in \cite{CCV}, we are able to derive the following result, which indeed gives a natural form of the Euler-Lagrange equation
associated
to the functional $\E$:
\begin{theorem}\label{idenminimth}
Assume that \eqref{DDR}, $(K1)$-$(K4)$ and either $(K5)$ or $(K6)$  hold.
Let $\rho_0\in\mathcal{Y_{M}}$ be a global minimizer of the free
energy functional $\E$. Then for some positive constant $\mathsf{D}[\rho_{0}]$, we have that
$\rho_0$ satisfies
\begin{equation}
\frac{m}{m-1}\rho_{0}^{m-1}+W\ast\rho_{0}=
\mathsf{D}[\rho_{0}]\quad a.e. \text{ in }
\mbox{supp}(\rho_{0})\label{EulerLagr}
\end{equation}
and
\begin{equation*}
\frac{m}{m-1}\rho_{0}^{m-1}+W\ast\rho_{0}\geq
\mathsf{D}[\rho_{0}]\quad a.e. \text{ outside }
\mbox{supp}(\rho_{0})
\end{equation*}
where
\[
\mathsf{D}[\rho_{0}]=\frac{2}{M}\mathsf{G}[\rho_0]+\frac{m-2}{M(m-1)}\|\rho_{0}\|_{L^m(\R^d)}^{m}.
\]
As a consequence, any global minimizer of $\E$ verifies
\begin{equation}\label{EulerLagr2}
\frac{m}{m-1}\rho_{0}^{m-1}=\left(\mathsf{D}[\rho_{0}]-W\ast\rho_{0}\right)_{+}.
\end{equation}
\end{theorem}

We now turn to show compactness of support and boundedness of the minimizers.

\begin{lemma}\label{lem:compsupp}
Assume that \eqref{DDR}, $(K1)$-$(K4)$ and either $(K5)$ or $(K6)$ hold and
let $\rho_0\in\mathcal{Y}_{M}$ be a global minimizer of the free
energy functional $\E$. Then $\rho_0$ is compactly supported.
\end{lemma}
\begin{proof}
By Theorems \ref{conccomp} and \ref{thm36}, $\rho_0$ is radially decreasing under either set of assumptions.
In addition, under the assumption (K5), Lemma \ref{asymptotic} gives that
\[
\frac{(W*\rho_0)(x)}{W(x)} \to \|\rho\|_{L^1(\R^d)} \text{ as }|x|\to \infty,
\]
hence combining this with (K5) gives us $(W*\rho_0)(x) \to +\infty$ as $|x|\to \infty$. It implies that the right hand side of \eqref{EulerLagr2} must have compact support, hence $\rho_0$ must have compact support too.

Under the assumption (K6), towards a contradiction, suppose $\rho_0$ does not have compact support. Then $\rho_0$ must be strictly positive in $\mathbb{R}^d$ since it is radially decreasing. We can then write \eqref{EulerLagr} as
\[
\frac{m}{m-1} \rho_0^{m-1} - \mathcal{K}*\rho_0 = C \quad\text{a.e. in }\mathbb{R}^d
\]
for some $C\in \mathbb{R}$, where $\mathcal{K} := \ell-W$ is as given in (K6). Indeed, $C$ must be equal to 0, since both $\rho_0(x)$ and $(\mathcal{K}*\rho_0)(x)$ tend to 0 as $|x|\to\infty$, where we used \eqref{eq:temp1} on the latter convergence. Thus
\begin{equation}
\label{eq:temp3}
\rho_0(x) = \left(\frac{m-1}{m}(\mathcal{K}*\rho_0)(x)\right)^{\tfrac{1}{m-1}} \geq \Big(\frac{m-1}{m} c \mathcal{K}(x)\Big)^{1/(m-1)} \text{ for a.e. }|x|>1,
\end{equation}
where we applied \eqref{eq:temp2} to obtain the last inequality, with $c:= 2^{-\alpha}  \int_{B_1(0)} \rho_0(y)dy>0$. Due to the assumptions \eqref{homogeneity} and $\alpha < d(m-1)$ in (K6), we have $\int_{|x|>1} \mathcal{K}(x)^{1/(m-1)}dx = +\infty$. Combining this with \eqref{eq:temp3} leads to $\rho_0 \not\in L^1(\mathbb{R}^d)$, a contradiction.
\end{proof}

\begin{lemma}\label{boundedness}Assume that \eqref{DDR}, $(K1)$-$(K4)$ and either $(K5)$ or $(K6)$ hold and
let $\rho_0\in\mathcal{Y}_{M}$ be a global minimizer of the free
energy functional $\E$. Then $\rho_0 \in L^\infty(\mathbb{R}^d)$.
\end{lemma}
\begin{proof}
By Theorem \ref{conccomp}, Theorem \ref{thm36} and Lemma  \ref{lem:compsupp}, $\rho_0$ is radially decreasing and has compact support say inside the ball $B_R(0)$. Let us first concentrate on the proof under assumption (K5). For notational simplicity in this proof, we will denote by $\|\rho_0\|_m$ the $L^m(\R^d)$-norm of $\rho_0$.

We will show that $\rho_0 \in L^\infty(\mathbb{R}^d)$ by different arguments in several cases:

{\it\bf Case A: $d\leq2$.} Since $\rho_0$ is supported in $B_R(0)$, we can then find some $C_w^1$ and $C_w^2$, such that $W \geq -C_w^1 \mathcal{N} - C_w^2$ in $B_{2R}(0)$. Hence for any $r<R$, we have
\[
-(\rho_0*W)(r) \leq -(\rho_0*(-C_w^1 \mathcal{N}- C_w^2))(r) \leq C_w^1 (\rho_0*\mathcal{N})(r) + C_w^2 \|\rho_0\|_1,
\]
thus recalling \eqref{lipspot}
\[
(\rho_0*W^{-})(r) \leq C_w^1 (\rho_0*\mathcal{N})(r) + C_w^2 \|\rho_0\|_1+ (\rho_0*W^{+})(r)\leq C_w^1 (\rho_0*\mathcal{N})(r) + C_w^2 \|\rho_0\|_1+\widetilde{C}.
\]
Then by equation \eqref{EulerLagr2} it will be enough to show that the \emph{Newtonian potential} $\rho_0*\mathcal{N}$ is bounded in $B_R(0)$ for $d=1,2$. In $d=1$, this is trivial. In $d=2$ it follows from \cite[Lemma 9.9]{Gilbarg} since we have that $\rho_0*\mathcal{N}\in\mathcal{W}^{2,m}(B_R(0))$, then Morrey's Theorem (see for instance \cite[Corollary 9.15]{BrezisFunc}) yields $\rho_0*\mathcal{N}\in L^{\infty}(B_R(0))$.

{\it\bf Case B: $d\geq 3$ and $m> d/2$.} In this case we get $W^{-}\leq C_{w}\,\mathcal{N}$ in the whole $\R^{d}$ for some constant $C_w$, so we have for $r>0$
\[
(\rho_0*W^{-})(r) \leq C_w (\rho_0*\mathcal{N})(r).
\]
Then using Sobolev's embedding theorem again (see again \cite[Corollary 9.15]{BrezisFunc}), we easily argue that for $m>d/2$ we find $(\rho_0*W^{-})(r)\in L^{\infty}_{loc}(\R^{d})$, hence $\rho_0 \in L^\infty(\mathbb{R}^d)$ by \eqref{EulerLagr2} again.

{\it\bf Case C: $d\geq 3$ and $2-\tfrac2{d}<m\leq d/2$.}
We aim to prove that $\rho_{0}(0)$ is finite which is sufficient for the boundedness of $\rho_0$ since $\rho_0$ is radially decreasing.
This is done by an inductive argument. To begin with, observe that since $\rho_0$ is radially decreasing we have that $\rho_0(r)^m |B(0,r)| \leq \|\rho_0\|_m^m < \infty,$
which leads to the basis step of our induction
\[
\rho_0(r) \leq C(d, m, \|\rho_0\|_m) r^{-d/m} \text{ for all }r>0.
\]
We set our first exponent $\tilde p =-d/m$. For the induction step, we claim that if $\rho_0(r) \leq C_1( 1+r^{p})$ with $-d<p<0$, then it leads to the refined estimate
\begin{equation}
\label{eq:refinedb}
\rho_0(r) \leq \begin{cases}
C_2 (1+r^{\frac{p+2}{m-1}}) &\text{ if }p\neq -2\\
C_2 (1+|\log r|^{\frac{1}{m-1}}) & \text{ if }p=-2,
\end{cases}
\end{equation}
where $C_2$ depends on $d,m, \rho_0, W$ and $C_1$.

Indeed, taking into account (K2) and (K5), the compact support of $\rho_0$ together with the fact that $\mathcal{N}> 0$ for $d\geq 3$, we deduce that $W \geq- C_{w,d} \,\mathcal{N}$ for some constant depending on $W$ and $d$. As a result, we have, for $r\in (0,1)$,
\begin{equation}
\label{eq:bd1}
-(\rho_0*W)(r) \leq C_{w,d} (\rho_0*\mathcal{N})(r) = C_{w,d}\left((\rho_0 * \mathcal{N})(1)- \int_r^1 \partial_r (\rho_0*\mathcal{N})(s) ds \right).
\end{equation}
We can easily bound $(\rho_0 * \mathcal{N})(1)$ by some $C(d,\|\rho_0\|_m)$. To control $ \int_r^1 \partial_r (\rho_0*\mathcal{N})(s) ds$, recall that
\begin{equation}
\label{eq:bd2}
-\partial_r (\rho_0*\mathcal{N})(s) = \frac{M(s)}{|\partial B(0,s)|} = \frac{M(s)}{\sigma_d s^{d-1}},
\end{equation} where $M(s)$ is the mass of $\rho_0$ in $B(0,s)$. By our induction assumption, we have
\[
M(s) \leq \int_0^s C_1( 1+t^{p}) \sigma_d t^{d-1} dt = C_1 \sigma_d \left( \frac{s^d}{d} + \frac{s^{d+p}}{d+p}\right).
\]
Combining this with \eqref{eq:bd2}, we have
\[
-\partial_r (\rho_0*\mathcal{N})(s) \leq C_1 \left( \frac{s}{d} + \frac{s^{1+p}}{d+p}\right),
\]
so we get, for $p\neq-2$,
\[
- \int_r^1 \partial_r (\rho_0*\mathcal{N})(s) ds\leq C_{1}\left[\frac{1}{2d}(1-r^{2})+\frac{1}{(d+p)(2+p)}(1-r^{2+p})\right]\,.
\]
Plugging it into the right hand side of \eqref{eq:bd1} yields
\[
-(\rho_0*W)(r) \leq C(d,m,\|\rho_0\|_m, C_{w,d}, C_1)(1+r^{2+p}),
\]
and using this inequality in the Euler-Lagrange Equation \eqref{EulerLagr2} leads to \eqref{eq:refinedb}.
Moreover, in the case $p=-2$, we have instead the inequality
\[
- \int_r^1 \partial_r (\rho_0*\mathcal{N})(s) ds\leq C_{1}\left[\frac{1}{2d}(1-r^{2})-\frac{1}{d-2}\log r\right].
\]

Now we are ready to apply the induction starting at $\tilde p=-d/m$ to show $\rho_0(0)<\infty$. We will show that after a finite number of iterations our induction arrives to
\begin{equation}\label{fbound}
\rho_0(r) \leq C (1+r^a)
\end{equation}
for some $a>0$, which then implies that $\rho_0(0) < \infty$. Let $g(p) := \frac{p+2}{m-1}$, which is a linear function of $p$ with positive slope, and let us denote $g^{(n)}(p)=: \underbrace{(g\circ g\dots \circ g)}_{n \text{ iterations}}(p)$.

{\it Subcase C.1: $m=d/2$.-} In this case, we have $\widetilde{p}=-2$ and by \eqref{eq:refinedb} we obtain
\[
\rho_0(r) \leq C_2 (1+|\log r|^{\frac{1}{m-1}})\leq C_2 (1+r^{-1})
\]
hence applying the first inequality in \eqref{eq:refinedb} for $p=-1$ gives us \eqref{fbound} with $a=1/(m-1)$.

Then it remains to consider the case $m<d/2$. Notice that $-d<\widetilde{p}<-2$. By \eqref{eq:refinedb} we get, for all $r\in (0,1)$,
\begin{equation}
\rho_{0}(r)\leq C_{2}(1+r^{g(\widetilde{p})}).\label{startiteration}
\end{equation}
Then we must consider three cases. We point out that in all the cases we need to discuss the possibility of  $g^{(n)}(p) = -2$ for some $n$: if this happens, the logarithmic case occurs again and the result follows in a final iteration step as in Subcase C.1.

{\it Subcase C.2: $m=2$ and $m<d/2$.-} In this case, we have $g(p) = p+2$, hence
$
g^{(n)}(p)=p+2n,
$
then
\[
\lim_{n\to\infty} g^{(n)}(p) =+\infty.
\]
Therefore we have $g^{(n)}(\widetilde{p}) > 0$ for some finite $n$, whence iterating \eqref{startiteration} $n$ times we find $\rho_0(0)<\infty$.

{\it Subcase C.3: $m>2$ and $m<d/2$.-} In this case, $p=2/(m-2)$ is the only fixed point for the linear function $g(p)$. For all $p<\tfrac{2}{m-2}$ we have
$g(p)>p$ which implies $g^{(n)}(p)>p$ for all $n\in \N$. Notice that
\begin{equation}
g^{(n)}(p)=\frac{2}{m-2}+\frac{p(m-2)-2}{(m-2)(m-1)^{n}},
\label{iteration}
\end{equation}
so the point $p=2/(m-2)$ is attracting in the sense that
\[
\lim_{n\to\infty} g^{(n)}(p) = \frac{2}{m-2}.
\]
Since $\frac{2}{m-2}>0$, it again implies that $g^{(n)}(p) > 0$ for some finite $n$. Then choosing $p=\widetilde{p}$, we have $g^{(n)}(\widetilde{p}) > 0$ for some $n$, then \eqref{startiteration} implies $\rho_0(0)<\infty$ again.

{\it Subcase C.4: $m<\min(2,d/2)$.-} In this case, the only fixed point $\frac{2}{m-2}$ is unstable, and we have $g(p)>p$ for any $p>\frac{2}{m-2}$, then by \eqref{iteration}
\[
\lim_{n\to\infty} g^{(n)}(p) = +\infty \text{ for any }p>\frac{2}{m-2}.
\]
Notice that $\widetilde{p} >\frac{2}{m-2}$, since this condition reads $m>2d/(d+2)$, a direct consequence of \eqref{DDR}. Hence we again obtain $g^{(n)}(\widetilde{p}) > 0$ for some finite $n$, which finishes the last case.

Let us finally turn back to the proof if we assume (K6) instead of (K5). Notice first that the proof of the Case C can also be done as soon as the potential $W$ satisfies the bound $W \geq - C_{w,d} (1+ \,\mathcal{N})$ for some $C_{w,d}>0$. This is trivially true regardless of the dimension if the potential satisfies (K6) instead of (K5).
\end{proof}

Finally, it is interesting to derive some \emph{regularity} properties of a minimizer $\rho_{0}$, as in \cite{CCV}. Since $W$ may not be the classical Newtonian kernel, we are led to prove a nice regularity for the $W$-potential $\psi_{\rho_{0}}(x)$ which can be transferred to $\rho_{0}$ via equation \eqref{EulerLagr} in the support of $\rho_{0}$.
Note that \eqref{EulerLagr2} ensures that $\rho_{0}$ satisfies equation
\eqref{steady} in the sense of distributions: indeed, as shown in \eqref{lipspot0}-\eqref{lipspot}, we find that $\psi_{\rho_{0}}\in\W^{1,\infty}_{loc}(\R^{d})$ thus we can take gradients on both sides of the Euler--Lagrange condition \eqref{EulerLagr2} and multiplying by $\rho$ and writing $\rho \nabla \rho^{m-1}=\tfrac{m-1}{m}\nabla \rho^{m}$ we reach \eqref{steady}.
\normalcolor
Now, using the regularity arguments of the proof of Lemma \ref{lem:regularity} again, together with the compact support property, we finally have $\rho_{0}\in \mathcal{C}^{0,\alpha}(\R^{d})$ with $\alpha=1/(m-1)$.

We can summarize all the results in this section in the following theorem.

\begin{theorem}\label{conccomp2}
In the diffusion dominated regime \eqref{DDR}, assume that  conditions $(K1)$-$(K4)$ and either $(K5)$ or $(K6)$ hold. Then for any positive mass $M$, there exists a global minimizer $\rho_{0}$ of the free energy functional $\E$  \eqref{eq:energy} defined in $\mathcal{Y}_{M}$, which is radially symmetric, decreasing, compactly supported, H\"older continuous, and a stationary solution of \eqref{aggregation} in the sense of Definition {\rm \ref{stationarystates}}.
\end{theorem}
Putting together the previous theorem with the uniqueness of radial stationary solutions for the attractive Newtonian potential proved in \cite{KY,CCV}, we obtain the following result.

\begin{corollary}\label{unique}
In the particular case of the attractive Newtonian potential $W(x)=-\mathcal{N}(x)$ modulo the addition of a constant factor, the global minimizer obtained in Theorem {\rm\ref{conccomp2}} is unique among all stationary solutions in the sense of Definition {\rm \ref{stationarystates}}.
\end{corollary}

\subsection{Some remarks about the minimization of energies with a potential term}
The aim of this subsection is to generalize the previous result of subsection \ref{Minimizsub} when dealing with free functionals involving a potential energy, namely
\begin{align*}
\E[\rho]=\frac{1}{m-1}\int_{\R^{d}}\rho^{m}\,dx+\frac{1}{2}\int_{\R^{d}}\int_{\R^{d}}W(x-y)\,\rho(x)\rho(y)\,dx\,dy\,+\int_{\R^{d}}V(x)\rho(x)\,dx,
\end{align*}
defined over the same admissible set $\mathcal{Y}_{M}$, for some $\mathcal{C}^{1}$ nonnegative radially \emph{increasing} potential $V=V(r)$, where $r=|x|$, such that
\[
\lim_{r\rightarrow+\infty}V(r)=+\infty.
\]
In this framework, the functional $\E$ might be infinite on some densities $\rho$. The presence of the confinement potential $V$ allows then to prove the following generalization of theorems {\rm\ref{conccomp}}--\ref{thm36}, where no asymptotic behavior at infinity is needed for the radial profile $\omega(r)$ of the kernel $W$:
\begin{theorem}
Assume that \eqref{DDR} and $(K1)$-$(K4)$ hold, then the conclusions of Theorem {\rm\ref{conccomp}}--\rm\ref{thm36} remain true.
\end{theorem}
\begin{proof}
We first observe that by Remark \ref{remarkV} and Lemma \ref{reversedRiesz} we can restrict to radially decreasing densities. Moreover, following the lines of the proof of Theorem \ref{thm36} we find that $\E$ is bounded from below and
\[
\E[\rho]\geq-C_{1}+C_{2}\|\rho\|^{m}_{L^{m}(\R^{d})}+\int_{\R^{d}}V(x)\rho(x)dx.
\]
This inequality easily implies the mass confinement of any minimizing sequence $\left\{\rho_{n}\right\}$, that is for some constant $C>0$
\[
\sup_{n\in \N}\int_{|x|>R}\rho_{n}(x)dx\leq\frac{C}{V(R)}
\]
for some large $R>0$. In particular, we have that the sequence $\left\{\rho_{n}\right\}$ is tight, and by Prokhorov's Theorem (see \cite[Theorem 5.1.3]{AmbrosioGigli}) we obtain that (up to subsequence) $\left\{\rho_{n}\right\}$ converges to a certain density $\rho\in L^{1}_{+}(\R^{d})\cap L^{m}(\R^{d})$, $\|\rho\|_{L^{1}(\R^{d})}=M$, with respect to the narrow topology. Then \cite[Lemma 5.1.7]{AmbrosioGigli} ensures the lower semicontinuity of the potential energies of $\left\{\rho_{n}\right\}$, that is
\[
\liminf_{n\rightarrow\infty}\int_{\R^{d}}V(x)\rho_{n}(x)dx\geq\int_{\R^{d}}V(x)\rho(x)dx.
\]
This implies that the infimum of $\E$ is achieved over a radially decreasing density $\rho\in \mathcal{Y}_{M}$. In order to check that all the global minimizers are radially decreasing, we pick any minimizer $\rho\in \mathcal{Y_{M}}$ and use Remark \ref{remarkV} and Lemma \ref{reversedRiesz} in order to see that
\[
\E[\rho]=\E[\rho^{\#}],
\]
thus
\[
\mathcal{I}[\rho]-\mathcal{I}[\rho^{\#}]=\int_{\R^{d}}V(x)(\rho^{\#}-\rho)dx\leq0
\]
then the equality case in Lemma \ref{reversedRiesz} yields the conclusion.
\end{proof}
We have the following generalization of Theorem \ref{idenminimth}:
\begin{theorem}
Assume that \eqref{DDR}, $(K1)$-$(K4)$ hold.
Let $\rho_0\in\mathcal{Y_{M}}$ be a global minimizer of the free
energy functional $\E$. Then for some positive constant $\mathsf{D}[\rho_{0}]$, we have that
$\rho_0$ satisfies
\begin{equation}
\frac{m}{m-1}\rho_{0}^{m-1}+W\ast\rho_{0}+V(x)=
\mathsf{D}[\rho_{0}]\quad a.e. \text{ in }
\mbox{supp}(\rho_{0})\label{EulerLagrV}
\end{equation}
and
\begin{equation*}
\frac{m}{m-1}\rho_{0}^{m-1}+W\ast\rho_{0}+V(x)\geq
\mathsf{D}[\rho_{0}]\quad a.e. \text{ outside }
\mbox{supp}(\rho_{0}).
\end{equation*}
As a consequence, any global minimizer of $\E$ verifies
\begin{equation*}
\frac{m}{m-1}\rho_{0}^{m-1}=\left(\mathsf{D}[\rho_{0}]-W\ast\rho_{0}-V(x)\right)_{+}.
\end{equation*}
\end{theorem}
The compactly supported property of the minimizers then follows from \eqref{EulerLagrV} and Lemmas \ref{asymptotic}--\ref{asymptotic2}. Moreover, it is straightforward to check that Lemma \ref{boundedness} continues to hold, as well as Theorem \ref{conccomp2}.


\section{Long-time Asymptotics}

We now consider the particular case of \eqref{aggregation} given by the Keller Segel model in two dimensions with nonlinear diffusion as
\begin{eqnarray}\label{KS}
\pa_t \rho&=&\Delta \rho^m-\nabla\cdot(\rho\nabla \NN \ast\r)\,,
 \end{eqnarray}
where $m>1$ and the logarithmic interaction kernel is defined as
$$
\NN(x)=-\frac{1}{2\pi}\log |x|\,.
$$
This system is also referred to as the parabolic-elliptic Keller-Segel system with nonlinear diffusion, since the attracting potential $c=\NN\ast \r$ solves the
Poisson equation $-\Delta c=\r$. It corresponds exactly to the range of diffusion dominated cases as discussed in \cite{CC} since solutions do not show blow-up
and are globally bounded. We will show based on the uniqueness part in Section 2 that not only the solutions to \eqref{KS} exist globally and are uniformly
bounded in time in $L^\infty$, but also the solutions achieve stabilization in time towards the unique stationary state for any given initial mass.

The main tool for analyzing stationary states and the existence of solutions to the evolutionary problem is again the following free energy functional
\begin{eqnarray}\label{E}
\E[\rho](t)=\int_{\mathbb{R}^2} \frac{\rho^m}{m-1}dx+\frac{1}{4\pi}\int_{\mathbb{R}^2} \int_{\mathbb{R}^2}\corr{\log} |x-y| \rho(x)\rho(y) dx\,dy\,.
\end{eqnarray}
A simple differentiation formally shows that $\E$ is decaying in time along the evolution corresponding to \eqref{KS}, namely
\[
\frac{d}{dt}\E[\rho](t)=-\D[\r](t)
\]
which gives rise to the following (free) energy - energy dissipation inequality for weak solutions
\begin{equation}\label{EI}
\E[\r](t) +\int_0^t \D[\r] d\tau \leq  \E[\r_0]
\end{equation}
for nonnegative initial data $\rho_0(x) \in L^1((1+\corr{\log}(1+|x|^2))dx)\cap L^m(\mathbb{R}^2)$. The entropy dissipation is given by
$$
\D[\r]=\int_{\mathbb{R}^2}\r|\nabla h[\r]|^2dx
\,,
$$
where here and in the following we use the notation
\begin{equation*}
h[\r]=\frac{m}{m-1}\r^{m-1}-\NN\ast \r\,.
\end{equation*}
We shall note that $h$ corresponds to $\frac{\delta \E}{\delta \r}$ and that in particular the evolutionary equation \eqref{KS} can be stated as $\pa_t \r
=
\nabla \cdot (\r \nabla h[\r])$. Thus, this equation bears the structure of being a gradient flow of the free energy functional in the sense of
probability
measures, see \cite{AGS08,BCC,BCL,CMV03} and the references therein.

We first prove the global well-posedness of weak solutions satisfying the energy inequality \eqref{EI} in the next subsection as well as global uniform in
time
estimates for the solutions. In the second subsection, we used the uniform in time estimates together with the uniqueness of the stationary states proved
in
Section 2 to derive the main result of this section regarding long time asymptotics for \eqref{KS}.


\subsection{Global well-posedness of the Cauchy problem}

In this section we analyze the existence and uniqueness of a bounded global weak solution for initial data in $L^1_{log}(\mathbb{R}^2)\cap
L^\infty(\mathbb{R}^2)$, where here and in the following we denote
$$
L^1_{log}(\mathbb{R}^2)= L^1((1+\corr{\log}(1+|x|^2))dx)\,.
$$
Assuming to have a sufficiently regular solution with the gradient of the chemotactic potential being uniformly bounded, Kowalczyk \cite{K} derived a
priori
bounds in $L^\infty$ with respect to space and time for the Keller-Segel model with nonlinear diffusion on  bounded domains. These a priori estimates have
been
improved and extended to the whole space by Calvez and Carrillo in \cite{CC}. We shall demonstrate here how these a priori estimates of \cite{CC} can be
made
rigorous when starting from an appropriately regularized equation leading to the following theorem.

\begin{theorem}[Properties of weak solutions]\label{thm.ex} For any nonnegative initial data $\r_0\in L^1_{log}(\mathbb{R}^2)\cap L^\infty(\mathbb{R}^2)$,
there
exists a unique global weak solution $\r$ to \eqref{KS}, which satisfies the energy inequality \eqref{EI} with the energy being bounded from above and
below in
the sense that
$$
\E_*\leq \E[\r](t)\leq \E[\r_0]
$$
for some (negative) constant $\E_*$. In particular $\r$ is uniformly bounded in space and time
\[\sup_{t\geq 0}\|\r(t,\cdot)\|_{L^\infty(\mathbb{R}^2)}\leq C\, ,\]
where $C$ depends only on the initial data.
Moreover the log-moment grows at most linearly in time
\[N(t)=\int_{\mathbb{R}^2}\corr{\log}(1+|x|^2)\r(t,x)dx\leq N(0)+C t\,,\]
where again $C$ depends only on the initial data.
\end{theorem}

We shall also state the existence result for radial initial data that was obtained in \cite{LY} and \cite{KY} for higher dimensions and the Newtonian potential. Similar methods can be applied in the case $d=2$ considered here:

\begin{theorem}[Properties of radial solutions]
Let $\r_0 \in L^1_{log}(\mathbb{R}^2)\cap L^\infty(\mathbb{R}^2)$ be nonnegative and radially symmetric.
\begin{itemize}
\item[(a)] Then the corresponding unique weak solution of \eqref{KS} remains radially symmetric for all $t>0$.
\item[(b)] If $\r_0$ is compactly supported, then the solution remains compactly supported for all $t>0$.
\item[(c)] If $\r_0$ is moreover monotonically decreasing, then the solution remains radially decreasing for all $t>0$.
\end{itemize}
\end{theorem}

In the remainder of this section we carry out the proof
of the existence of a bounded global weak solution to \eqref{KS} as stated in Theorem \ref{thm.ex}. We therefore introduce the following regularization of
\eqref{KS}
\begin{equation}\label{KS.e}
\pa_t \rho_\e=\Delta (\rho_\e^m + \e \r_\e) -\nabla\cdot(\rho_\e\nabla \NN_{\e} \ast\r_\e)\,,
 \end{equation}
where $m>1$ and the regularized logarithmic interaction potential is defined as
\begin{equation*}
\NN_\e(x)=-\frac{1}{4\pi}\log (|x|^2+\e^2)\,.
\end{equation*}
Moreover we have for the derivatives
$$
\nabla \NN_\e=-\frac{1}{2\pi}\frac{x}{|x|^2+\e^2}\,,\qquad \Delta \NN_\e=-\frac{1}{\pi}\frac{\e^2}{\left(|x|^2+\e^2\right)^2}=-J_\e
$$
satisfying
\[\|J_\e\|_{L^1(\R^2)}=1\,.\]

The regularization in \eqref{KS.e} was used by Bian and Liu \cite{BL}, who studied the Keller-Segel equation with nonlinear diffusion and the Newtonian potential for $d\geq3$, which has been modified accordingly for the logarithmic interaction kernel in $d=2$.
The additional linear diffusion term in \eqref{KS.e} removes the degeneracy and the regularized logarithmic potential $\NN_\e$ possesses a uniformly
bounded gradient, such that the local well posedness of (\ref{KS.e}) is a standard result for any $\e>0$. We shall note that a slightly different regularization for such
nonlinear diffusion Keller-Segel type of equations has been introduced by Sugiyama in \cite{S}, which also yields the existence and uniqueness of a global weak solution. The advantage of the regularization in \eqref{KS.e} resembling the one in \cite{BL} is the fact that the regularized problem satisfies a free energy inequality, that in the limit gives exactly \eqref{EI}, whereas in \cite{S} the dissipation term could only be retained with a factor of $3/4$.

We point out that in the case $d=2$ other a priori estimates are available than in higher space dimensions leading to a different proof for global well posedness of the Cauchy problem for \eqref{KS.e} and the limit $\e\rightarrow 0$ compared to \cite{BL}.

\subsubsection{Global well posedness of the regularized Cauchy problem}
To derive a priori estimates for the regularized problem \eqref{KS.e} we use the iterative method used by Kowalczyk \cite{K} based on employing test
functions
that are powers of $\r_{\e,k}=(\r_{\e}-k)_+$ for some $k>0$. When testing \eqref{KS.e} against $p\r_{\e,k}^{p-1}$ for any $p\geq 2$, we obtain:
\begin{align}
&\frac{d}{dt}\int_{\mathbb{R}^2}\r^{p}_{\e,k}dx\nonumber \\
&\label{est.1} = -\frac{4(p-1)}{p}\int_{\mathbb{R}^2}(m \r_\e^{m-1}+\e) |\nabla \r^{\frac{p}{2}}_{\e,k}|^2dx+ p\int_{\mathbb{R}^2}(\r_{\e,k}+k)(\nabla \NN_\e\ast
\r_\e)\cdot\nabla \r_{\e,k}^{p-1}dx\\
& \leq-\frac{4(p-1)}{p}m\int_{\mathbb{R}^2}\r_\e^{m-1} |\nabla \r^{\frac{p}{2}}_{\e,k}|^2dx
+ \int_{\mathbb{R}^2}(\nabla \NN_\e\ast \r_\e)\cdot((p-1)\nabla \r_{\e,k}^{p} + k p \nabla \r_{\e,k}^{p-1})dx\nonumber\\
&\leq-\frac{4(p-1)}{p}mk^{m-1}\|\nabla \r_{\e,k}^{\frac{p}{2}}\|^2_{L^2} + \int_{\mathbb{R}^2} J_\e\ast\r_{\e}((p-1)\r_{\e,k}^{p}+k p \r_{\e,k}^{p-1}
)dx\nonumber\\
&\leq-\frac{4(p-1)}{p}mk^{m-1}\|\nabla \r_{\e,k}^{\frac{p}{2}}\|^2_{L^2} + \int_{\mathbb{R}^2} (J_\e\ast\r_{\e,k} + k)((p-1)\r_{\e,k}^{p}+k p
\r_{\e,k}^{p-1} )
dx\nonumber\\
&\label{est.2} \leq-\frac{4(p-1)}{p}mk^{m-1}\|\nabla \r_{\e,k}^{\frac{p}{2}}\|^2_{L^2} + C(p-1) \int_{\mathbb{R}^2} \r^{p+1}_{\e,k} dx  + C k p
\int_{\mathbb{R}^2} \r_{\e,k}^p dx + k^2 p \int_{\mathbb{R}^2} \r_{\e,k}^{p-1} dx\,,
\end{align}
where for estimating the integrals involving convolution terms we used the inequality
\beq\label{est.conv}
\int f(x) (g\ast h) (x) dx \leq C \|f\|_{L^p}\|g\|_{L^q}\|h\|_{L^r},\qquad \frac{1}{p}+\frac{1}{q}+\frac{1}{r}=2, \quad\textnormal{where} \  p,q,r\geq1\,,
\eeq
see e.g. Lieb and Loss \cite{LL}.
Closing the estimate \eqref{est.2} would yield an estimate for $\r_{\e,k}$ in $L^\infty(0,T;L^p(\mathbb{R}^2))$ and thus also for $\r_\e \in
L^\infty(0,T;L^p(\mathbb{R}^2))$, since
\begin{align}
\int_{\mathbb{R}^2} \r^p_\e dx \leq&\, k^{p-1}\int_{\{\r_\e<k\}}\r_\e dx+\int_{\{\r_\e\geq k\}}(\r_\e-k)^p dx +C(p,k)\int_{\{\r_\e\geq k\}}k dx \nonumber \\
\corr{\label{rho.p}}\leq&  \int_{\mathbb{R}^2}\r_{\e,k}^p dx + (k^{p-1}+C(p,k))M\,.
\end{align}
Kowalczyk proceeded from \eqref{est.1} with the assumption corresponding to $\|\nabla \NN_\e\ast \r_{\e}\|_{L^\infty}\leq C$. Observe that it would be
sufficient
to prove $\r_\e \in L^\infty(0,T;L^p(\mathbb{R}^2))$ for some $p>2$ implying $\Delta \NN_\e\ast\r_\e \in  L^\infty(0,T;L^p(\mathbb{R}^2))$ and hence the
uniform
boundedness of the gradient term by Sobolev imbedding. Calvez and Carrillo \cite{CC} circumvent this assumption and derive the bound by using an
equi-integrability property in the inequality \eqref{est.2}.
Hence, in order to being able to follow the ideas of \cite{CC} for the regularized problem, we need to derive the corresponding energy inequality for the
latter.

\begin{proposition}\label{Prop1}
For any finite time $T>0$ the solution  $\r_\e$ to the Cauchy problem \eqref{KS.e} supplemented with initial data $\r_{0} \in L^1_{log}(\mathbb{R}^2)\cap
L^\infty(\mathbb{R}^2)$ satisfies the energy inequality
\beq\label{EED.e}
{\cal E}_\e[\r_\e](t) + \int_0^t\D_\e[\r_\e](t)dt \leq{\cal E}_\e[\r_{0}]+\e C (1+t)t\,,
\eeq
for a positive constant $C=C(M,\|\r_{0}\|_{\infty})$ and $0\leq t\leq T$, where ${\cal E}_\e$ is an approximation of the free energy functional in \eqref{E}:
$$
{\cal E}_\e[\r_\e] =\int_{\mathbb{R}^2}\left(\frac{\r_\e^m}{m-1}-\frac{\r_\e}{2}  \NN_\e\ast \r_\e \right)dx
$$
and $\D_\e$ the corresponding dissipation
$$
\D_\e[\r_\e](t)=\int_{\mathbb{R}^2}\r_\e|\nabla h_\e[\r_\e]|^2 dx\,\quad \textnormal{with} \quad h_\e[\r_\e]=\frac{m}{m-1}\nabla \r_\e^{m-1}-\nabla
\NN_\e\ast\r_\e\,.
$$
In particular, we obtain equi-integrability
\[\lim_{k\rightarrow \infty}\sup_{t\in [0,T]}\int_{\mathbb{R}^2}(\r_\e-k)_+ dx =0 \,.\]
\end{proposition}

\begin{remark}
Note that due to the $\epsilon\Delta \rho_\epsilon$ regularization term in \eqref{KS.e}, its associated energy functional actually includes an extra term $\epsilon \int \rho_\epsilon \log \rho_\epsilon$ compared to $\mathcal{E}_\epsilon$. But in this lemma we choose to obtain an energy inequality for $\mathcal{E}_\epsilon$ (rather than the actual associated energy functional), since the absence of the extra term $\epsilon \int \rho_\epsilon \log \rho_\epsilon$ will make it easier for us to obtain a priori estimates independent of $\e$ later.
\end{remark}

\begin{proof}
Testing \eqref{KS.e} with $\frac{m}{m-1}\r_\e^{m-1}-\NN_\e\ast \r_\e$ we obtain
\begin{align*}
\frac{d}{dt}{\cal E}_\e(t) + \int_{\R^{2}}\r_{\e}&\left|\frac{m}{m-1}\nabla \r_\e^{m-1}-\nabla \NN_\e\ast\r_\e\right|^2dx +\e
\frac{4}{m}\int_{\mathbb{R}^2}\left|\nabla \r_\e^{\frac{m}{2}}\right|^2dx \\
& =\e\int_{\mathbb{R}^2} \nabla \NN_\e \ast \r_\e \cdot \nabla \r_\e dx =\e\int_{\mathbb{R}^2} \r_\e (J_\e*\r_\e) dx \leq \e
\|\r_\e\|^2_{L^2(\mathbb{R}^2)}\,,
\end{align*}
where we have used \eqref{est.conv} and the fact that $\|J_\e\|_{L^1(\mathbb{R}^2)}=1$. Hence we need to derive an a priori bound for $\r_\e$ in
$L^2(\mathbb{R}^2)$.
We use the estimate \eqref{est.2} for $p=2$ and bound $\int_{\mathbb{R}^2}\r_{\e,k}^3 dx$ using the Gagliardo-Nirenberg inequality (see for instance
\cite{Gagli}, \cite{Nir}) as follows:
\begin{align*}
\|\r_{\e,k}\|^3_{L^3(\mathbb{R}^2)}&\leq C\|\nabla \r_{\e,k}\|_{L^2(\mathbb{R}^2)}^2\|\r_{\e,k}\|_{L^1(\mathbb{R}^2)}\\
&\leq
 C M\|\nabla \r_{\e,k}\|_{L^2(\mathbb{R}^2)}^2.
\end{align*}
Then by \eqref{est.2} and interpolation of the $L^2$-integral, we have
\begin{align*}
\frac{d}{dt}\int_{\mathbb{R}^2}\r^{2}_{\e,k}dx &\leq -2mk^{m-1}\|\nabla\r_{\e,k}\|^{2}_{L^{2}} + C  \int_{\mathbb{R}^2} \r_{\e,k}^3\, dx + 3k^2
\int_{\mathbb{R}^2} \r_{\e,k}dx\nonumber\\
&\leq -(2mk^{m-1}-CM)\|\nabla\r_{\e,k}\|^{2}_{L^{2}} +Ck^2M.
\end{align*}
Hence, choosing $k$ large enough, \corr{recalling $m>1$ and estimate \eqref{rho.p}, we can conclude by integrating in time that}
\[
\|\r_\e(t,\cdot)\|_{L^2(\mathbb{R}^2)}^2\leq C (1+t)\,
\]
for some constant  $C=C(M,\|\r_{0}\|_{L^\infty(\R^2)})$, which implies the stated energy inequality.

In order to obtain a priori bounds and in particular the equi-integrability property, we need to bound the energy functional also from below. The
difference to
the corresponding energy functional for the original model \eqref{KS} lies only in the regularized interaction kernel. Since clearly for all $x\in
\mathbb{R}^2$
we have $\corr{\log}(|x|^2+\e^2)\geq 2\corr{\log}|x| $, we obtain
\begin{align*}
{\cal E}_\e[\r_\e]&= \int_{\mathbb{R}^2} \frac{\r_\e^{m}}{m-1}dx + \frac{1}{8\pi}\int_{\mathbb{R}^2}  \int_{\mathbb{R}^2}  \r_\e(x) \corr{\log}(|x-y|^2+\e^2)
\r_\e(y) dx
dy\\
&\geq \int_{\mathbb{R}^2}  \frac{\r_\e^{m}}{m-1}dx + \frac{1}{4\pi}\int_{\mathbb{R}^2}  \int_{\mathbb{R}^2}  \r_\e(x) \corr{\log}|x-y| \r_\e(y) dx dy \ = \  {\cal
E}[\r_\e]
\end{align*}
Following \cite{CC} we can estimate further using  the logarithmic Hardy-Littlewood-Sobolev inequality
\begin{equation}
{\cal E}_\e[\r_{\e}]\geq {\cal E}[\r_{\e}]\geq -\frac{M}{8\pi}C(M)+\int_{\R^{2}}\Theta(\rho_{\e})\,dx\,,\label{useHLS}
\end{equation}
where $C(M)$ is a constant depending on the mass $M$ and
\[
\Theta(\rho):=\frac{\rho^{m}}{m-1}-\frac{M}{8\pi}\rho\log\rho.
\]
Now it is easy to verify there is a constant $\kappa=\kappa(m,M)>1$ for which
\[
\Theta(\rho)\geq0\quad\textnormal{ for }\rho\geq\kappa,
\]
such that
\[
\int_{\R^{2}}\Theta^{-}(\r_{\e})dx=\int_{1\leq\r_{\e}\leq \kappa} \Theta^{-}(\r_{\e})dx\leq\frac{M^{2}}{8\pi}\log\kappa\,,
\]
implying in particular
\begin{equation}\label{bound.below}
{\cal E}_\e[\r_\e]\geq {\cal E}[\r_\e]\geq -\frac{M}{8\pi}C(M)-\frac{M^2}{8\pi}\corr{\log} \kappa =:{\cal E}_*\,.
\end{equation}
We therefore find from \eqref{EED.e}, \eqref{useHLS} and \eqref{bound.below} that
\[
\int_{\R^{2}}\Theta^{+}(\r_{\e}(t))dx\leq C+\e CT^2\,,
\]
with $C =C(m,\|\rho_0\|_{L^1(\R^2)}, \|\rho_0\|_{L^\infty(\R^2)})$ being a constant independent of $t$. Since $\Theta^{+}$ is superlinear at infinity, we obtain the equi-integrability as in Theorem 5.3 in
\cite{CC}.
\end{proof}

The equi-integrability from  Proposition \ref{Prop1} allows to close the estimate \eqref{est.2} analogously to Lemma 3.1 of \cite{CC} leading to  a bound
for
$\r_{\e}$ in $L^\infty(0,T;L^p(\mathbb{R}^2))$. Moreover, using Moser's iterative methods of Lemma 3.2 in \cite{CC} we finally get a bound for $\r_{\e}$
in
$L^\infty(0,T;L^\infty(\mathbb{R}^2))$.
In order to avoid mass loss at infinity typically the boundedness of the second moment of the solution is employed. We here however demonstrate that the
bound of
the log-moment provides sufficient compactness, having the advantage of less restrictions on the initial data.  We therefore denote for the regularization
$$
N_\e(t)=\int_{\mathbb{R}^2}\corr{\log}(1+|x|^2)\r_\e(t,x) dx\,.
$$
The following lemma is now obtained following the ideas of \cite{CC}:

\begin{lemma}\label{lem.mom} The solution $\r_\e$ to \eqref{KS.e} for a nonnegative initial data $\r_{0}\in L^1_{log}(\mathbb{R}^2)\cap
L^\infty(\mathbb{R}^2)$
satisfies for any $T>0$:
\begin{equation*}
\sup_{t\in[0,T]}\|\r_\e(t,\cdot)\|_{L^\infty(\mathbb{R}^2)} +  \|N_\e(t)\|_{L^\infty(0,T)} \leq C(1+T+\epsilon T^2)\,,
\end{equation*}
where the constant $C$ depends on the initial data.
\end{lemma}
\begin{proof}
Computing formally the evolution of the log-moment in \eqref{KS.e} in a similar fashion to \cite{CC2}, we find for the test function $\phi(x)=\corr{\log}
(1+|x|^2)$
after integrating by parts
\begin{eqnarray*}
\frac{d}{dt}N_{\e}&=&\int_{\mathbb{R}^2} \pa_t\r_\e\, \phi dx =-\int_{\R^{2}}\r_\e \nabla h_\e[\r_\e]\cdot \nabla \phi dx  + \e \int_{\mathbb{R}^2}\r_\e \Delta \phi  dx
\nonumber\\
&\leq& \frac{1}{2}\int_{\R^{2}}\r_\e | \nabla \phi|^2 dx + \frac12\int_{\R^{2}}\r_\e |\nabla h_\e[\r_\e]|^2dx + \e \int_{\mathbb{R}^2}\r_\e \Delta \phi
dx\,.
\end{eqnarray*}
Computing the derivatives of $\phi$ we see
\[|\nabla \phi|=\left|\frac{2 x}{1+|x|^2}\right|\leq 1\,,\qquad |\Delta \phi| = \frac{4}{(1+|x|^2)^2}\leq 4\,.\]
We thus obtain
\begin{eqnarray*}
\frac{d}{dt}N_{\e}\leq \frac{1}{2}((1+8\e)M+  {\cal D}_\e[\r_\e])\,.
\end{eqnarray*}
Integration in time and making use of the energy - energy dissipation inequality \eqref{EED.e} and the uniform bound on ${\cal E}_\e$ from below in \eqref{bound.below} gives
\begin{eqnarray*}
N_\e(t)\leq N_\e (0)+\frac{1}{2}(1+8\e)M t + {\cal E}_\epsilon(\rho_0) - \cal E_*+\e C(1+t)t  \leq C(1+ t+\epsilon t^2)
\end{eqnarray*}
The argument can easily be made rigorous by using compactly supported approximations of $\phi$ on $\mathbb{R}^2$ as test functions, see e.g. also
\cite{BDP}. The
proof is concluded by referring to Lemma 3.2 in \cite{CC} for the proof of uniform boundedness of $\r_\e$.
\end{proof}

\begin{remark}\
\begin{itemize}
\item[(i)]
The fact that the uniform bound of $\r_\e$ grows linearly with time originates from the term of order $\e$ in the energy inequality  for the
regularized
equation. Hence the bound on the energy and therefore the modulus of equi-continuity for the regularized problem are depending on time. However, for
the
limiting equation \eqref{KS} this term vanishes and the energy is decaying for all times, which allows to deduce uniform boundedness of the solution
to
\eqref{KS} globally in time and space, see also \cite[Lemma 5.7]{CC}.
\item[(ii)] The log-moment of $\r_\e$ grows at most linearly in time. The same statement is true for the limiting function. Hence it is only possible
    to
    guarantee confinement of mass for finite times. This property allowing for compactness results will in the following be used to pass to the limit
    in the
    regularized problem. Due to the growth of the bound with time it cannot be employed for the long-time behavior. Hence different methods will be
    required.
\end{itemize}
\end{remark}

\subsubsection{The limit $\e\rightarrow 0$}
In order to deduce the global well-posedness of the Cauchy problem for \eqref{KS} it remains to carry out the limit $\e\rightarrow 0$.  Knowing that the
solution
remains uniformly bounded and having the bounds from the energy inequality, we obtain weak convergence properties of the solution. In order to pass to the
limit
with the nonlinearities and in the entropy inequality, strong convergence results will be required.
The following lemma summarizes the uniform bounds we obtain from Proposition \ref{Prop1} and Lemma \ref{lem.mom}:
\begin{lemma}\label{lem.bounds}
Let $\r_\e$ be the solution as in Proposition {\rm\ref{Prop1}}, then we obtain the following uniform in $\e$ bounds
\beqs
&&\|\r_\e\|_{L^\infty(0,T;L_{log}^1(\mathbb{R}^2))}+\|\r_\e\|_{L^\infty((0,T)\times \mathbb{R}^2)}\leq C\,,
\\
&&\|\sqrt{\r_\e}\nabla \NN_\e\ast \r_\e\|_{L^2((0,T)\times\mathbb{R}^2)} + \|\nabla \NN_\e\ast \r_\e\|_{L^\infty((0,T)\times\mathbb{R}^2)} + \sqrt{\e}\|\nabla \r_\e\|_{L^2((0,T)\times\mathbb{R}^2)} \leq C\,,
\\
&&\|\pa_t \r_\e\|_{L^2(0,T;H^{-1}(\mathbb{R}^2))}+\|\r^{q}_\e\|_{L^2(0,T;H^1(\mathbb{R}^2))}\leq C \qquad \textnormal{for any} \ \ q\geq m-\frac12\,,
\eeqs
where $C$ depends on $m, q, \rho_0$ and $T$.
\end{lemma}
\begin{proof} The uniform bounds of the $L_{log}^1(\mathbb{R}^2)$- and $L^\infty(\mathbb{R}^2)$-norms follow from the conservation of mass and Lemma \ref{lem.mom}.
The convolution term
\[\nabla \NN_\e\ast \r_\e  = -\frac{1}{2\pi}\int_{\mathbb{R}^2}\frac{y-x}{|y-x|^2+\e^2}\r_\e(t,y) dy\]
can be estimated as follows:
	\begin{equation}
2\pi |\nabla \NN_\e\ast \r_\e|\leq \|\r_\e\|_{L^\infty(\R^2)}\int_{|x-y|\leq 1}\frac{1}{|x-y|}dy + M \leq C\,.\label{estimconv}
\end{equation}
The bound of $\sqrt{\r_\e} \nabla \NN_\e\ast \r_\e$ in $L^2((0,T)\times \mathbb{R}^2)$ follows now easily by using the conservation of mass.

The basic $L^2$-estimate corresponding to \eqref{est.1} for $p=2$ and $k=0$ implies after integration in time
\beqs
\frac12\int_{\mathbb{R}^2}\r_\e^2 dx &\leq&  \frac12\int_{\mathbb{R}^2}\r_0^2 dx   -  \e \int_0^T\int_{\mathbb{R}^2} |\nabla \r_\e|^2 dxdt - m\int_0^T\int_{\mathbb{R}^2}\r_\e^{m-1} |\nabla \r_\e|^2 dxdt \\
&&\qquad + \int_0^T\int_{\mathbb{R}^2}(\corr{J_\e}\ast \r_\e) \r_\e^2 \,  dxdt \,.
\eeqs
Using the above a priori estimates we can further bound employing the inequality in \eqref{est.conv}
\beqs
\e \|\nabla \r_\e\|^2_{L^2((0,T)\times\mathbb{R}^2)}&\leq& \frac12\int_{\mathbb{R}^2}\r_0^2 dx+ C\int_0^T\int_{\mathbb{R}^2}\r_\e^3  dxdt \leq C\,.
\eeqs
Since $m>1$,  the conservation of mass and the uniform boundedness of $\r_\e$ give $\r_\e^{m-1/2}$ in $L^2((0,T)\times\mathbb{R}^2)$. For the gradient we
now use
the bound on the entropy dissipation \eqref{EED.e}
\begin{eqnarray}
\|\nabla \r^{m-1/2}_\e\|^2_{L^2((0,T)\times\mathbb{R}^2)}&\leq& 2\frac{(m-1/2)^{2}}{m^{2}}\left( \int_0^T\D_\e[\r_\e]dt + \|\sqrt{\r_\e}\nabla \NN_\e\ast
\r_\e\|^2_{L^2((0,T)\times\mathbb{R}^2)}\right)\nonumber\\
&\leq& C + C\|\nabla \NN_\e\ast \r_\e\|^{2}_{L^\infty((0,T)\times\mathbb{R}^2)}MT\label{boundpwrrho}\,.
\end{eqnarray}
The bound for $\nabla \r^q$ follows easily by rewriting
\[\frac{m-\frac12}{q}\,\nabla \r_\e^q=\r_\e^{q-m+\frac12}\nabla \r_\e^{m-\frac12}\]
and using the uniform boundedness of $\r_\e$.

It thus now remains to derive the estimate for the time derivative. Using the previous estimates  we have for any test function $\phi\in
L^2(0,T;H^1(\mathbb{R}^2))$,
\begin{eqnarray*}
&&\left|\int_0^T\!\!\int_{\mathbb{R}^2}\pa_t \r_\e \phi \, dx\,dt\right|\leq \int_0^T\int_{\mathbb{R}^2}\left|\nabla (\r_\e^m +\e\r_\e)\cdot \nabla \phi \right| \,
dx\,dt +
\int_0^T\int_{\mathbb{R}^2}\left| \r_\e (\nabla \NN_\e\ast \r_\e )\cdot \nabla \phi \right| \, dx\,dt\\
&&\quad \leq \left(\|\nabla\r_\e^m\|_{L^2((0,T)\times\mathbb{R}^2)}+ \e\|\nabla\r_\e\|_{L^2((0,T)\times\mathbb{R}^2)}\right.\\
&&\qquad\qquad\qquad\qquad\qquad\qquad\qquad
\left.+\|\sqrt{\r_\e}|\nabla
\NN_\e\ast\r_\e|\|_{L^\infty((0,T)\times\mathbb{R}^2)}\sqrt{TM}\right)\|\nabla \phi\|_{L^2((0,T)\times\mathbb{R}^2)} \\
&&\quad \leq C\|\nabla \phi\|_{L^2((0,T)\times\mathbb{R}^2)}\,\,.
\end{eqnarray*}
\end{proof}

We now use these bounds to derive weak convergence properties. The Dubinskii Lemma (see Lemma \ref{lem-Dub} in the Appendix) can be applied to obtain the strong convergence
locally in
space, which can be extended to global strong convergence using the boundedness of the log-moment.
\begin{lemma}\label{lem.conv} Let $\r_\e$ be the solution as in Proposition {\rm\ref{Prop1}}. Then, up to a subsequence,
\begin{eqnarray}
\label{conv.1}\r_\e\ &\rightarrow & \ \r \qquad  \ \ \textnormal{in} \quad L^{q}((0,T)\times\mathbb{R}^2)\quad \textnormal{for any} \ \ 1\leq
q<\infty\,,\\
\label{conv.2}\r^{p}_{\e}\ &\rightharpoonup & \ \r^{p} \ \qquad  \textnormal{in} \quad L^2(0,T;H^1(\mathbb{R}^2))\,\quad \textnormal{for any} \ \
m-\frac{1}{2}\leq p<\infty,\\
\label{conv.3}\sqrt{\r_\e}\ &\rightarrow & \ \sqrt{\r} \qquad  \textnormal{in} \quad \ L^2((0,T)\times \mathbb{R}^2)\,,\\
\label{conv.4}\e \nabla \r_\e\ &\rightarrow &\  0 \qquad  \textnormal{in} \quad \ L^2((0,T)\times \mathbb{R}^2;\R^2)
\end{eqnarray}
\end{lemma}
\begin{proof}
Since $\{\r_\e\}_\e$ are uniformly bounded in $L^q((0,T)\times \mathbb{R}^2)$ for any $1\leq q\leq \infty$, we obtain from the reflexivity of the Lebesgue
spaces
for $1<q<\infty$\corr{, up to a subsequence,} the weak convergence
\beq
\label{weak.limit}
\r_\e\rightharpoonup \r \qquad \textnormal{in} \ \ L^q((0,T)\times\mathbb{R}^2) \quad \textnormal{for any } \ 1<q<\infty\,.\eeq
Moreover due to the uniform bounds from  Lemma \ref{lem.bounds}
\[\|\pa_t \r_\e\|_{L^2(0,T;H^{-1}(\mathbb{R}^2))}+\|\r^{r}_\e\|_{L^2(0,T;H^1(\mathbb{R}^2))}\leq C
\]
for any $r\geq m-\frac12$, we can apply the Dubinskii Lemma stated in the Appendix to derive
\[\r_\e\rightarrow \r  \qquad  \textnormal{in} \quad L^{r}((0,T)\times B_{R}(0))\qquad \textnormal{for any} \ \   2m\leq r<\infty\
\textnormal { and any }R>0
\,.\]
The boundedness of the log-moment $N(t)$ allows to extend the strong convergence to the whole space, since for any $1\leq q<\infty$  we have
$$
\int_0^T\!\!\!\int_{|x|>R} \r_\e^q dx dt \leq \|\r_\e\|^{q-1}_{L^\infty((0,T)\times\mathbb{R}^2)}\int_0^T\!\!\!\int_{|x|>R}\frac{\corr{\log}(1+|x|^2)}{\corr{\log}
(1+R^2)}\r_\e
dx dt\leq \frac{C(1+T)}{\corr{\log} (1+R^2)} \rightarrow 0 \,,
$$
as $R\rightarrow \infty$. Due to the weak lower semi-continuity of the $L^q$-norm we can now conclude with \eqref{weak.limit} that also
\[\int_{|x|>R}\r^q(t,x) dx\leq\liminf_{\e>0}\int_{|x|>R}\r_\e^q(t,x) dx \rightarrow 0 \qquad \textnormal{as} \  R\rightarrow \infty \quad \textnormal{for
all} \
q\geq 1\,.\]
Hence we can extend the strong convergence locally in space to strong convergence in $\mathbb{R}^2$:
\[\r_\e\rightarrow \r  \qquad  \textnormal{in} \quad L^{r}((0,T)\times\mathbb{R}^2)\qquad \textnormal{for any} \ \   2m\leq r<\infty\,.\]
Additionally the strong convergence in $L^1((0,T)\times \mathbb{R}^2)$ can be deduced using the bound from the energy as stated in Lemma \ref{lem.A.1} in
the Appendix. Interpolation now yields \eqref{conv.1}.

The weak convergence of $ \r_\e^{m-1/2}$ in $L^2(0,T;H^1(\mathbb{R}^2))$ holds due to its uniform boundedness given by inequality \eqref{boundpwrrho} and
the
reflexivity of the latter space, where the limit is identified arguing by the density of spaces. Due to the uniform boundedness of $\r_\e$ this assertion
can be
extended to any finite power bigger than $m-1/2$.

Since moreover $\sqrt{\r_\e}$ is uniformly bounded in $L^2((0,T)\times \mathbb{R}^2)$ we have the weak convergence towards $\sqrt{\r}$ in $L^2((0,T)\times
\mathbb{R}^2)$, where again the limit is identified by using the a.e. convergence of $\r_\e$  from the strong convergence above. To see \eqref{conv.3} we
rewrite
\begin{eqnarray*}
\|\sqrt{\r_\e}-\sqrt{\r}\|^2_{L^2((0,T)\times \mathbb{R}^2)}&=&\int_0^T\int_{\mathbb{R}^2}(\r_\e-2\sqrt{\r_\e}\,\sqrt{\r} +\r )dx\,dt\\
&=&\int_0^T\int_{\mathbb{R}^2}(\r_\e-\r)dx\,dt - 2\int_0^T\int_{\mathbb{R}^2}\sqrt{\r}\,(\sqrt{\r_\e} -\sqrt{\r} )dx\,dt\,.
\end{eqnarray*}
The first integral vanishes and the second one converges to $0$ due to the weak convergence of $\sqrt{\r_\e}\rightharpoonup \sqrt{\r}$ in $L^2((0,T)\times
\mathbb{R}^2)$.

Finally the convergence in \eqref{conv.4} is a direct consequence of the bound $\sqrt{\e}\|\nabla \r_\e\|_{L^2((0,T)\times \mathbb{R}^2)}\leq C$ in Lemma \ref{lem.bounds}.
\end{proof}

These convergence results from Lemma \ref{lem.conv}  are sufficient to obtain the weak convergence of the nonlinearities $\sqrt{\r_\e}\nabla h_\e[\r_\e]$
and
$\r_\e\nabla h_\e[\r_\e]$ in $L^2((0,T)\times \mathbb{R}^2)$, which allow to pass to the limit in the weak formulation and to deduce the weak lower
semicontinuity of the entropy dissipation term:
\begin{lemma}\label{lem.diss} Let $\r_\e$ and $\r$ be as in Lemma {\rm\ref{lem.conv}}. Then
\beq
\label{conv.diss.e}\sqrt{\r_\e}\,\nabla h_\e[\r_\e]\rightharpoonup \sqrt{\r}\,\nabla h[\r] \qquad \textnormal{in} \ L^2((0,T)\times
\mathbb{R}^2;\R^{2})\,\\
\label{conv.nonl.e}\r_\e\,\nabla h_\e[\r_\e]\rightharpoonup \r\,\nabla h[\r] \qquad \textnormal{in} \ L^2((0,T)\times \mathbb{R}^2;\R^{2})\,.
\eeq
\end{lemma}
\begin{proof}
Due to \eqref{conv.2} and \eqref{conv.3} it remains to verify
\[\sqrt{\r_\e}\,\nabla \NN_\e\ast\r_\e \rightharpoonup  \sqrt{\r}\,\nabla \NN\ast \r   \quad \ \textnormal{in} \quad L^2((0,T)\times
\mathbb{R}^2;\R^{2})\,.\]
Due to Lemma \ref{lem.bounds}, we have the weak convergence of $\sqrt{\r_\e}\,\nabla  \NN_\e\ast\r_\e$ in $L^2((0,T)\times \mathbb{R}^2;\R^{2})$. In order
to
identify the limit we consider for a $\phi \in L^2((0,T)\times \mathbb{R}^2;\R^{2})$:
\begin{eqnarray}
&&\int_0^T\int_{\mathbb{R}^2}(\sqrt{\r_\e}\,\nabla \NN_\e\ast\r_\e-\sqrt{\r}\,\nabla \NN \ast\r)\cdot\phi  \, dx dt \nonumber\\
&&=\int_0^T\int_{\mathbb{R}^2}(\sqrt{\r_\e}-\sqrt{\r})(\nabla  \NN_\e\ast\r_\e)\cdot\phi\, dx dt +
\int_0^T\int_{\mathbb{R}^2}\sqrt{\r}(\nabla ( \NN_\e- \NN)\ast\r_\e)\cdot\phi\, dx dt\nonumber\\
\label{int.diss.e}&&\quad +\int_0^T\int_{\mathbb{R}^2}\sqrt{\r}\,\nabla  \NN\ast(\r_\e-\r)\cdot\phi\, dx dt  \end{eqnarray}
The first term converges to zero using \eqref{conv.3}, since by \eqref{estimconv} it is bounded by
\[\|\sqrt{\r_\e}-\sqrt{\r}\|_{L^2((0,T)\times\mathbb{R}^2)}\|\nabla
\NN_\e\ast\r_\e\|_{L^\infty((0,T)\times\mathbb{R}^2)}\|\phi\|_{L^2((0,T)\times\mathbb{R}^2)}\leq
C\|\sqrt{\r_\e}-\sqrt{\r}\|_{L^2((0,T)\times\mathbb{R}^2)}\rightarrow 0\,.\]
For the second term we first use the Cauchy-Schwarz inequality
\beqs
&&\int_0^T\int_{\mathbb{R}^2}\sqrt{\r}(\nabla  (\NN_\e- \NN)\ast\r_\e)\cdot\phi\, dx dt \leq \sqrt{MT}\|\phi\|_{L^2((0,T)\times\mathbb{R}^2)}\|\nabla
(\NN_\e-\NN)\ast\r_\e\|_{L^\infty((0,T)\times\mathbb{R}^2)}\,.
\eeqs
To see that this convolution term vanishes we bound further
\beqs
&&|\nabla (\NN_\e-\NN)\ast \r_\e|=\left|\int_{\mathbb{R}^2}\left(\frac{x-y}{|x-y|^2+\e^2}-\frac{x-y}{|x-y|^2}\right)\r_\e(y)dy\right|\\
&&\quad \leq \e^2\|\r_\e\|_{L^\infty(\mathbb{R}^2)}\int_{\mathbb{R}^2}\frac{|x-y|}{(|x-y|^2+\e^2)|x-y|^2}dy=\e C\int_0^\infty\frac{1}{s^2+1}ds \leq \e
C\rightarrow 0
\eeqs
uniformly in $x,t$, where we substituted $s = |x-y|/\e$.
For the remaining term in \eqref{int.diss.e} we proceed changing the order of integration, where we again skip the dependence of $\r_\e$ and $\phi$ on $t$
in the
following:
\begin{eqnarray*}
&&\int_0^T\int_{\mathbb{R}^2}\sqrt{\r}\,(\nabla  \NN\ast(\r_\e-\r))\cdot\phi \, dx dt \\
&&=\frac{1}{2\pi}\int_0^T\int_{\mathbb{R}^2}\left(\sqrt{\r_\e(y)}-\sqrt{\r(y)}\right) \left(\sqrt{\r_\e(y)}+\sqrt{\r(y)}\right)\left(\int_{\mathbb{R}^2}
\sqrt{\r(x)}\frac{x-y}{|x-y|^2}\cdot\phi(x) dx\right) dydt\\
&&\leq\frac{1}{2\pi}\left\|\sqrt{\r_\e}-\sqrt{\r}\right\|_{L^2((0,T)\times\mathbb{R}^2)}\left\|
\left(\sqrt{\r_\e(\cdot)}+\sqrt{\r(\cdot)}\right)\left(\int_{\mathbb{R}^2} \sqrt{\r(x)}\frac{1}{|x-\cdot|}|\phi(x)| dx\right)
\right\|_{L^2((0,T)\times\mathbb{R}^2)}\\
\end{eqnarray*}
 To prove that this integral vanishes in the limit, due to \eqref{conv.3} it suffices to show that
\[\int_0^T \int_{\mathbb{R}^2} \left(\left(\sqrt{\r_\e(y)}+\sqrt{\r(y)}\right)\int_{\mathbb{R}^2} \sqrt{\r(x)}\frac{1}{|x-y|}|\phi(x)| dx \right)^2 dydt
\leq C
\,.\]
We shall therefore split the integral into two parts and consider first
\begin{align*}
&\int_0^T \int_{\mathbb{R}^2} \left(\left(\sqrt{\r_\e(y)}+\sqrt{\r(y)}\right)\int_{|x-y|\leq 1} \sqrt{\r(x)}\frac{1}{|x-y|}|\phi(x)| dx \right)^2 dydt\\
& \leq 2 \int_0^T \int_{\mathbb{R}^2} (\r_\e(y)+\r(y))\left(\int_{|x-y|\leq1}|\phi|^2(x) \frac{1}{|x-y|}dx\right)\left( \int_{|x-y|\leq1}\r(x)
\frac{1}{|x-y|}dx\right)\, dy dt\\
& \leq C\|\r\|_{L^\infty((0,T)\times\mathbb{R}^2)}\left(\|\r\|_{L^\infty((0,T)\times\mathbb{R}^2)}+\|\r_\e\|_{L^\infty((0,T)\times\mathbb{R}^2)}\right)
 \int_0^T \int_{\mathbb{R}^2}\int_{|x-y|\leq1} \frac{|\phi|^2(x)}{|x-y|}dy dx dt \\
& \leq C
 \int_0^T \int_{\mathbb{R}^2}|\phi|^{2}(x)\left(\int_{|x-y|\leq1} \frac{1}{|x-y|}dy\right) dx dt \ \leq \ C\|\phi\|_{L^2((0,T)\times\mathbb{R}^2)}^2
\end{align*}
It remains to bound the integral for $|x-y|>1$:
\begin{align*}
&\int_0^T \int_{\mathbb{R}^2} \left(\left(\sqrt{\r_\e(y)}+\sqrt{\r(y)}\right)\int_{|x-y|> 1} \sqrt{\r(x)}\frac{1}{|x-y|}|\phi(x)| dx \right)^2 dydt\\
&\leq 2 \int_0^T\|\phi\|_{L^2(\mathbb{R}^2)}^2 \int_{\R^{2}} \left(\r_\e(y)+\r(y)\right)  \int_{|x-y|>1}\r(x) \frac{1}{|x-y|^2}dxdy dt\\
&\leq 2M \int_0^T\|\phi\|_{L^2(\mathbb{R}^2)}^2\int_{\R^{2}} \left(\r_\e(y)+\r(y)\right) dydt\ = \ 4M^{2}\|\phi\|^2_{L^{2}((0,T)\times\R^{2})}.
\end{align*}
\end{proof}

\begin{proof}[Proof of Theorem {\rm\ref{thm.ex}}] The convergence property of the nonlinearity in \eqref{conv.nonl.e} and the weak convergence of the time
derivative
due to Lemma \ref{lem.bounds} allow to pass to the limit in the weak formulation of the Cauchy problem for \eqref{KS}, where the linear diffusion term vanishes due to \eqref{conv.4}. The uniqueness of the solution
is implied from Theorem 1.3 and Corollary 6.1 of \cite{CLM}, where we shall not go further into detail here.

It thus remains to pass to the limit in the energy inequality. Since the energy dissipation is weakly lower semicontinuous due to \eqref{conv.diss.e}, we
get
\beqs
\int_0^T\D[\r](t)dt\leq \liminf_{\e>0}\int_0^T\D_\e[\r_\e](t)dt\,.
\eeqs
In order to obtain the energy inequality \eqref{EI} in the limit $\e\rightarrow 0$ it thus remains to show
${\cal E}_\e[\r_\e](t) \rightarrow {\cal E}[\r](t)\,$ for $t\in[0,T]$.
Lemma \ref{lem.A.1} and the uniform bounds on $\r_\e$ in Lemma \ref{lem.bounds} directly imply the strong convergence of $\r_\e$ in
$L^\infty(0,T;L^m(\mathbb{R}^2))$. It is therefore left to prove the convergence for the convolution term and we rewrite
\beqs
&&-4\pi \int_{\R^2} (\r_\e \NN_\e\ast \r_\e-\r \NN \ast \r ) dx
=\int_{\mathbb{R}^2}\int_{\mathbb{R}^2}\r_\e(x)\r_\e(y)\corr{\log}\frac{|x-y|^2+\e^2}{|x-y|^2} dxdy\\
&&\qquad \qquad+ 2\int_{\mathbb{R}^2}\int_{\mathbb{R}^2}\left(\r_\e(x)(\r_\e(y)-\r(y))+\r(y)(\r_\e(x)-\r(x))\right)\corr{\log} |x-y| dxdy\,.
\eeqs
We split the domain of integration and first analyze the case $|x-y|\geq 1$. In this domain, we get
$$
\int_{\mathbb{R}^2}\int_{|x-y|\geq 1}\!\!\!\!\!\!\r_\e(x)\r_\e(y)\corr{\log} \frac{|x-y|^2+\e^2}{|x-y|^2} dxdy\leq \int_{\mathbb{R}^2}\int_{|x-y|\geq 1}\!\!\!\!\!\!\r_\e(x)\r_\e(y)\corr{\log}
(1+\e^2)
dxdy\leq \e^2M^2\,,
$$
and thus it converges to zero as $\e\to 0$. Using the Cauchy-Schwarz inequality, we obtain moreover
\beqs
&&\left(\int_{\mathbb{R}^2}\int_{|x-y|\geq 1}\r_\e(x)(\r_\e(y)-\r(y))\corr{\log} |x-y| dxdy\right)^2\\
&&\quad \leq \|\r_\e-\r\|_{L^1(\mathbb{R}^2)}\int_{\mathbb{R}^2}
|\r_\e(y)-\r(y)|\left|\int_{|x-y|\geq 1}\r_\e(x)\corr{\log}|x-y| dx\right|^2dy\\
&&\quad \leq 2M\|\r_\e-\r\|_{L^1(\mathbb{R}^2)}\int_{\mathbb{R}^2}\int_{|x-y|\geq 1}(\corr{\log}(1+|x|)+\corr{\log}(1+|y|))\r_\e(x)(\r_\e(y)+\r(y))dxdy\\
&&\quad \leq 4M^2(N(t)+N_\e(t))\|\r_\e-\r\|_{L^1(\mathbb{R}^2)}\leq C(1+T)\|\r_\e-\r\|_{L^{\infty}(0,T;L^1(\mathbb{R}^2))}\rightarrow 0
\eeqs
We now turn to the integration domain $|x-y|<1$, \corr{where by dominated convergence}
\beqs
\int_{\mathbb{R}^2}\int_{|x-y|< 1}\r_\e(x)\r_\e(y)\corr{\log}\frac{|x-y|^2+\e^2}{|x-y|^2} dx dy &\leq& \|\r_\e\|_{L^\infty(\mathbb{R}^2)}
\int_{\mathbb{R}^2}\r_\e(y)\int_0^1 r\corr{\log}\frac{r^2+\e^2}{r^2} drdy\\
&\leq& CM\int_0^1 r\corr{\log}\frac{r^2+\e^2}{r^2} dr \corr{ \  \rightarrow \  0. }
\eeqs
This proves the convergence of the entropy, which together with the weak lower semicontinuity of the entropy-dissipation leads to the desired
energy-energy
dissipation inequality \eqref{EI} for the limiting solution $\r$.
\end{proof}


\subsection{Long-Time Behavior of Solutions}
Our main result of Section 2 together with the uniqueness argument for radial stationary solutions to \eqref{KS} of \cite{KY} and the characterization of global minimizers in \cite{CCV} and Corollary \ref{unique} leads to the following result:

\begin{theorem}\label{thm.stst.long} There exists a unique stationary state $\r_M$ of \eqref{KS} with mass $M$ and zero center of mass in the sense of Definition {\rm\ref{stationarystates}}
with the property $\r_M\in L^1_{log}(\R^2)$.  Moreover, $\r_M$ is compactly supported, bounded, radially symmetric and non-increasing. Moreover, the unique stationary state is characterized as the
unique global minimizer of the free energy functional \eqref{E} with mass $M$.
\end{theorem}

As a consequence, all stationary states of \eqref{KS} in the sense of Definition \ref{stationarystates} with mass $M$  are given by translations of the given profile $\r_M$:
$$
\mathcal{S}=\left\{\r_M(x-x_0) \mbox{ such that } x_0\in\R^2 \right\} \,.
$$

\begin{remark}\label{rmk:support}
As in \cite[Corollary 2.3]{KY} we have the following result comparing the support and height for stationary states with different masses based on a
scaling
argument: Let $\r_1$ be the radial solution with unit mass. Then the radial solution with mass $M$ is of the form
$$
\r_M(x) = M^{\frac{1}{m-1}}\r_1(M^{-\frac{m-2}{2(m-1)}}x)\,.
$$
For two stationary states $\r_{M_1}$ and $\r_{M_2}$ with masses $M_1>M_2$ the following relations hold:
\begin{itemize}
\item[(a)] If $m>2$, then $\r_{M_1}$ has a bigger support and a bigger height than $\r_{M_2}$.
\item[(b)] If $m=2$, then all stationary states have the same support.
\item[(c)] If $1<m<2$, then $\r_{M_1}$ has smaller support and bigger height than $\r_{M_2}$.
\end{itemize}
\end{remark}

We will study now the long time asymptotics for the global weak solutions $\r$ of \eqref{KS} that according to the entropy inequality in Theorem \ref{thm.ex} satisfy
$$
\lim_{t\to\infty}\E[\r](t)+\int_0^\infty \D[\r](t) dt \leq \E[\r_0]\,.
$$
Since the entropy is bounded from below, this implies for the entropy dissipation
\[\lim_{t\rightarrow\infty}\int_t^\infty \D[\r](s)ds =0 \,.\]
Let us therefore now consider the sequence
\[\r_k(t,x)=\r(t+t_k,x)\qquad \textnormal{on} \ (0,T)\times\mathbb{R}^2 \quad \textnormal{for some } \ t_k\rightarrow \infty\,, \]
for which we obtain
\begin{eqnarray*}
0&=&\lim_{k\rightarrow\infty}\int_{t_k}^\infty \D[\r](t) dt
\geq \lim_{k\rightarrow \infty}\int_0^T\D[\r](t+t_k) dt\geq0\,.\end{eqnarray*}
Thus $ \D[\r_k]\rightarrow 0$ in $L^1(0,T)$, or equivalently
\begin{equation*}
\|\sqrt{\r_k}\left|\nabla h[\r_k]\right|\|^2_{L^2((0,T)\times\mathbb{R}^2)}\rightarrow 0 \qquad \textnormal{as} \  k\rightarrow \infty\,.
\end{equation*}
The proof of convergence towards the steady state will be based on weak lower semicontinuity of the entropy dissipation. Assume that $\r_k\rightharpoonup
\overline{\rho} $ in $L^\infty(0,T;L^1(\mathbb{R}^2)\cap L^m(\mathbb{R}^2))$, then we have to derive
\[\|\sqrt{\overline{\rho}}|\nabla h[\overline{\rho}]|\|_{L^2((0,T)\times\mathbb{R}^2)}\leq \liminf_{k\rightarrow \infty}\|\sqrt{\r_k}|\nabla
h[\r_k]|\|_{L^2((0,T)\times\mathbb{R}^2)}=0\,.\]
Since the $L^2$-norm is weakly lower semicontinuous, it therefore remains to show similarly as in Lemma \ref{lem.diss}
\[\sqrt{\r_k}\nabla h[\r_k]\rightharpoonup \sqrt{\overline{\rho}}\nabla h[\overline{\rho}]\,\qquad \textnormal{in} \ L^2((0,T)\times \mathbb{R}^2)\,.\]
From there it can be deduced that $\overline{\rho}$ is the stationary state $\r_M$ with $M=\|\r_0\|_{L^1(\mathbb{R}^2)}$ by the uniqueness theorem
\ref{thm.stst.long}, if we can guarantee that no mass gets lost in the limit.

The main difficulty for passing to the limit in the long-time behavior lies in obtaining sufficient compactness avoiding the loss of mass at infinity. Even though the mass of $\rho(t,\cdot)$ is conserved for all time, if a positive amount of mass escapes to infinity, then a subsequence of $\rho(t,\cdot)$ may weakly converge to a stationary solution with mass strictly less than $M$. To rule out this scenario, we need to show that the sequence $\{\rho(t,\cdot)\}_{t>0}$ is tight, which can be done by obtaining uniform-in-time bounds for certain moments for $\rho(t,\cdot)$. So far we only have a time-dependent bound on the logarithmic moment in Theorem~\ref{thm.ex}, which is not enough.  Moreover, even if we know that $\{\rho(t,\cdot)\}_{t>0}$ is tight, if we want to choose the right limiting profile among all stationary states in $\mathcal{S}$, we need to show the conservation of some symmetry. In fact, it is easy to check that the center of mass should formally be preserved by the evolution due to the antisymmetry of the gradient of the Newtonian potential. But to rigorously justify this, we need to work with moments that are larger than first moment, so the center of mass is well defined.

Below we state the main theorem in this section, where a key argument is to establish a uniform-in-time bound on the second moment of $\rho(t,\cdot)$, if $\rho_0$ has a finite second moment.

\begin{theorem}\label{thm.long2} Let $\r$ be the weak solution to \eqref{KS} given in Theorem {\rm \ref{thm.ex}} with nonnegative initial data $\r_0\in L^1((1+|x|^2)dx)\cap L^\infty(\mathbb{R}^2)$.  Then, as $t\rightarrow \infty$, $\r(\cdot,t)$ converges to the unique stationary state with the same mass and center of mass as the initial data, i.e., to
$$
\r_M^c:=\r_M(x-x_c) \qquad \mbox{where } x_c=\frac{1}{M}\int_{\R^2} x \rho_0(x) \,dx\,,
$$
with $M=\|\rho_{0}\|_{L^{1}(\R^{2})}$, ensured by Theorem {\rm\ref{thm.stst.long}}. More precisely, we have
$$
\lim_{t\to \infty}\|\r(t,\cdot)-\r_M^c(\cdot)\|_{L^q(\mathbb{R}^2)}\rightarrow 0 \qquad \textnormal{for all} \ 1\leq q<\infty\,.
$$
\end{theorem}

Our aim is to show that the second moment of solutions to \eqref{KS} is uniformly bounded in time for all $t\geq 0$. This in turn shows easily that the first moment is preserved in time for all $t\geq 0$, as we will prove below. Recall that by \eqref{secondmoment} we denote by $M_2[f]$ the second moment of $f\in L^1_+(\R^d)$.
We first derive rigorously the evolution of the second moment in time:
\beq\label{law2ndm}
M_2[\r(t,\cdot)]-M_2[\r(0,\cdot)]=4\int_0^t\int_{\mathbb{R}^2} \r^m dx \,dt-\frac{t M^2}{2\pi}\,
\eeq
starting from the regularized system \eqref{KS.e}.
Computing the second moment of the regularized problem, we obtain \corr{
\beq
\frac{d}{dt}M_{2}[\r_\e]&=&4\int_{\R^{2}}(\r_\e^m+\e\r_\e)dx-\frac{1}{\pi}\int_{\R^{2}}\int_{\R^{2}}\r_{\e}(x,t)\r_{\e}(y,t)\frac{x\cdot(x-y)}{|x-y|^{2}+\e^{2}}dx\,dy\nonumber\\
&=& 4\int_{\R^{2}}(\r_\e^m+\e\r_\e)dx-\frac{1}{2\pi}\int_{\R^{2}}\int_{\R^{2}}\r_{\e}(x,t)\r_{\e}(y,t)\frac{|x-y|^2}{|x-y|^{2}+\e^{2}}dx\,dy\,
\nonumber\\
\label{M2}&=&4\int_{\R^{2}}(\r_\e^m+\e\r_\e)dx-\frac{M^2}{2\pi}+R_\e(t)\,.
\eeq
The strong convergence in \eqref{conv.1} allows to pass to the limit $\e\to 0$ in the first integral of \eqref{M2} and for the remainder term we moreover have due to the conservation of mass and the uniform boundedness of $\r_\e$
\beqs
R_\e(t)&=&\frac{1}{2\pi}\int_{\R^{2}}\int_{\R^{2}}\r_{\e}(x,t)\r_{\e}(y,t)\frac{\e^2}{|x-y|^{2}+\e^2}dx\,dy\leq \frac{\e}{4\pi}\int_{\R^{2}}\int_{\R^{2}}\r_{\e}(x,t)\r_{\e}(y,t)\frac{1}{|x-y|}dx\,dy \\
&\leq& \e C \quad \rightarrow \quad 0.
\eeqs
}The argument can easily be made rigorous by using compactly supported approximations of $|x|^2$ on $\mathbb{R}^2$ as test functions, see e.g. also \cite{BDP}. We finally obtain \eqref{law2ndm} by integrating in time.

Now, we want to compare general solutions to \eqref{KS} with its radial solutions. In order to do this we will make use of the concept of mass concentration, which has been recalled in \eqref{Massconcdef}, and used for instance in \cite{KY,DNR} for classical applications to Keller-Segel type models.

Following exactly the same proof as in \cite{KY}, the following two results hold for the solutions of \eqref{KS}. The first result says that for two radial solutions, if one is initially ``more concentrated'' than the other one, then this property is preserved for all time. The second result compares a general (possibly non-radial) solution $\rho(t,\cdot)$ with another solution $\mu(t,\cdot)$ with initial data $\rho^\#(0,\cdot)$, i.e., the decreasing rearrangement of the initial data for $\rho(t,\cdot)$, and it says that the symmetric rearrangement of $\rho(t,\cdot)$ is always ``less concentrated'' than the radial solution $\mu(t,\cdot)$. This result generalizes  the results from \cite{DNR} to nonlinear diffusion with totally different proofs. We also refer the interested reader to the survey \cite{VANS05} for a general exposition of the mass concentration comparison results for local nonlinear parabolic equations and to the recent developments obtained in \cite{VazVol1}, \cite{VazVol2} in the context of nonlinear parabolic equations with fractional diffusion.

\begin{proposition}
\label{prop111}
Let $m>1$ and $f, g$ be two radially symmetric solutions to \eqref{KS} with $f(0,\cdot) \prec g(0,\cdot)$. Then we have $f(t,\cdot) \prec g(t,\cdot)$ for all $t>0$.
\end{proposition}

\begin{proposition}
\label{prop222}
Let $m>1$ and $\rho$ be a solution to \eqref{KS}, and let $\mu$ be a solution to \eqref{KS} with initial condition $\mu(0,\cdot) = \rho^\#(0,\cdot).$ Then we have that $\mu(t,\cdot)$ remains radially symmetric for all $t\geq 0$, and in addition we have
\[
\rho^\#(t,\cdot) \prec \mu(t,\cdot) \quad\text{ for all } t\geq 0.
\]
\end{proposition}

Now we are ready to bound the second moment of solutions in the two-dimensional case: we will show that if $\rho(t,\cdot)$ is a solution to \eqref{KS} with $M_2[\rho_0]$ finite, then $M_2[\rho(t)]$ must be uniformly bounded for all time.

\begin{theorem}\label{2ndmbound}
Let $\rho_0 \in L^1((1+|x|^2)dx) \cap L^\infty(\mathbb{R}^2)$. Let $\rho(t,\cdot)$ be the solution to \eqref{KS} with initial data $\rho_0$. Then we have that
\[
M_2[\rho(t)] \leq M_2[\rho_0] + C(\|\rho_0\|_{L^1}) \quad\text{for all }t\geq 0.
\]
\end{theorem}

\begin{proof}
Recalling that $\rho_M$ is the unique radially symmetric stationary solution with the same mass as $\rho_0$ and zero center of mass, we let $\rho_{M,\lambda} := \lambda^2 \rho_M(\lambda x)$ with some parameter $\lambda>1$. Since $\rho_0\in L^1(\mathbb{R}^2) \cap L^\infty(\mathbb{R}^2)$, we can choose a sufficiently large $\lambda$ such that $\rho_0^\# \prec \rho_{M,\lambda}$. Note that $\lambda>1$ also directly yields that $\rho_M \prec \rho_{M,\lambda}$.

Let $\mu(t,\cdot)$ be the solution to \eqref{KS} with initial data $\rho_{M,\lambda}$. Combining Proposition \ref{prop111} and Proposition \ref{prop222}, we have that
\[
\rho^\#(t,\cdot) \prec \mu(t,\cdot) \quad\text{ for all }t\geq 0.
\]
It then follows from \eqref{declp} \corr{and Lemma \ref{lemma1}} that
\begin{equation}
\label{eq:comp_lm}
\int_{\mathbb{R}^2}\rho^m(t,x) dx = \int_{\mathbb{R}^2} [\rho^\#]^m(t,x) dx \leq \int_{\mathbb{R}^2} \mu^m(t,x) dx \quad\text{ for all }t\geq 0.
\end{equation}
Now using the computation of the time derivative of $M_2[\rho(t)]$ in \eqref{law2ndm}, where $\rho(\cdot,t)$ is a solution to \eqref{KS}, we get
\begin{equation}\label{eq:dm2_dt}
M_2[\rho(t)] - M_2[\rho_0] = 4\int^t_0\int_{\mathbb{R}^2} \rho^m(t,x) dx\,dt - \frac{t M^2}{2\pi}\,.
\end{equation}
Since $\mu(t,\cdot)$ is also a solution to \eqref{KS}, \eqref{eq:dm2_dt} also holds when $\rho$ is replaced by $\mu$. Combining this fact with \eqref{eq:comp_lm}, we thus have
\begin{equation}
\label{eq:m2_diff}
M_2[\rho(t)] - M_2[\rho_0] \leq M_2[\mu(t)] - M_2[\mu(0)] \leq M_2[\mu(t)].
\end{equation}

Finally, it suffices to show $M_2[\mu(t)]$ is uniformly bounded for all time. Since $\rho_M$ is a stationary solution and we have $\rho_M \prec \rho_{M,\lambda}$, it follows from Proposition \ref{prop111} that $\rho_M \prec \mu(t,\cdot)$ for all $t\geq 0$, hence we have $M_2[\rho_M] \geq M_2[\mu(t)]$ due to Lemma \ref{lem:lp}. Plugging this into \eqref{eq:m2_diff} yields
\[
M_2[\rho(t)] \leq M_2[\rho_0] + M_2[\rho_M] \quad\text{ for all }t\geq 0,
\]
where $M_2[\rho_M]$ is a constant only depending on the mass $M:=\|\rho_0\|_{L^1(\R^2)}$, which can be computed as follows: using Remark \ref{rmk:support}, we know the support of $\rho_M$ is given by the ball centered at 0 of radius $R(M) = C_0 M^{\frac{m-2}{2(m-1)}}$ (where $C_0$ is the radius of the support for the stationary solution with unit mass), hence $M_2[\rho_M] \leq M R(M)^2 \leq C_0^2 M^{\frac{2m-3}{m-1}}$.
\end{proof}

\begin{remark}
\corr{The last result showing uniform-in-time bounds for the second moment for $m>1$ finite is also interesting in comparison to the results} in \cite{Craig,CKY} where the case $m\to\infty$ limit of the gradient flow is analysed. In the ``$m=\infty$" case, the second moment of any solution is actually decreasing in time, leading to the result that all solutions converge towards the global minimizer with some explicit rate. As mentioned in the introduction, a result of this sort for any other potential rather than the attractive logarithmic potential is lacking.
\end{remark}

As already mentioned above, a key ingredient in the proof of Theorem \ref{thm.long2} is the confinement of mass, which is first now obtained as follows:

\begin{lemma}\label{conv.weak.L1.long} Let $\r$ be a global weak solution as in Theorem {\rm\ref{thm.ex}} with mass $M$ with initial data $\rho_0 \in L^1((1+|x|^2)dx)\cap L^\infty(\mathbb{R}^2)$ and consider as above the sequence $\{\r_k\}_{k\in\mathbb{N}}=\{\r(\cdot+t_k,\cdot)\}_{k\in \mathbb{N}}$ in $(0,T)\times \mathbb{R}^2$. Then there exists a $\overline{\rho} \in
L^1((0,T)\times\mathbb{R}^2)\cap L^m((0,T)\times\mathbb{R}^2)$ and a subsequence, that we denote with the same index without loss of generality, such that:
\beqs
\r_k(t,x)\rightharpoonup \overline{\rho}(t,x) \qquad \textnormal{in} \ L^1((0,T)\times\mathbb{R}^2)\cap L^m((0,T)\times\mathbb{R}^2)
\eeqs
as $k\rightarrow \infty$.
\end{lemma}
\begin{proof}
Due to the entropy being uniformly bounded from below and by the entropy inequality \eqref{E}, we have $\r_k\in L^\infty((0,T);L^m(\mathbb{R}^2))$. Using Theorem \ref{2ndmbound}, we deduce that
\begin{equation}\label{2ndunifb}
M_2[\rho_k(t)] \leq M_2[\rho_0] + C(\|\rho_0\|_{L^1(\R^2)}) \quad\text{for all }k\in\mathbb{N} \mbox{ and } 0\leq t\leq T.
\end{equation}
Since $\{\r_k\}_{k\in\N}$ are also uniformly bounded in $L^\infty(0,T;L^m(\mathbb{R}^2))$ we obtain equi-integrability and can therefore apply the Dunford-Pettis theorem (see Theorem \ref{thm.dun} in Appendix) to obtain the weak convergence  in $L^1((0,T)\times\mathbb{R}^2)\cap L^m((0,T)\times\mathbb{R}^2)$.
\end{proof}

In order to obtain weak lower semicontinuity of the entropy dissipation term, we need additional convergence results. These are derived from the following uniform bounds:

\begin{lemma}
Let $\r$ be a global weak solution as in Theorem {\rm\ref{thm.ex}} with mass $M$ and consider as above the sequence $\{\r_k\}_{k\in\mathbb{N}}=\{\r(\cdot+t_k,\cdot)\}_{k\in \mathbb{N}}$ in $(0,T)\times \mathbb{R}^2$. Then
\beqs
&&\|\r_k\|_{L^\infty(0,T;L^1(\mathbb{R}^2))}+\|\r_k\|_{L^\infty((0,T)\times \mathbb{R}^2)}\leq C
\\
&&\|\sqrt{\r_k}\nabla \NN\ast \r_k\|_{L^2((0,T)\times\mathbb{R}^2)} + \|\nabla \NN\ast \r_k\|_{L^\infty((0,T)\times\mathbb{R}^2)} \leq C
\\
&&\|\pa_t \r_k\|_{L^2(0,T;H^{-1}(\mathbb{R}^2))}+\|\r^{q}_k\|_{L^2(0,T;H^1(\mathbb{R}^2))} \leq  C \qquad \textnormal{for any} \ \ q\geq m-\frac12\,.
\eeqs
\end{lemma}
\begin{proof}
The bounds are obtained from the energy-energy dissipation inequality \eqref{EI} in an analogous way to the ones given in Lemma \ref{lem.bounds} with the only difference concerning the replacement of $\NN_\e$ by $\NN$, which however makes no difference in the estimate \eqref{estimconv}.
\end{proof}

Using these estimates the following convergence properties can be derived in an analogous way to the proof of Lemma \ref{lem.conv}.
\begin{lemma}\label{lem.conv.long}
Let the assumptions of Lemma {\rm\ref{conv.weak.L1.long}} hold. Then, up to subsequences that we denote with the same
index,
\begin{eqnarray*}
&&\r_k\ \rightarrow  \ \overline{\rho} \qquad  \textnormal{in} \quad L^{q}((0,T)\times\mathbb{R}^2))\qquad \textnormal{for any} \ \
1\leq
q<\infty\,,\\
&&\r^{p}_{k}\ \rightharpoonup  \ \overline{\rho}^{p} \qquad  \textnormal{in} \quad L^2(0,T;H^1(\mathbb{R}^2))\qquad \mbox{for any } m-\tfrac12 \leq p<\infty,\\
&&\sqrt{\r_k}\ \rightarrow  \ \sqrt{\overline{\rho}} \qquad  \textnormal{in} \quad \ L^2((0,T)\times \mathbb{R}^2)\,.
\end{eqnarray*}
\end{lemma}
These convergence results from Lemma \ref{lem.conv.long} and Lemma \ref{conv.weak.L1.long} are sufficient to obtain the weak convergence of the nonlinearities $\sqrt{\r_k}\nabla h[\r_k]$ and $\r_k\nabla h[\r_k]$ in $L^2((0,T)\times \mathbb{R}^2)$, which allows to deduce the weak lower semicontinuity of the entropy dissipation term and to pass to the limit in the weak formulation of \eqref{KS} in the
same way as in the proof of Lemma \ref{lem.diss}.

\begin{lemma}\label{lem.diss.long} Let $\r_k$ and $\overline{\rho}$ be as in Lemma {\rm\ref{lem.conv.long}}. Then
\begin{eqnarray*}
&&
\sqrt{\r_k}\,\nabla h[\r_k]\rightharpoonup \sqrt{\overline{\rho}}\,\nabla h[\overline{\rho}] \qquad \textnormal{in} \ L^2((0,T)\times
\mathbb{R}^2;\R^{2})\\
&&
\r_k\,\nabla h[\r_k]\rightharpoonup \overline{\rho}\,\nabla h[\overline{\rho}] \qquad \textnormal{in} \ L^2((0,T)\times
\mathbb{R}^2;\R^{2})\,.
\end{eqnarray*}
\end{lemma}

This enables us to close the proof of convergence towards the set of stationary states.

\begin{proof}[Proof of Theorem {\rm\ref{thm.long2}}]
Let us first notice that $\overline{\rho}\in L^\infty((0,T)\times\mathbb{R}^2)$ due to the first convergence in Lemma \eqref{lem.conv.long} and the uniform in time bound on the weak solutions in Theorem \ref{thm.ex}.
Due to the weak lower semicontinuity of the $L^2((0,T)\times\mathbb{R}^2)$-norm and the bound from below of the entropy as done in Proposition \ref{Prop1} implies that $ \D[\r_k]\rightarrow 0$ in $L^1(0,T)$, and as consequence
$$
\|\sqrt{\overline{\rho}}|\nabla h[\overline{\rho}]|\|_{L^2((0,T)\times\mathbb{R}^2)}^2\leq
\liminf_{k\rightarrow \infty}\|\sqrt{\r_k}|\nabla h[\r_k]|\|_{L^2((0,T)\times\mathbb{R}^2)}^2=0\,.
$$
Thus $\overline{\rho}$ solves
\begin{equation}\label{crucial}
\overline{\rho} |\nabla h[\overline{\rho}]|^2 =0 \qquad \textnormal{a.e. in } (0,T)\times \mathbb{R}^2\,.
\end{equation}
Moreover, due to the convergence properties in Lemmas \ref{lem.conv.long} and \ref{lem.diss.long} the limiting density $\overline{\rho}$ is a weak distributional solution to \eqref{KS} with test functions is $L^2(0,T;H^1(\R^2))$. Due to \eqref{crucial}, we get that $\overline{\rho} \nabla h[\overline{\rho}]=0$ a.e. in
$(0,T)\times \mathbb{R}^2$ and thus $\partial_t \overline{\rho}=0$ in $L^2(0,T;H^{-1}(\R^2))$. This yields that $\overline{\rho}(t,x) \equiv \overline{\rho}(x)$ does not depend on time.

Due to the convergence properties in Lemma \ref{lem.conv.long}, the uniform bound on the second moment \eqref{2ndunifb} together with Lemma \ref{lem.A.1} in the Appendix, we can deduce that $\overline{\rho}\in L^1((1+|x|^2)dx)$ and that $\r_k \rightarrow \overline{\rho}$ in $L^\infty(0,T;L^1(\mathbb{R}^2))$. In particular, $\overline{\rho}$ has mass $M$.

Putting together all the properties of $\overline{\rho}$ just proved together with the fact that $\nabla \overline{\rho}^m \in L^2(\R^2)$ due to Lemma \eqref{lem.conv.long}, we infer that $\overline{\rho}$ corresponds to a steady state of equation \eqref{KS} in the sense of Definition \ref{stationarystates}. The uniqueness up to translation of stationary states in Theorem \ref{thm.stst.long} shows that $\overline{\rho}$ is a translation of $\r_M$, and thus $\overline{\rho}\in\mathcal{S}$.
In fact, we have shown that the limit of all convergent sequences $\{\rho_k\}_{k\in\N}$ must be a translation of $\r_M$. This in turn shows that the set of accumulation points of any time diverging sequence belongs to $\mathcal{S}$.

Finally, in order to identify uniquely the limit, we take advantage of the translational invariance. We first remark that the center of mass of the initial data is preserved for all time due to the antisymmetry of $\nabla \mathcal{N}$. Due to Proposition \ref{2ndmbound}, all time diverging sequences have uniformly bounded second moments, thus since $\overline{\rho}$ is an accumulation point of a sequence $\rho_{k}$, by Lemma \ref{lem.A.1} we have
\begin{align*}
\left|\int_{\R^2}x\overline{\rho}(x)dx- x_{c} M\right|&=\left|\int_{\R^2}x(\overline{\rho}(x)-\rho_{k}(t,x))dx- \int_{\R^2}x(\rho_{k}(t,x)-\rho_{0}(x))dx\right|\\
&\leq \int_{\R^2}|x||\overline{\rho}(x)-\rho_{k}(t,x)|dx\leq M_{2}[|\rho_{k}(t)-\overline{\rho}|]^{1/2}\|\rho_{k}(t)-\overline{\rho}\|^{1/2}_{L^{1}(\R^{2})}\\
&\leq C \|\rho_{k}(t)-\overline{\rho}\|^{1/2}_{L^{\infty}(0,T;L^{1}(\R^{2}))}\rightarrow 0.
\end{align*}
Hence all accumulation points of the sequences have the same center of mass as the initial data. Then, all possible limits reduce to
the translation of $\r_M$ to the initial center of mass as desired.
\end{proof}


\section*{Appendix}

\begin{theorem}[Dunford-Pettis Theorem]\label{thm.dun}
Let $(X,\Sigma, \mu)$ be a probability space and ${\mathcal F}$ be a bounded subset of $L^1(\mu)$. Then ${\mathcal F}$ is equi-integrable if and only if
${\mathcal F}$ is a relatively compact subset in $L^1(\mu)$ with the weak topology.
\end{theorem}

\begin{lemma}\label{lem.A.1}
 Let $(f_\e)$ be a sequence of nonnegative functions uniformly bounded in the space $L^\infty(0,T;L^1_{log}(\mathbb{R}^2)\cap L^\infty(\mathbb{R}^2))$ with
 $\|f_\e\|_{L^1(\mathbb{R}^2)}=\|f\|_{L^1(\mathbb{R}^2)}=M$. Assume moreover that $f_\e\rightarrow f$ a.e. in $\mathbb{R}^2\times (0,T)$. Then, $f\in
 L^\infty(0,T;L^1_{log}(\mathbb{R}^2)\cap L^\infty(\mathbb{R}^2))$ and
\[f_\e \rightarrow f \qquad \textnormal{in} \ L^\infty(0,T;L^1(\mathbb{R}^2))\,.\]
The same result holds by replacing the logarithmic moment by the second moment, i.e., by replacing $L^1_{log}(\mathbb{R}^2)$ by $L^1((1+|x|^2)dx)$ everywhere.
\end{lemma}
\begin{proof} A similar argument was used in the proof of Proposition 2.1  in \cite{HJ}. First observe that by the Fatou lemma, for any $m>1$
$$
  \sup_{(0,T)}\int f^m dx
  = \sup_{(0,T)}\int\lim_{\e\rightarrow 0}f^m_\e dx
  \le \liminf_{\e\rightarrow 0}\sup_{(0,T)}\int f^m_\e dx \le C.
$$
Let now $L>1$ and $g_\eps=\min\{f_\eps,L\}$. Then $g_\eps\to g=\min\{f,L\}$ a.e. Let moreover $R>1$, then by the  dominated convergence, it holds for
sufficiently small $\eps>0$ on the ball $B_R(0)$:
$$
  \sup_{(0,T)}\int_{B_R(0)}|g_\eps-g|dx \le \frac{C}{L^{m-1}}.
$$
and we obtain
\begin{align*}
  \sup_{(0,T)}\int_{B_R(0)}|f_\eps-f|dx
  &\le \sup_{(0,T)}\int_{B_R(0)}|f_\eps-g_\eps|dx
  + \sup_{(0,T)}\int_{B_R(0)}|g_\eps-g|dx
  + \sup_{(0,T)}\int_{B_R(0)}|g-f|dx \\
  &\le \sup_{(0,T)}\int_{\{f_\eps\ge L\}\cap B_R(0)}(f_\eps-L)dx + \frac{1}{L^{m-1}}
  + \sup_{(0,T)}\int_{\{f\ge L\}\cap B_R(0)}(f-L)dx \\
  &\le \sup_{(0,T)}\int_{B_R(0)}\frac{f_\eps^m}{L^{m-1}}dx
  + \frac{C}{L^{m-1}}
  + \sup_{(0,T)}\int_{B_R(0)}\frac{f^m }{L^{m-1}}dx
  \le \frac{3C}{L^{m-1}}.
\end{align*}
Using additionally the confinement of mass from the bound on the log-moment, we obtain
\[\sup_{(0,T)}\int_{\{|x|>R\}}|f_\eps-f|dx\leq \int_{\mathbb{R}^2}\frac{\corr{\log}(1+|x|^2)}{\corr{\log}(1+R^2)}|f_\eps-f|dx\leq \frac{C}{\corr{\log}(1+R^2)}\rightarrow 0 \qquad
\textnormal{as} \quad R\rightarrow \infty.\]
Since $L>1$ is arbitrary and $m>1$, this shows that $f_\eps\to f$ strongly in
$L^\infty(0,T;L^1(\mathbb{R}^2))$. The proof in case we replace $L^1_{log}(\mathbb{R}^2)$ by $L^1((1+|x|^2)dx)$ is done analogously.
\end{proof}

For the proof of the following Dubinskii Lemma we refer to \cite{CHJ} or Theorem 12.1 in \cite{L}:
\begin{lemma}\label{lem-Dub}
Let $\Omega\subset \mathbb{R}^2$ be bounded with $\partial \Omega \in C^{0,1}$ and let $\{f_\eps\}$, $0<\eps<1$, satisfy
\[\|\pa_tf_\eps\|_{L^1(0,T;(H^{s}(\Omega))')}+\|f_\eps^p\|_{L^q(0,T;H^1(\Omega))}\leq C,\]
for some $p\geq 1$, $q\geq 1$ and $s\geq 0$. Then $\{f_\eps\}$ is relatively compact in $L^{pl}(0,T;L^r(\Omega))$ for any $r<\infty$ and $l<q$.
\end{lemma}


\section*{Acknowledgments}
\small{JAC was partially supported by the Royal Society by a Wolfson Research Merit Award and the EPSRC grants EP/K008404/1 and  EP/P031587/1. SH acknowledges support by the Austrian Science Fund via the Hertha-Firnberg project T-764, and the previous funding by the Austrian Academy of Sciences \"OAW via the New Frontiers project NST-0001. BV is partially supported by the INDAM-GNAMPA project 2015 ``\textit{Propriet\`a qualitative di soluzioni di equazioni ellittiche e paraboliche}'' and by This work has been partially supported by GNAMPA of
the Italian INdAM (National Institute of High Mathematics) and by ``Programma triennale della Ricerca
dell'Universit\`{a} degli Studi di Napoli ``Parthenope'' - Sostegno alla ricerca individuale 2015-2017'' (ITALY). YY was partially supported by the NSF grant DMS-1565480 and DMS-1715418. YY wants to thank Almut Burchard for a helpful discussion.}



\begin{thebibliography}{99}

\bibitem{AlvTrombLion}
{\sc A.~Alvino, G.~Trombetti, P. L.~Lions.} {\em On optimization problems with prescribed rearrangements.}  Nonlinear Anal. {\bf 13} (1989), no. 2, 185-220.

\bibitem{AGS08}
{\sc L.~Ambrosio, N.~Gigli, and G.~Savar{\'e}.} {\em Gradient flows in metric
  spaces and in the space of probability measures}, Lectures in Mathematics ETH
  Z\"urich, Birkh\"auser Verlag, Basel, 2005.

\bibitem{AmbrosioGigli}
L.~Ambrosio and N.~Gigli.
\newblock A user's guide to optimal transport.
\newblock In {\em Modelling and optimisation of flows on networks}, volume 2062
  of {\em Lecture Notes in Math.}, pages 1--155. Springer, Heidelberg, 2013.

\bibitem{BCLR2}
{\sc D.~Balagu\'e, J.~A. Carrillo, T.~Laurent, and G.~Raoul.}
{\em Dimensionality of local minimizers of the interaction energy.}
Arch. Rat. Mech. Anal. {\bf 209} (2013), 1055--1088.

\bibitem{Bedrossian}
{\sc J.~Bedrossian.} {\em Global minimizers for free energies of subcritical aggregation equations with degenerate diffusion.}  Appl. Math. Lett. {\bf 24} (2011), no. 11, 1927--1932.

\bibitem{Bennett-Sharpley}
{\sc C.~Bennett and R.~Sharpley.} {\em Interpolation of operators.}, vol.~129 of
  Pure and Applied Mathematics, Academic Press Inc., Boston, MA, 1988.

\bibitem{bigring}
{\sc A.~L. Bertozzi, James~H. von Brecht, H.~Sun, Theodore Kolokolnikov, and D.~Uminsky.}
{\em Ring patterns and their bifurcations in a nonlocal model of biological swarms.}
Comm. Math. Sci. {\bf 13} (2015), 955--985.

\bibitem{BL} {\sc S.~Bian, J.~G.~Liu.} {\em Dynamic and steady states for multi-dimensional Keller-Segel model with diffusion exponent $m>0$.} Comm. Math.  Phys. {\bf 323} (3) (2013), 1017--1070.

\bibitem{BCC}
 {\sc A.~Blanchet, V.~Calvez, and J.~A. Carrillo.} {\em Convergence of the
  mass-transport steepest descent scheme for the subcritical {P}atlak-{K}eller-{S}egel model.} SIAM J. Numer. Anal. {\bf 46} (2008), 691--721.

\bibitem{BCC2}
{\sc A. Blanchet, E. A. Carlen, J. A. Carrillo.} {\em Functional inequalities, thick tails and asymptotics for the critical mass Patlak-Keller-Segel model.} J. Funct. Anal.  {\bf 262} (2012), 2142--2230.

\bibitem{BCL} {\sc A. Blanchet, J.A. Carrillo, P. Lauren\c{c}ot.} {\em Critical mass for a Patlak-Keller-Segel model with degenerate diffusion in higher dimensions.} Calc. Var. Partial Differential Equations {\bf35} (2009), 133--168.

\bibitem{BCM08}
{\sc A. Blanchet, J. A. Carrillo, N. Masmoudi.} {\em Infinite time aggregation for the critical {P}atlak-{K}eller-{S}egel model in {$\mathbb R^2$}.} Comm. Pure Appl. Math. {\bf 61} (2008), 1449--1481.

\bibitem{BDP}{\sc A.~Blanchet, J.~Dolbeaut, B.~Perthame.} {\em Two-dimensional Keller-Segel model: optimal critical mass and qualitative properties of the solutions.} Electronic Journal of Differential Equations 2006, No. 44, 1--32.

\bibitem{BV}
{\sc M. Bodnar, J. J. L. Vel\'azquez.} {\em Friction dominated dynamics of interacting particles locally close to a crystallographic lattice.}, Math. Methods Appl. Sci. {\bf 36} (2013), 1206--1228.

\bibitem{BCM1} {\sc S. Boi, V. Capasso, and D. Morale.} {\em Modeling the aggregative behavior of ants of the species polyergus rufescens.} Nonlinear Anal. Real World Appl., {\bf1}(2000): 163--176.

\bibitem{BLL}{\sc H. Brascamp, E. H. Lieb, J. M. Luttinger.} {\em A general rearrangement
inequality for multiple integrals.} J. Funct. Anal. {\bf17} (1974), 227--237.

\bibitem{BrezisFunc}{\sc H. Brezis.} \emph{Functional Analysis, Sobolev Spaces and Partial Differential Equations.}
    Springer New York 2010.

\bibitem{Brock} {\sc F. Brock.} {\em Continuous Steiner symmetrization}. Math. Nachrichten. {\bf172} (1995), 25--48.

\bibitem{Brock2} {\sc F. Brock.} {\em Continuous rearrangement and symmetry of solutions of elliptic problems}. Proceedings of the Indian Academy of Sciences - Mathematical Sciences, {\bf 110} (2000), 157--204.

\bibitem{BCM} {\sc M. Burger, V. Capasso, and D. Morale.} {\em On an aggregation model with long and short range interactions.} Nonlin. Anal. Real World Appl., {\bf8} (2007), 939--958.

\bibitem{BDF} {\sc M. Burger, M. DiFrancesco, M. Franek.} {\em Stationary states of quadratic diffusion equations with long-range attraction.} Commun. Math. Sci. {\bf11} (2013), 709--738.

\bibitem{CD}
{\sc J. F. Campos, J. Dolbeault.} {\em Asymptotic estimates for the parabolic-elliptic Keller-Segel model in the plane.} Comm. Partial Differential Equations {\bf39} (2014), 806--841.

\bibitem{CC} {\sc V.~Calvez, J.~A.~Carrillo.} {\em Volume effects in the Keller-Segel model: energy estimates preventing blow-up.} J. Math. Pures Appl. (9) {\bf86}, no. 2 (2006), 155--175.

\bibitem{CCH1} {\sc V.~Calvez, J.~A.~Carrillo, F. Hoffmann.} {\em Equilibria of homogeneous functionals in the fair-competition regime.} Nonlinear Analysis TMA {\bf 159} (2017), 85--128.

\bibitem{CCH2} {\sc V.~Calvez, J.~A.~Carrillo, F. Hoffmann.} {\em The geometry of diffusing and self-attracting particles in a one-dimensional fair-competition regime.} Lecture Notes in Mathematics 2186, CIME Foundation Subseries, Springer, 2018.

\bibitem{CC2} {\sc V.~Calvez, L.~Corrias.} {\em The parabolic-parabolic Keller-Segel model in $\mathbb{R}^2$.} Comm. Math. Sci. {\bf6} (2) (2008), 417--447.

\bibitem{CarlenLoss}
 {\sc E.~Carlen and M.~Loss.} {\em Competing symmetries, the logarithmic {HLS}
  inequality and {O}nofri's inequality on {$S^n$}.} Geom. Funct. Anal., {\bf2}
  (1992), 90--104.

\bibitem{CCV} {\sc J.~A.~Carrillo, D.~Castorina, B.~Volzone.} {\em Ground States for Diffusion Dominated Free Energies with Logarithmic Interaction.} SIAM J. Math. Anal. {\bf47} (2015), no. 1, 1--25.

\bibitem{CCH}
{\sc J. A. Carrillo, Y.-P. Choi, M. Hauray.}
{\em The derivation of swarming models: mean-field limit and Wasserstein distances. Collective dynamics from bacteria to crowds.} 1--46, CISM Courses and Lectures, 553, Springer, Vienna, 2014.

\bibitem{CHJ} {\sc J.~A.~Carrillo, S.~Hittmeir, A.~J\"ungel.} {\em Cross diffusion and nonlinear diffusion preventing blow up in the Keller-Segel model.} Math. Models Methods Appl. Sci. {\bf22}, no. 12 (2012), 1250041, 35 pp.

\bibitem{CHMV} {\sc J.~A.~Carrillo, F.~Hoffmann, E.~Mainini, B.~Volzone.} {\em Ground States in the Diffusion-Dominated Regime}, to appear in Calc. Var. Partial Differential Equations {\bf 57}, (2018), 57--127.

\bibitem{CLM} {\sc J.~A.~Carrillo, S.~Lisini, E.~Mainini.} {\em Uniqueness for Keller-Segel-type chemotaxis model}  Discrete Contin. Dyn. Syst. {\bf34} no. 4 (2014), 1319--1338.

\bibitem{CMV03}
{\sc  J.~A. Carrillo, R.~J. McCann, and C.~Villani.} {\em Kinetic equilibration
  rates for granular media and related equations: entropy dissipation and mass
  transportation estimates.} Rev. Matem\'atica Iberoamericana {\bf19} (2003), 1--48.

\bibitem{Carrillo-McCann-Villani06}
{\sc J.A. Carrillo, R.J. McCann, and C. Villani.} {\em
Contractions in the $2$-Wasserstein length space and
thermalization of granular media.} Arch. Ration. Mech. Anal.  {\bf179}
(2006), 217--263.

\bibitem{CS2018}{\sc J.A. Carrillo and F. Santambrogio.} {\em $L^\infty$ estimates for the JKO scheme in parabolic-elliptic Keller-Segel systems.} Quart. Appl. Math. {\bf76} (2018), 515--530.

\bibitem{CT00}{\sc J.A. Carrillo and G. Toscani.} {\em Asymptotic $L^1$-decay of solutions of the porous medium equation to self-similarity.} Indiana Univ. Math. J. {\bf 49} (2000), 113--141.

\bibitem{CT}
{\sc D.~Chae and G.~Tarantello.} {\em On planar selfdual electroweak vortices.}
  Ann. Inst. H. Poincar\'e Anal. Non Lin\'eaire {\bf21} (2004), 187--207.

\bibitem{CPJ}{\sc T. Champion, L. D. Pascale, and P. Juutinen.} {\em The $\infty$-Wasserstein distance: local solutions and existence of optimal transport maps.} SIAM J. Math. Anal. {\bf40} (2008), 1--20.

\bibitem{CLW}
{\sc L. Chen, J.-G. Liu, J. Wang.} {\em Multidimensional degenerate {K}eller-{S}egel system with critical diffusion exponent {$2n/(n+2)$}.} SIAM J. Math. Anal. {\bf44} (2012), 1077--1102.

\bibitem{CW}
{\sc L. Chen, J. Wang.} {\em Exact criterion for global existence and blow up to a degenerate {K}eller-{S}egel system.} Doc. Math. {\bf19} (2014), 103--120.

\bibitem{Chong}
{\sc  K.M. Chong},
{\sl Some extensions of a theorem of Hardy, Littlewood and P\'olya and their applications}, Canad. J. Math. 26 (1974), 1321--1340.

\bibitem{Craig} {\sc K. Craig.} {\em Nonconvex gradient flow in the Wasserstein metric and applications to constrained nonlocal interactions}, Proc. London Math. Soc. {\bf114} (2017), 60--102.

\bibitem{CKY} {\sc K. Craig, I. Kim, Y. Yao.} {\em Congested aggregation via Newtonian interaction}, Arch. Ration. Mech. Anal. {\bf227} (2018), 1--67.

\bibitem{DNR}{\sc J. I. Diaz, T. Nagai, J.-M. Rakotoson.} {\em Symmetrization techniques on unbounded domains: application to a chemotaxis system on {${\bf R}^N$}.} J. Differential Equations {\bf145} (1998), 156--183.

\bibitem{DolEstLos}{\sc J. Dolbeault, M. Esteban, M. Loss.} {\em Rigidity versus symmetry breaking via nonlinear flows on cylinders and Euclidean spaces.}
 Invent. Math. {\bf 206} (2016), pp. 397--440.

\bibitem{DolEstLos2}{\sc J. Dolbeault, M. Esteban, M. Loss.} {\em Symmetry and symmetry breaking: rigidity and flows in elliptic PDEs.}
  	arXiv:1711.11291 (2017).

\bibitem{Folland}{\sc G. B. Folland.} \emph{Introduction to partial differential equations.}
    Second edition. Princeton University Press, Princeton, NJ, 1995.

\bibitem{Fraenkel}{\sc L. E. Fraenkel.} \emph{An introduction to maximum principles and symmetry in elliptic problems.}
    Cambridge Tracts in Mathematics, 128. Cambridge University Press, Cambridge, 2000.

\bibitem{Gagli}{\sc E.~Gagliardo.} \emph{Ulteriori propriet\`a di alcune classi di funzioni in
pi\'u variabili.} Ricerche Mat. {\bf8} (1959), 24--51.

\bibitem{Gilbarg}
{\sc D.~Gilbarg and N.~S. Trudinger.} {\em Elliptic partial differential
  equations of second order.} Classics in Mathematics, Springer-Verlag, Berlin,
  2001.
\newblock Reprint of the 1998 edition.

\bibitem{Hardy}
{\sc G.~H. Hardy, J.~E. Littlewood, and G.~P{\'o}lya.}
{\em Some simple inequalities satisfied by convex functions},
Messenger Math.  {\bf58} (1929), 145--152. {\em ``Inequalities''},
  Cambridge University Press, 1952,  2d ed.

\bibitem{HJ}{\sc S.~Hittmeir, A.~J\"ungel.} \emph{Cross diffusion preventing blow-up in the two-dimensional Keller-Segel model.} SIAM J. Math. Anal. {\bf43} (2) (2011), 997--1022.

\bibitem{horst}
{\sc D.~Horstmann}. {\em From 1970 until
  present: the {K}eller-{S}egel model in chemotaxis and its consequences. {I}},
  Jahresber. Deutsch. Math.-Verein. {\bf105} (2003), 103--165.

\bibitem{JL}
{\sc W.~J\"{a}ger and S.~Luckhaus}. {\em On explosions of
solutions to a system of partial differential equations modelling
chemotaxis}, Trans. Amer. Math. Soc. {\bf329} (1992), 819--824.

\bibitem{JKO}R. Jordan, D. Kinderlehrer, and F. Otto. {\em The variational formulation of the Fokker-Planck equation.} SIAM J. Math. Anal. {\bf29} (1998), 1--17.

\bibitem{Ka1} {\sc B. Kawohl.} {\em Rearrangements and convexity of level sets in PDE.}
Lecture Notes in Mathematics, 1150. Springer-Verlag, Berlin, 1985.

\bibitem{Ka11} {\sc B. Kawohl.} {\em Continuous symmetrization and related problems.} Differential equations (Xanthi, 1987), 353--360, Lecture Notes in Pure and Appl. Math. 118, Dekker, New York, 1989.

\bibitem{Ka2} {\sc B. Kawohl}. {\em Symmetrization -- or how to prove symmetry of solutions to a PDE.} Partial differential equations (Praha, 1998), 214--229, Chapman \& Hall/CRC Res. Notes Math., 406, Chapman \& Hall/CRC, Boca Raton, FL, 2000.

\bibitem{Keller-Segel-70}
{\sc E.F. Keller, L.A. Segel}. {\em Initiation of slide mold
aggregation viewed as an instability.}, J. Theor. Biol. {\bf26} (1970), 399--415.

\bibitem{Kesavan}{\sc S.~Kesavan.} {\em Symmetrization \& applications.} Series in Analysis, 3. World Scientific Publishing Co. Pte. Ltd., Hackensack, NJ, 2006.

\bibitem{KY}{\sc I.~Kim, Y.~Yao.} {\em The Patlak-Keller-Segel model and its variations: properties of solutions via maximum principle.} SIAM J. Math. Anal. {\bf44} (2012), 568--602.

\bibitem{KSUB}
{\sc T.~Kolokonikov, H.~Sun, D.~Uminsky, and A.~Bertozzi.}
{\em Stability of ring patterns arising from 2d particle interactions.}
Physical Review E {\bf84} (2011), 015203.

\bibitem{K}{\sc R.~Kowalczyk.} {\em Preventing blow-up in a chemotaxis model.} J. Math. Anal. Appl. {\bf305} (2) (2005), 566--588.

\bibitem{LL} {\sc E.~H.~Lieb, M.~Loss.} \emph{Analysis.} Graduate Studies in Mathematics, 14. American Mathematical Society, Providence, RI, 1997.

\bibitem{LY}{\sc E.~H.~Lieb, H.~T.~Yau.}{\em The Chandrasekhar theory of stellar collapse as the limit of quantum mechanics.} Comm. Math. Phys. {\bf112} (1) (1987),
    147--174.

\bibitem{L}{\sc J.-L.~Lions.} {\em Quelques m\'{e}thodes de r\'{e}solution des probl\`{e}mes aux limites non lin\'{e}aires.} (French) Dunod; Gauthier-Villars, Paris, 1969.

\bibitem{Lions}
{\sc P.-L. Lions.} {\em The concentration-compactness principle in the calculus
  of variations. {T}he locally compact case. {I}} Ann. Inst. H. Poincar\'e
  Anal. Non Lin\'eaire, {\bf1} (1984), pp.~109--145.

\bibitem{LS} {\sc S. Luckhaus, Y.~Sugiyama.} {\em Asymptotic profile with the optimal convergence rate for a parabolic equation of chemotaxis in super-critical cases.}
Indiana Univ. Math. J. {\bf56} (2007), 1279--1297.

\bibitem{ML1}
{\sc A. Mogilner, L. Edelstein-Keshet}. {\em A non-local model for a swarm.} J. Math. Biol. {\bf38} (1999), 534--570.

\bibitem{ML2}
{\sc A. Mogilner, L. Edelstein-Keshet, L. Bent, A. Spiros}. {\em Mutual interactions, potentials, and individual distance in a social aggregation.} J. Math. Biol. {\bf47} (2003), 353--389.

\bibitem{MCO}
{\sc D. Morale, V. Capasso, K. Oelschl\"Ager}. {\em An interacting particle system modelling aggregation behavior: from individuals to populations.} J. Math. Biol. {\bf50} (2005), 49--66.

\bibitem{Morgan} {\sc F. Morgan.} {\em A round ball uniquely minimizes gravitational potential energy.} Proc. Amer. Math. Soc. {\bf133} (2005), 2733--2735.

\bibitem{Mossino}
{\sc J.~Mossino and J.-M. Rakotoson. }{\em Isoperimetric inequalities in
  parabolic equations.} Ann. Scuola Norm. Sup. Pisa Cl. Sci. (4), {\bf13} (1986),
  51--73.

\bibitem{Nir}{\sc L.~Nirenberg.} {\em On elliptic partial differential equations.} Ann. Scuola Norm. Sup. Pisa (3) {\bf13} (1959), 115--162.

\bibitem{O}
{\sc K. Oelschl\"ager.} {\em Large systems of interacting particles and the porous medium equation.} J. Differential Equations {\bf88} (1990), 294--346.

\bibitem{Otto} F. Otto. {\em The  geometry  of  dissipative  evolution  equations: the  porous  medium  equation.} Comm.  Part. Diff. Eq. {\bf26} (2001),   101--174.

\bibitem{patlak}
{\sc C.~S. Patlak}. {\em Random walk with persistence and external
bias.} Bull. Math. Biophys. {\bf15} (1953), 311--338.

\bibitem{SC}
{\sc C. Sire, P.-H. Chavanis.} {\em Critical dynamics of self-gravitating Langevin particles and bacterial populations.} Phys. Rev. E {\bf78} (2008), 061111.

\bibitem{Strohmer}
{\sc G.~Str{\"o}hmer.} {\em Stationary states and moving planes.} In Parabolic and {N}avier-{S}tokes equations. {P}art 2, Banach Center Publ. {\bf81} (2008), 501--513.

\bibitem{S} {\sc Y.~Sugiyama.} {\em The global existence and asymptotic behavior of solutions to degenerate quasi-linear parabolic systems of chemotaxis.}
    Differential Integral Equations, {\bf20} (2007), 133--180.

\bibitem{Talenti1}{\sc G. Talenti.} {\em Elliptic equations and rearrangements.} Ann. Scuola Norm. Sup. (4) {\bf3} (1976), 697--718.

\bibitem{Talenti2}{\sc G. Talenti. }{\em Inequalities in rearrangement invariant function spaces.}
Nonlinear analysis, function spaces and applications, Vol. 5 (Prague, 1994), 177-230, Prometheus, Prague, 1994.

\bibitem{Talenti3}{\sc G. Talenti.} {\em  Linear elliptic
  p.d.e.'s: level sets, rearrangements and a priori estimates of solutions.}
  Boll. Un. Mat. Ital. B (6), {\bf4} (1985), 917--949.

\bibitem{TBC}
{\sc C. M. Topaz, A. L. Bertozzi, M. A. Lewis}. {\em A nonlocal continuum model for biological aggregation.} Bull. Math. Biol.  {\bf68} (2006), 1601--1623.

\bibitem{Vsym82} {\sc J. L. V\'azquez}.
{\em Sym\'etrisation pour $u_t=\Delta\varphi(u)$ et applications,}
C. R. Acad. Sc. Paris {\bf295} (1982), 71--74.

\bibitem{VANS05} {\sc J.~L. V{\'a}zquez}. {\em Symmetrization and Mass Comparison for
Degenerate Nonlinear Parabolic and  related Elliptic Equations},
Advances in Nonlinear Studies, {\bf5} (2005), 87--131.

\bibitem{Vazquezbook}
{\sc J. L. V\'azquez.} {\em The porous medium equation. Mathematical theory.} Oxford Mathematical Monographs. The Clarendon Press, Oxford University Press, Oxford, 2007.

\bibitem{VazVol1}
{\sc J.~L. V{\'a}zquez and B.~Volzone.} {\em Symmetrization for linear and
  nonlinear fractional parabolic equations of porous medium type.} J. Math.
  Pures Appl. (9), {\bf101} (2014), 553--582.

\bibitem{VazVol2}
\leavevmode\vrule height 2pt depth -1.6pt width 23pt, {\em Optimal estimates
  for fractional fast diffusion equations.} J. Math. Pures Appl. (9), {\bf103}
  (2015), pp.~535--556.

\end{thebibliography}
\end{document}